\documentclass[hidelinks,onefignum,onetabnum]{siamart220329}
\usepackage{mathtools}
\usepackage{algorithm}
\usepackage{algorithmic}
\usepackage{subfigure}
\usepackage{enumitem}
\usepackage{amsmath}

\usepackage{microtype}

\makeatletter
\def\cref@override@label@type#1\@nil#2{#1}
\makeatother

\newlist{questions}{enumerate}{1}
\setlist[questions]{label=(Q\arabic*), leftmargin=*, align=left}

\usepackage{lipsum}
\usepackage{amsfonts}
\usepackage{graphicx}
\usepackage{epstopdf}
\usepackage{algorithmic}
\ifpdf
  \DeclareGraphicsExtensions{.eps,.pdf,.png,.jpg}
\else
  \DeclareGraphicsExtensions{.eps}
\fi

\newsiamremark{remark}{Remark}
\newsiamremark{hypothesis}{Hypothesis}
\crefname{hypothesis}{Hypothesis}{Hypotheses}
\newsiamthm{claim}{Claim}

\headers{Parallelizable Splitting Scheme for NS-Biot System}{}

\title{Error Analysis of the Explicit Splitting Scheme for Fluid-Poroelastic Structure Interaction Problems \thanks{Submitted to the editors DATE.
\funding{\v{C}ani\'{c}'s research has been supported in part by the
National Science Foundation under grants DMS-2408928, DMS-2247000 and by the  U.S. Department of Energy, Office of Science, Office of Advanced Scientific Computing Research's Applied Mathematics Competitive Portfolios program under Contract No. AC02-05CH11231. Yifan Wang’s research has been supported in part by the National Science Foundation under grant DMS-2247001 and by a Simons Foundation Travel Award.}}}

\author{
Yifan Wang \thanks{Department of Mathematics and Statistics, Texas Tech University, Lubbock, TX, USA (\email{yifan.wang@ttu.edu}).}
\and 
Jeonghun Lee \thanks{Department of Mathematics, Baylor University, Waco, TX, USA (\email{Jeonghun\_Lee@baylor.edu}).}
\and 
Sun\v{C}ica \v{C}ani\'{c} \thanks{\textbf{Corresponding author.} Department of Mathematics, University of California, Berkeley, Berkeley, CA, USA (\email{canics@berkeley.edu}).}
}

\usepackage{amsopn}

\allowdisplaybreaks[4]
\ifpdf
\hypersetup{
  pdftitle={Error Analysis of the Explicit Splitting Scheme for Fluid-Poroelastic Structure Interaction Problems},
  pdfauthor={Yifan Wang, Jeonghun Li, Sun\v{C}ica \v{C}ani\'{c}}
}
\fi

\newcommand{\sunny}{\color{black}}

\begin{document}

\maketitle
\tableofcontents

\begin{abstract}
We present \emph{a priori} error analysis for a fully discrete, parallelizable, explicit loosely coupled scheme for the time-dependent Stokes-Biot problem. The method decouples the fluid and poroelastic subproblems in a fully explicit fashion, allowing each problem to be solved independently at each time step, {\sunny{with a consistent treatment of the interface conditions that provides stability and convergence of the scheme.}} 
The {\sunny{error}}  analysis is carried out in a discrete energy framework. {\sunny{More specifically, we introduce}} Ritz-type projections in each subdomain, {\sunny{and subtract}} the fully discrete scheme from the time-discrete continuous formulation. {\sunny{This yields}} reduced error equations in which the dominant interpolation contributions cancel. The remaining consistency terms stem primarily from time discretization residuals and lagged interface data inherent to the explicit splitting. 
{\sunny{The main result of this manuscript is the derivation of a discrete error energy identity, and establishment of unconditional error estimates in a combined energy-dissipation norm via a Gronwall--type argument. These estimates demonstrate first-order accuracy in time and optimal spatial convergence rates, as determined by the degree of the finite element polynomials.}}
Numerical experiments based on a manufactured solution corroborate the theory, confirming first-order temporal convergence for all variables, and spatial convergence orders consistent with the chosen approximation spaces.
\end{abstract}
\begin{keywords}
Fluid-poroelastic structure interaction, Stokes-Biot problem, A priori error estimates, Explicit splitting scheme.
\end{keywords}

\begin{MSCcodes}
65M22,	65M60, 74F10, 76D05, 76S05
\end{MSCcodes}

\section{Introduction}

Coupled {\sunny{fluid-poroelastic structure interaction problems occur}} across a broad spectrum of applications, including bioartificial organ design, perfusion through soft biological tissues, and fluid injection in deformable porous materials \cite{Buka202404,fluids7070222,Benjamin2014,Martina2014,CanicSiam2021}.
From a mathematical standpoint, these processes are governed by a coupling between the incompressible {\sunny{time-dependent}} Stokes equations, which describe the free-fluid flow, and the Biot poroelasticity equations, which capture the mechanics of the saturated solid matrix \cite{biot1941general,biot1955theory}.
This coupled system, commonly referred to as the \emph{Stokes–Biot model}, has been the subject of extensive theoretical and numerical investigation, forming the foundation for numerous partitioned and monolithic computational formulations \cite{ambartsumyan2018lagrange,wen2020strongly,guo2022decoupled,cesmelioglu2017analysis,CanicBook,BociuMulti, SingularLimit, guo2025uncond, Andrew2026,Yotov2023,Mortar,NitCoup,MultiLayer,Martina2014,RobRob,MACscheme}.

In the coupled setting, the primary difficulty lies in the enforcement of
interface conditions expressing continuity of normal flux, and balance
of normal and tangential stresses across the interface separating the
fluid and poroelastic domains.
Monolithic methods, which solve both subproblems simultaneously,
are stable and accurate but computationally expensive.
Partitioned or splitting schemes, on the other hand, provide modularity
and enable independent solution of the subproblems,
but they require appropriate stabilization to guarantee convergence.
Nitsche-type coupling techniques have proven highly effective for this
purpose because they impose the interface conditions weakly and
symmetrically, allowing for nonmatching meshes and full parallelization
\cite{BADIA20097986,BUKAC2015138,MartinaOyekole,cesmelioglu2016optimization}.

{\sunny{Recently, using Nitche-type interface approaches, we developed a first fully discrete parallel splitting scheme for the transient Stokes-Biot problem in the context of both linear and nonlinearly--coupled problems \cite{24M1695713}.}} 
%
{\sunny{In this scheme,}} the fluid and poroelastic subproblems are solved
independently at each time step, enabling complete parallelization across the interface.
The coupling is realized through Nitsche-type interface terms weighted
by penalty parameters $\gamma>0$ (for tangential velocity continuity)
and $L>0$ (for normal velocity and pressure continuity). Moreover, the pore pressure is employed via the coupling conditions to replace the stress terms, thereby enhancing the stability of the system.
{\sunny{While this scheme has been shown in \cite{24M1695713} to be unconditionally stable, there has been no error analysis. 

In the present manuscript, we present a comprehensive error analysis. More specifically,}} 
we establish a rigorous \emph{a~priori error estimate} that demonstrates
first-order accuracy in time and optimal $O(h^{k+1})$ spatial accuracy
for finite elements of degree~$k$.

Our analysis employs a discrete energy argument.
Ritz projections are introduced in each subdomain to split the total
error into interpolation and discrete components.
By subtracting the continuous and discrete formulations,
we derive error equations in which the projection errors cancel.
An energy identity is then established, balancing subdomain and
interface contributions.
Through a series of trace and time-difference inequalities and the
discrete Gronwall lemma, we derive the convergence bound
\[
\max_{0\le n\le N} X_n
\;+\;
\Big(\sum_{n=1}^N Y_n^2\Big)^{1/2}
\;\le\;
C \big( h^{k+\frac{r}{2}} + \Delta t \big),
\]
where $X_n$ and $Y_n$ denote the discrete energy and dissipation norms
of the total error at time~$t_n$.

The remainder of the paper is organized as follows.
Section~\ref{sec:continuous} introduces the continuous
Stokes--Biot model and its weak formulation.
Section~\ref{sec:discrete} presents the fully discrete
parallel splitting scheme and the associated finite element spaces.
Section~\ref{sec:projections} defines the projection operators and
states their approximation properties.
Section~\ref{sec:error} derives the discrete error equations and the
main energy identity.
Section~\ref{sec:estimate} provides the detailed
a~priori error estimates and proves the main theorem.
Finally, Section~\ref{sec:numerics} then offers a numerical study demonstrating agreement with the theoretical error bounds.

\section{Continuous Problem}
\label{sec:continuous}

Let $\Omega_f\subset\mathbb{R}^d$ ($d=2,3$) denote the fluid domain
and $\Omega_p\subset\mathbb{R}^d$ the poroelastic (or Biot) domain.
The interface between the two subdomains is denoted by
$\Gamma=\partial\Omega_f\cap\partial\Omega_p$.
We write $\boldsymbol{n}_f$ for the unit outward normal on $\partial\Omega_f$
(so that $\boldsymbol{n}_p=-\boldsymbol{n}_f$ on $\Gamma$), and define the tangential projection
operators
\[
\boldsymbol{P}_f \boldsymbol{v} = \boldsymbol{v} - (\boldsymbol{v}\!\cdot\!\boldsymbol{n}_f)\boldsymbol{n}_f,
\qquad
\boldsymbol{P}_p \boldsymbol{v} = \boldsymbol{v} - (\boldsymbol{v}\!\cdot\!\boldsymbol{n}_p)\boldsymbol{n}_p .
\]
The outer boundaries of $\Omega_f$ and $\Omega_p$ are denoted
by $\Sigma_{D,f}$ and $\Sigma_{D,p}$, respectively, where appropriate
Dirichlet or Neumann data are prescribed.

\subsection{Governing equations}

The coupled Stokes--Biot system combines the time-dependent incompressible Stokes
equations in $\Omega_f$ with the quasi-static Biot equations
in $\Omega_p$:
\begin{subequations}\label{eq:stokes-biot}
\begin{align}
\rho_f \frac{\partial \boldsymbol{u}}{\partial t}
- \nabla\!\cdot \boldsymbol{\sigma}_f
  &= \boldsymbol{f}_f
\quad &&\text{in }\Omega_f\times(0,T),\\[0.3em]
\nabla\!\cdot \boldsymbol{u} &= 0
\quad &&\text{in }\Omega_f\times(0,T),\\[0.3em]
\rho_p \frac{\partial^2\boldsymbol{\eta}}{\partial t^2}
- \nabla\!\cdot \boldsymbol{\sigma}_s
  &= \boldsymbol{f}_p
\quad &&\text{in }\Omega_p\times(0,T),\\[0.3em]
C_0\frac{\partial\phi}{\partial t}
+ \alpha\nabla\!\cdot\!\boldsymbol{\xi}
- \nabla\!\cdot (K\nabla\phi)
  &= g_p
\quad &&\text{in }\Omega_p\times(0,T),\\
\boldsymbol{u}_p
  &= -K\nabla\phi
\quad &&\text{in }\Omega_p\times(0,T).
\end{align}
\end{subequations}
Here
$\boldsymbol{u}$ and $\boldsymbol{P}_f$ are the fluid velocity and pressure,
$\boldsymbol{\eta}$ and
$\boldsymbol{\xi}=\partial_t\boldsymbol{\eta}$
are the structure displacement and velocity respectively.
$\phi$ is the pore pressure and $\boldsymbol{u}_p$ is filtration velocity. The Cauchy stress tensors are defined as:
\begin{equation}\label{eq:stresses}
\boldsymbol{\sigma}_f
  = 2\mu_f \boldsymbol{D}(\boldsymbol{u}) - \boldsymbol{P}_f\boldsymbol{I},
\qquad
\boldsymbol{\sigma}_p
  =2\mu_p \boldsymbol{D}(\boldsymbol{\eta})
    + \lambda_p (\nabla\!\cdot\!\boldsymbol{\eta})\boldsymbol{I}
    - \alpha \phi \boldsymbol{I},
\end{equation}
where
$\boldsymbol{D}(\boldsymbol{v})
  =(\nabla \boldsymbol{v}+\nabla \boldsymbol{v}^{\!\top})/2$
is the strain-rate tensor.
The parameters $\rho_f,\mu_f$ denote the fluid density and viscosity,
$\rho_p,\mu_p,\lambda_p$ the solid density and Lamé constants,
$\alpha$ the Biot--Willis coefficient,
$K$ the hydraulic conductivity tensor,
and $C_0\ge0$ the storage coefficient
\cite{biot1941general,biot1955theory,temam2001navier,richter2017fluid}.

\subsection{Interface conditions}
To couple the fluid model and the Biot model, we impose the following set of coupling conditions on interface $\Gamma$:
\begin{subequations}
\begin{align}
    &(\boldsymbol{\xi} + \boldsymbol{u}_p)\cdot\boldsymbol{n}_f = \boldsymbol{u}\cdot\boldsymbol{n}_f \quad&&\text{on}\ \Gamma\times(0, T), \label{interface1}\\
    &\boldsymbol{\tau}_{f, j}\cdot\boldsymbol{\sigma}_f\boldsymbol{n}_f = -\gamma(\boldsymbol{u} - \boldsymbol{\xi})\cdot\boldsymbol{\tau}_{f, j}\quad \text{for } j = 1,\cdots, d-1 \quad&&\text{on}\ \Gamma\times(0, T),\label{interface2}\\
    &\boldsymbol{n}_f\cdot\boldsymbol{\sigma}_f\boldsymbol{n}_f = -\phi\quad&&\text{on}\ \Gamma\times(0, T), \label{interface3}\\
    &\boldsymbol{\sigma}_f\boldsymbol{n}_f =  \boldsymbol{\sigma}_p\boldsymbol{n}_f\quad&&\text{on}\ \Gamma\times(0, T)\label{interface4},
\end{align}
\label{interface condition}
\end{subequations}
where $\boldsymbol{\tau}_{f, j}$ denotes an orthonormal sets of unit vectors on the tangential plane to $\Gamma$. $\gamma > 0$ denotes the slip rate in the Beavers-Joseph-Saffman (BJS) interface condition \cite{BADIA20097986,BUKAC2015138,ambartsumyan2018lagrange, sun2021domain,chen2011parallel}.

\subsection{Weak formulation}

We introduce the function spaces:
\[
\begin{aligned}
\boldsymbol{V}_f &= \{\,\boldsymbol{v}\in H^1(\Omega_f)^d :
   \boldsymbol{v}=0 \text{ on } \Sigma_{D,f}\,\}, \qquad
& Q_f &= L_0^2(\Omega_f),\\
\boldsymbol{V}_p &= \{\,\boldsymbol{z}\in H^1(\Omega_p)^d :
   \boldsymbol{z}=0 \text{ on } \Sigma_{D,p}\,\}, \qquad
& Q_p &= H^1(\Omega_p).
\end{aligned}
\]
Denoting the $L^2$ inner products on $\Omega_f$ and $\Omega_p$
by $(\cdot,\cdot)_{f}$ and $(\cdot,\cdot)_{p}$,
and the duality pairing on $\Gamma$ by
$\langle\cdot,\cdot\rangle_\Gamma$,
the weak formulation reads:

Find
\[
(\boldsymbol{u}, ~\boldsymbol{P}_f)\in \boldsymbol{V}_f\times Q_f, \qquad
(\boldsymbol{\eta}, ~\boldsymbol{\xi}, ~\phi)
  \in \boldsymbol{V}_p\times \boldsymbol{V}_p\times Q_p,
\]
such that, for almost every $t\in(0,T)$,
\begin{align}
&\rho_f\!
\left(\frac{\partial\boldsymbol{u}}{\partial t}, \boldsymbol{v}\right)_f
+ 2\mu_{f}\left(\boldsymbol{D}(\boldsymbol{u}),\boldsymbol{D}(\boldsymbol{v})\right)_f
- \left(p, \nabla\!\cdot\!\boldsymbol{v}\right)_f + \left(\nabla\!\cdot\!\boldsymbol{u}, q\right)_f\label{coupled}\\
&\quad
+ \left\langle\phi,\,\boldsymbol{v}\!\cdot\!\boldsymbol{n}_{f}\right\rangle_\Gamma
+ \gamma\left\langle \boldsymbol{P}_{f}\!\left(\boldsymbol{u}-\boldsymbol{\xi}\right),
                   \boldsymbol{P}_{f}(\boldsymbol{v})\right\rangle_\Gamma \nonumber\\[0.3em]
&\quad
+ \rho_{p}\left(\frac{\partial\boldsymbol{\xi}}{\partial t},\boldsymbol{\zeta}\right)_p
+ 2\mu_{p}\left(\boldsymbol{D}(\boldsymbol{\eta}),\boldsymbol{D}(\boldsymbol{\zeta})\right)_p
+ \lambda_{p}\left(\nabla\!\cdot\!\boldsymbol{\eta},\,\nabla\!\cdot\!\boldsymbol{\zeta}\right)_p
- \alpha\left(\phi,\,\nabla\!\cdot\!\boldsymbol{\zeta}\right)_p\nonumber\\
&\quad
+ C_0\left(\frac{\partial\phi}{\partial t},\psi\right)_p
+ \alpha\left(\nabla\!\cdot\!\boldsymbol{\xi},\,\psi\right)_p
+ (K\nabla\phi,\,\nabla\psi)_p\nonumber\\[0.3em]
&\quad
- \gamma\left\langle \boldsymbol{P}_{p}\!\left(\boldsymbol{u}-\boldsymbol{\xi}\right),
                     \boldsymbol{P}_{p}(\boldsymbol{\zeta})\right\rangle_{\Gamma}
+ \left\langle (\boldsymbol{\xi}-\boldsymbol{u})\!\cdot\!\boldsymbol{n}_{p},\,\psi\right\rangle_{\Gamma}
+ \left\langle \phi,\,\boldsymbol{\zeta}\!\cdot\!\boldsymbol{n}_{p}\right\rangle_{\Gamma}\nonumber\\[0.4em]
&= (\boldsymbol{f}_{f},\boldsymbol{v})_f
+(\boldsymbol{f}_{p},\boldsymbol{\zeta})_p
+ (g_{p},\psi)_p,
\qquad
\forall (\boldsymbol{v}, q, \boldsymbol{\zeta}, \psi)
\in \boldsymbol{V}_f\times Q_{f}\times \boldsymbol{V}_{p}\times Q_{p}.\nonumber
\end{align}

\noindent
We note that in the above weak formulation, the pore pressure variable $\phi$
serves as the primary unknown for the Darcy equation in the poroelastic domain.
The Darcy filtration velocity $\boldsymbol{u}_p = -K\nabla\phi$
can subsequently be recovered through post-processing if desired. 

In the following derivation, we assume all external force terms are zero for the sake of simplicity.

\subsection{The Decoupled Continuous Stokes–Biot System}

{\sunny{We will now decouple the weak formulation \eqref{coupled} into a fluid subproblem and a Biot subproblem. To do this,  we will rewrite the integrals over $\Gamma$ in such a way that the interface conditions \eqref{interface1}-\eqref{interface4} can be expressed as Robin boundary conditions for the subproblems. More specifically, for the fluid subproblem we will utilize the following interface conditions, for each given  penalty parameter $L > 0$:
\begin{equation}
\begin{array}{ll}
\boldsymbol{n}_f \cdot\left(\boldsymbol{\sigma}_f \boldsymbol{n}_f\right)+L \boldsymbol{u} \cdot \boldsymbol{n}_f
= L \boldsymbol{u} \cdot \boldsymbol{n}_f-\phi
& \text{on } \Gamma, \\[2mm]
\boldsymbol{P}_f\!\left(\boldsymbol{\sigma}_f \boldsymbol{n}_f\right)
+\gamma \boldsymbol{P}_f\!\left(\boldsymbol{u}\right)
=\gamma \boldsymbol{P}_f\!\left(\boldsymbol{\xi}\right)
& \text{on } \Gamma, 
\end{array}
\end{equation}
and for the Biot subproblem we will utilize the following interface conditions:
\begin{equation}
\begin{array}{ll}
 {\boldsymbol{n}}_p \cdot \left( {\boldsymbol{\sigma}}_p{\boldsymbol{n}}_p
\right)+ {\phi}
+{\boldsymbol{\xi}}\cdot {\boldsymbol{n}}_p
={\boldsymbol{\xi}}\cdot {\boldsymbol{n}}_p
& \text{on } {\Gamma}, \\[2mm]
{K}{\nabla}{\phi}\cdot {\boldsymbol{n}}_p
+ {\phi}/L
- {\boldsymbol{\xi}}\cdot  {\boldsymbol{n}}_p
= {\phi}/L- {\boldsymbol{u}}\cdot  {\boldsymbol{n}}_p
& \text{on } {\Gamma}, \\[2mm]
  \boldsymbol{P}_p\!\left( {\boldsymbol{\sigma}}_p {\boldsymbol{n}}_p\right)
+\gamma   \boldsymbol{P}_p\!\left( {\boldsymbol{\xi}}\right)
=\gamma   \boldsymbol{P}_p\!\left( {\boldsymbol{u}}\right)
& \text{on } {\Gamma}.
\end{array}
\end{equation}
In the fully discrete scheme, presented below in equations \eqref{eq:fluid-discrete} and \eqref{eq:poro-discrete}, the left hand-sides will be considered at time $t^{n+1}$, while the right hand-sides will be considered at time $t^n$.
}}

%

{\sunny{
The decoupled continuous problem can now be written as follows:
}}
\begin{align}
	\label{eq:continuous-Stokes-eq}
	&\rho_f \left( \partial_t \boldsymbol{u} , \boldsymbol{v} \right)_{f}
	+ 2\mu_f \left( \boldsymbol{D}(\boldsymbol{u}), \boldsymbol{D}(\boldsymbol{v}) \right)_{f}
	- (p, \nabla \cdot \boldsymbol{v} )_{f} 
	\\
	\notag 
	&+ (\nabla \cdot \boldsymbol{u}, q )_{f} + \gamma \langle \boldsymbol{P}_f(\boldsymbol{u}), \boldsymbol{P}_f(\boldsymbol{v}) \rangle_\Gamma
	+ L \langle \boldsymbol{u} \cdot \boldsymbol{n}_f,\; \boldsymbol{v} \cdot \boldsymbol{n}_f \rangle_\Gamma 
	\\
	\notag 
	&= \langle \gamma \boldsymbol{P}_f(\boldsymbol{\xi}), \boldsymbol{P}_f(\boldsymbol{v}) \rangle_{\Gamma} 
    + \langle L \boldsymbol{u} \cdot \boldsymbol{n}_f - \phi, \boldsymbol{v} \cdot \boldsymbol{n}_f \rangle_{\Gamma}, 
\end{align}
and 
\begin{align}
	\label{eq:continuous-poroelasticity-eq}
	&\rho_p \left( \partial_t \boldsymbol{\xi}, \boldsymbol{\zeta} \right)_{p} + 2\mu_p \left( \boldsymbol{D}(\boldsymbol{\eta}), \boldsymbol{D}(\boldsymbol{\zeta}) \right)_{p} + \lambda_p ( \nabla \cdot \boldsymbol{\eta}, \nabla \cdot \boldsymbol{\zeta} )_{p} 
	\\
	\notag 
	&- \alpha ( \phi, \nabla \cdot \boldsymbol{\zeta} )_{p} + C_0 \left( \partial_t \phi, \psi \right)_{p}
+ \alpha ( \nabla \cdot \boldsymbol{\xi}, \psi )_{p} + ( K \nabla \phi, \nabla \psi )_{p} 
	\\
	\notag 
	&+ \gamma \langle \boldsymbol{P}_p(\boldsymbol{\xi}), \boldsymbol{P}_p(\boldsymbol{\zeta}) \rangle_\Gamma + \langle \boldsymbol{\xi} \cdot \boldsymbol{n}_p,\; \boldsymbol{\zeta} \cdot \boldsymbol{n}_p \rangle_\Gamma + \langle \phi, \boldsymbol{\zeta} \cdot \boldsymbol{n}_p \rangle_\Gamma 
	+ \frac{1}{L} \langle \phi, \psi \rangle_\Gamma - \langle \boldsymbol{\xi} \cdot \boldsymbol{n}_p, \psi \rangle_\Gamma 
	\\
	\notag 
	&= \langle \gamma \boldsymbol{P}_p (\boldsymbol{u}), \boldsymbol{P}_p(\boldsymbol{\zeta}) \rangle_{\Gamma}
    + \langle \boldsymbol{\xi} \cdot \boldsymbol{n}_p, \boldsymbol{\zeta} \cdot \boldsymbol{n}_p \rangle_{\Gamma} + \langle -\boldsymbol{u} \cdot \boldsymbol{n}_p + \phi/L, \psi \rangle_{\Gamma}. 
\end{align}

\section{Discretized problem: A parallelizable explicit splitting scheme}
\label{sec:discrete}

Let $t_n = n\Delta t$ for $n=0,1,\ldots,N$, with final time $T=N\Delta t$. For spatial discretization we consider conforming finite element subspaces:
\[
\boldsymbol{V}_{f,h}\subset\boldsymbol{V}_f,\qquad
Q_{f,h}\subset Q_f,\qquad
\boldsymbol{V}_{p,h}\subset\boldsymbol{V}_p,\qquad
Q_{p,h}\subset Q_p,
\]
where $(\boldsymbol{V}_{f,h},Q_{f,h})$ is a Stokes-stable pair.
All discrete spaces consist of piecewise polynomials of degree~$k$
and satisfy the usual approximation properties:
\[
\inf_{\boldsymbol{v}_h\in\boldsymbol{V}_{f,h}}
  \|\boldsymbol{v}-\boldsymbol{v}_h\|_{H^1(\Omega_f)}
  \le C h^k \|\boldsymbol{v}\|_{H^{k+1}(\Omega_f)}, \]
\[
\inf_{q_h\in Q_{f,h}}
  \|q-q_h\|_{L^2(\Omega_f)}
  \le C h^{k+1} \|q\|_{H^{k+1}(\Omega_f)},
\]
and analogously on~$\Omega_p$.

{\sunny{Our {\bf{fully discrete, parallelizable explicit splitting algorithm}} for solving the coupled Stokes-Biot problem \eqref{coupled} is obtained by solving the
fluid and poroelastic subproblems \emph{independently and in parallel},
where the coupling is enforced weakly via Nitsche-type interface terms
with penalty parameters $\gamma>0$ and $L>0$. More specifically, 
given $(\boldsymbol{u}_h^n$, $p_h^n$,
$\boldsymbol{\eta}_h^n$, $\boldsymbol{\xi}_h^n$, $\phi_h^n)$,
find 
$(\boldsymbol{u}_h^{n+1}$, $p_h^{n+1}$,
$\boldsymbol{\eta}_h^{n+1}$, $\boldsymbol{\xi}_h^{n+1}$, $\phi_h^{n+1})$
such that the following two {\emph{independent}} subproblems are solved:
}}
\vskip 0.1in
\begin{enumerate}[leftmargin=12pt,labelsep=0.5em]
\item{\emph{Discrete fluid subproblem}}

Find $(\boldsymbol{u}_h^{n+1},p_h^{n+1})
\in\boldsymbol{V}_{f,h}\times Q_{f,h}$ such that:
\begin{equation}\label{eq:fluid-discrete}
\begin{aligned}
&\rho_f\!\left(
  \frac{\boldsymbol{u}_h^{n+1}-\boldsymbol{u}_h^n}{\Delta t},
  \boldsymbol{v}\right)_f
+ 2\mu_f \left(\boldsymbol{D}(\boldsymbol{u}_h^{n+1}),
\boldsymbol{D}(\boldsymbol{v})\right)_f
- ( p_h^{n+1},
     \nabla\boldsymbol{v})_f
\\
&\quad
+ (\nabla\!\cdot\!\boldsymbol{u}_h^{n+1},q)_f
+ \gamma\langle
   \boldsymbol{P}_f\boldsymbol{u}_h^{n+1},
\boldsymbol{P}_f\boldsymbol{v}\rangle_{\Gamma}
+ L\langle
\boldsymbol{u}_h^{n+1}\!\cdot\!\boldsymbol{n}_f,
\boldsymbol{v}\!\cdot\!\boldsymbol{n}_f\rangle_{\Gamma}
\\
&= (\boldsymbol{F}_f^{n+1},\boldsymbol{v})_f
+ \langle
\gamma\,\boldsymbol{P}_f\boldsymbol{\xi}_h^n,
\boldsymbol{P}_f\boldsymbol{v}\rangle_{\Gamma}
+ \langle
L\,\boldsymbol{u}_h^n\!\cdot\!\boldsymbol{n}_f
   - \phi_h^n,
\boldsymbol{v}\!\cdot\!\boldsymbol{n}_f\rangle_{\Gamma}
+ \langle
\boldsymbol{f}_f^{n+1},\boldsymbol{v}\rangle_{f},
\end{aligned}
\end{equation}
for all $(\boldsymbol{v}_h,q_h)\in\boldsymbol{V}_{f,h}\times Q_{f,h}$.
\vskip  0.1in
\item{\emph{Discrete poroelastic structure subproblem}}
Find $(\boldsymbol{\eta}_h^{n+1},
        \boldsymbol{\xi}_h^{n+1},
        \phi_h^{n+1})
\in\boldsymbol{V}_{p,h}\times\boldsymbol{V}_{p,h}$ $\times Q_{p,h}$
such that:
\[
\boldsymbol{\xi}_h^{n+1}
  = \frac{\boldsymbol{\eta}_h^{n+1}-\boldsymbol{\eta}_h^n}{\Delta t},
\]
and
\begin{equation}\label{eq:poro-discrete}
\begin{aligned}
&{\small{\rho_p\!\left(
  \frac{\boldsymbol{\xi}_h^{n+1}-\boldsymbol{\xi}_h^n}{\Delta t},
  \boldsymbol{\zeta}\right)_p
+ 2\mu_p\left( \boldsymbol{D}(\boldsymbol{\eta}_h^{n+1}),
\boldsymbol{D}(\boldsymbol{\zeta})\right)_p
 + \lambda_p\left( \nabla\!\cdot\!\boldsymbol{\eta}_h^{n+1},
     \nabla\boldsymbol{\zeta}\right)_p
- \alpha(\phi_h^{n+1},\nabla\!\cdot\!\boldsymbol{\zeta})_p
}}
\\
& \quad
{\small{+ C_0\!\left(
  \frac{\phi_h^{n+1}-\phi_h^n}{\Delta t},
  \psi\right)_p
+ \alpha(\nabla\!\cdot\!\boldsymbol{\xi}_h^{n+1},\psi)_p
+ (K\nabla\phi_h^{n+1},\nabla\psi)_p
+ \gamma\langle
   \boldsymbol{P}_p\boldsymbol{\xi}_h^{n+1},
   \boldsymbol{P}_p\boldsymbol{\zeta}\rangle_{\Gamma}
   }}
 \\
& \quad
+\langle\boldsymbol{\xi}_h^{n+1}\!\cdot\!\boldsymbol{n}_p,
\boldsymbol{\zeta}\!\cdot\!\boldsymbol{n}_p\rangle_{\Gamma}
+ \langle
   \phi_h^{n+1},
   \boldsymbol{\zeta}\cdot\boldsymbol{n}_p\rangle_{\Gamma}
+ \frac{1}{L}\langle\phi_h^{n+1},\psi\rangle_{\Gamma}
- \langle
\boldsymbol{\xi}_h^{n+1}\!\cdot\!\boldsymbol{n}_p,
   \psi\rangle_{\Gamma}
\\
&= (\boldsymbol{f}_p^{n+1},\boldsymbol{\zeta})_p
+ (g_p^{n+1},\psi)_p
+ \langle
   \gamma\,\boldsymbol{P}_p\boldsymbol{u}_h^n,
   \boldsymbol{P}_p\boldsymbol{\zeta}\rangle_{\Gamma}
+ \langle
   \boldsymbol{\xi}_h^n\!\cdot\!\boldsymbol{n}_p,
\boldsymbol{\zeta}\!\cdot\!\boldsymbol{n}_p\rangle_{\Gamma}
\\
& \quad + \left\langle
   -\boldsymbol{u}_h^n\!\cdot\!\boldsymbol{n}_p
   + \tfrac{1}{L}\phi_h^n,
   \psi\right\rangle_{\Gamma}
\end{aligned}
\end{equation}
for all $(\boldsymbol{\zeta}_h,\psi_h)
 \in\boldsymbol{V}_{p,h}\times Q_{p,h}$.
\end{enumerate}
\if 1 = 0
\smallskip
\textcolor{black}{Specifically, We consider the following Robin interface boundary conditions:
\begin{equation}
\begin{array}{ll}
\boldsymbol{n}_f^n \cdot\left(\boldsymbol{\sigma}_f \boldsymbol{n}_f\right)^{n+1}+L \boldsymbol{u}^{n+1} \cdot \boldsymbol{n}_f^n
= L \boldsymbol{u}^n \cdot \boldsymbol{n}_f^n-\phi^n
& \text{on } \Gamma, \\[2mm]
\boldsymbol{P}_f\!\left(\boldsymbol{\sigma}_f \boldsymbol{n}_f\right)^{n+1}
+\gamma \boldsymbol{P}_f\!\left(\boldsymbol{u}^{n+1}\right)
=\gamma \boldsymbol{P}_f\!\left(\boldsymbol{\xi}^n\right)
& \text{on } \Gamma, \\[2mm]
 {\boldsymbol{n}}_p \cdot \left( {\boldsymbol{\sigma}}_p{\boldsymbol{n}}_p
\right)^{n+1}+ {\phi}^{n+1}
+{\boldsymbol{\xi}}^{n+1}\cdot {\boldsymbol{n}}_p
={\boldsymbol{\xi}}^n\cdot {\boldsymbol{n}}_p
& \text{on } {\Gamma}, \\[2mm]
\mathbb{K}{\nabla}{\phi}^{n+1}\cdot {\boldsymbol{n}}_p
+ {\phi}^{n+1}/L
- {\boldsymbol{\xi}}^{n+1}\cdot  {\boldsymbol{n}}_p
= {\phi}^n/L- {\boldsymbol{u}}^n\cdot  {\boldsymbol{n}}_p
& \text{on } {\Gamma}, \\[2mm]
  \boldsymbol{P}_p\!\left( {\boldsymbol{\sigma}}_p {\boldsymbol{n}}_p\right)^{n+1}
+\gamma   \boldsymbol{P}_p\!\left( {\boldsymbol{\xi}}^{n+1}\right)
=\gamma   \boldsymbol{P}_p\!\left( {\boldsymbol{u}}^n\right)
& \text{on } {\Gamma}.
\end{array}
\end{equation}
}
\fi

Equations \eqref{eq:fluid-discrete}–\eqref{eq:poro-discrete} define a fully parallelizable splitting scheme, in which the fluid and poroelastic subproblems are decoupled and can be solved simultaneously at each time step \cite{24M1695713}. The Nitsche-type terms on $\Gamma$ ensure the consistency and stability of the interface coupling while maintaining parallel efficiency \cite{BADIA20097986,BUKAC2015138,MartinaOyekole}. Furthermore, the proposed scheme attains additional stabilization by replacing the stress terms with the pore pressure through the coupling condition.

\section{Ritz Projections and General Approximation Properties}
\label{sec:projections}

To relate the continuous problem~\eqref{eq:continuous-Stokes-eq}--\eqref{eq:continuous-poroelasticity-eq}
and the discrete scheme
\eqref{eq:fluid-discrete}--\eqref{eq:poro-discrete},
{\sunny{we introduce projection operators $\boldsymbol{\Pi}_u$ and ${\Pi}_p$, from continuous space $V_f \times Q_f$ to discrete space $V_{f,h} \times Q_{f,h}$, which are defined
in such a way that they preserve the
bilinear forms of each subproblem.}}
These projections play a key role in the error analysis,
ensuring that interpolation errors disappear from the
variational equations when subtracting continuous and discrete formulations.

\paragraph{Ritz projection for Stokes velocity and pressure:}
Consider the Ritz projections
$\left(\boldsymbol{\Pi}_u \boldsymbol{u}^n, {\Pi}_p p^n\right) : \boldsymbol{V}_f \times Q_p \to  \boldsymbol{V}_{f, h} \times Q_{f, h}$ defined by:
\begin{equation}\label{eq:proj-fluid}
\begin{aligned}
2 & \mu_f\left(\boldsymbol{D}\left(\boldsymbol{\Pi}_u \boldsymbol{u}^n\right), \boldsymbol{D}(\boldsymbol{v})\right)_f
-\left({\Pi}_p p^n, \nabla \cdot \boldsymbol{v}\right)_f+\left(\nabla \cdot \boldsymbol{\Pi}_u \boldsymbol{u}^n, q\right)_f \\
& +\gamma\left\langle \boldsymbol{P}_f\left(\boldsymbol{\Pi}_u \boldsymbol{u}^n\right), \boldsymbol{P}_f(\boldsymbol{v})\right\rangle_{\Gamma}+L\left\langle\boldsymbol{\Pi}_u \boldsymbol{u}^n \cdot \boldsymbol{n}_f, \boldsymbol{v} \cdot \boldsymbol{n}_f\right\rangle_{\Gamma} \\
= & 2 \mu_f\left(\boldsymbol{D}\left(\boldsymbol{u}^n\right), \boldsymbol{D}(\boldsymbol{v})\right)_f-\left(p^n, \nabla \cdot \boldsymbol{v}\right)_f+\left(\nabla \cdot \boldsymbol{u}^n, q\right)_f \\
& +\gamma\left\langle \boldsymbol{P}_f\left(\boldsymbol{u}^n\right), \boldsymbol{P}_f(\boldsymbol{v})\right\rangle_{\Gamma}+L\left\langle\boldsymbol{u}^n \cdot \boldsymbol{n}_f, \boldsymbol{v} \cdot \boldsymbol{n}_f\right\rangle_{\Gamma}, \quad \forall(\boldsymbol{v}, q) \in \boldsymbol{V}_{f, h} \times Q_{f, h} .
\end{aligned}
\end{equation}
{\sunny{We now show the boundedness and continuity of the interface/boundary terms 
$\gamma\left\langle \boldsymbol{P}_f\left(\boldsymbol{u}^n\right), \boldsymbol{P}_f(\boldsymbol{v})\right\rangle_{\Gamma}$ and $L\left\langle\boldsymbol{u}^n \cdot \boldsymbol{n}_f, \boldsymbol{v} \cdot \boldsymbol{n}_f\right\rangle_{\Gamma}$, which can be considered parts of the penalty terms enforcing the interface conditions in the continuous weak formulation. We will use this result to show the Ritz bilinear form \eqref{eq:proj-fluid} associated with the fluid subproblem is well-defined.
}}
\begin{lemma}[Trace boundedness and continuity of the interface penalty terms]\label{BoundaryTerms}
Let $\Omega_f \subset \mathbb{R}^d$ be a Lipschitz domain and let $\Gamma \subset \partial \Omega_f$ denote the fluid structure interface with unit outward normal $\boldsymbol{n}_f$.
Then, for all $\boldsymbol{v} \in [H^1(\Omega_f)]^d$, the tangential projector $\boldsymbol{P}_f$ {\sunny{satisfies}}
\[
\|\boldsymbol{P}_f \boldsymbol{v}\|_{L^2(\Gamma)}
\le
\|\boldsymbol{v}\|_{L^2(\Gamma)}
\le
C_{\mathrm{tr}}\|\boldsymbol{v}\|_{H^1(\Omega_f)} .
\]
Consequently, for all $\boldsymbol{w},\boldsymbol{v} \in [H^1(\Omega_f)]^d$, there exists a constant $C>0$ such that
\[
\left|\left\langle \boldsymbol{P}_f \boldsymbol{w}, \boldsymbol{P}_f \boldsymbol{v}\right\rangle_{\Gamma}\right|
\le
C\|\boldsymbol{w}\|_{H^1(\Omega_f)}\|\boldsymbol{v}\|_{H^1(\Omega_f)},
\]
and
\[
\left|\left\langle \boldsymbol{w}\cdot \boldsymbol{n}_f, \boldsymbol{v}\cdot \boldsymbol{n}_f\right\rangle_{\Gamma}\right|
\le
C\|\boldsymbol{w}\|_{H^1(\Omega_f)}\|\boldsymbol{v}\|_{H^1(\Omega_f)} .
\]
\end{lemma}

{\sunny{We now show the well-posedness of \eqref{eq:proj-fluid}, namely, we show that the Ritz bilinear form is well-defined. For this purpose we introduce the bilinear form describing the left hand-side of \eqref{eq:proj-fluid}:
\begin{equation}\label{RitzForm}
\begin{aligned}
\mathcal{A}_f\big((\boldsymbol{w}, r),(\boldsymbol{v}, q)\big)
:=\;&
2 \mu_f(\boldsymbol{D}(\boldsymbol{w}), \boldsymbol{D}(\boldsymbol{v}))_f
-(r, \nabla \cdot \boldsymbol{v})_f
+(\nabla \cdot \boldsymbol{w}, q)_f \\
&\quad
+\gamma\left\langle \boldsymbol{P}_f(\boldsymbol{w}), \boldsymbol{P}_f(\boldsymbol{v})\right\rangle_{\Gamma}
+L\left\langle \boldsymbol{w} \cdot \boldsymbol{n}_f, \boldsymbol{v} \cdot \boldsymbol{n}_f\right\rangle_{\Gamma},
\end{aligned}
\end{equation}
where 
$(\boldsymbol{w}, r),(\boldsymbol{v}, q) \in \boldsymbol{V}_f \times Q_f$, with $(\boldsymbol{V}_f \times Q_f)$ satisfying the continuous inf-sup condition.
Because of Lemma~\ref{BoundaryTerms} we see that the boundary terms are both well-defined since $\boldsymbol{w}, \boldsymbol{v} \in \boldsymbol{V}_f \subset [H^1(\Omega_f)]^d$. 

Furtheremore, we have the following properties of the Ritz bilinear form $\mathcal{A}_f$ associated with the fluid subproblem:
}}
\begin{lemma}[Well-posedness and boundedness of the fluid Ritz bilinear form]\label{lemma4p2}
{\sunny{Let $\mathcal{A}_f$ be the Ritz bilinear form defined in \eqref{RitzForm}.}} Then there exists a constant $C>0$ (depending on $\mu_f, \gamma, L$, and the trace constant) such that
\[
\left|\mathcal{A}_f\big((\boldsymbol{w}, r),(\boldsymbol{v}, q)\big)\right|
\le
C\left(\|\boldsymbol{w}\|_{H^1(\Omega_f)}+\|r\|_{L^2(\Omega_f)}\right)
\left(\|\boldsymbol{v}\|_{H^1(\Omega_f)}+\|q\|_{L^2(\Omega_f)}\right).
\]
Moreover, for $\gamma \ge 0$ and $L \ge 0$,
\[
\mathcal{A}_f\big((\boldsymbol{v}, q),(\boldsymbol{v},q)\big)
=
2 \mu_f\|\boldsymbol{D}(\boldsymbol{v})\|_{L^2(\Omega_f)}^2
+\gamma\|\boldsymbol{P}_f(\boldsymbol{v})\|_{L^2(\Gamma)}^2
+L\|\boldsymbol{v} \cdot \boldsymbol{n}_f\|_{L^2(\Gamma)}^2 .
\]
\end{lemma}
{\sunny{The proof is a consequence of Lemma~\ref{BoundaryTerms}.}}

{\sunny{Next, we state the following two approximation estimates for the two projection operators $\boldsymbol{\Pi}_u$ and $\Pi_p$, introduced above, which will be used in Section~\ref{sec:residual_estimates} to derive the estimates of the residual terms needed for the final error estimate. }}

\vskip 0.1in
\begin{itemize}[leftmargin=12pt,labelsep=0.5em]
\item{\emph{Approximation estimate 1:}}
For sufficiently smooth $\boldsymbol{u}$ and $p$, the following approximation estimates hold \cite{Brezzi1991,Girault1986}:
\begin{align}
    \label{eq:proj-error-fluid}
    \textcolor{black}{\|\boldsymbol{u}-\boldsymbol{\Pi}_u\boldsymbol{u}\|_{H^1(\Omega_f)} + \|p-\Pi_p p\|_{L^2(\Omega_f)}  \le C h^k (\|\boldsymbol{u}\|_{H^{k+1}(\Omega_f)} + \|p\|_{H^k(\Omega_f)}) . }
\end{align}
\textcolor{black}{We emphasize that the operators $\boldsymbol{\Pi}_u$ and $\Pi_p$ appearing in the above inequality are not constructed independently; instead, they depend jointly on the velocity-pressure pair $(\boldsymbol{u}, p)$.}
\vskip 0.1in
\item {\emph{Approximation estimate 2:}} 
We will use the following generalization of the $L^2$-estimate for the projection $\boldsymbol{\Pi}_u \boldsymbol{u}$ (see \cite{Grisvard:85,Dauge:88,Mazya-Rossmann:10}), adapted to the setting in which the trace of the fluid velocity is not required to vanish on the entire boundary $\partial\Omega_f$. This formulation is particularly suited for the FSI problem considered in this manuscript.
\begin{lemma}\label{LemmaReg2}
Let $\boldsymbol{V}_{f,0}$ be the subspace of $\boldsymbol{V}_f$ such that the trace on $\partial \Omega_f \setminus \Gamma$ vanishes, and let $Q_{f,0}$ be \textcolor{black}{the  mean-value zero functions on $\Omega_f$ if $\Gamma = \emptyset$ and $Q_{f,0} = Q_f$ otherwise}. Assume the following dual regularity property: for any
$\boldsymbol g\in \boldsymbol L^2(\Omega_f)$, the solution
$(\boldsymbol\Psi_u,\Psi_p)\in \boldsymbol V_{f,0}\times Q_{f,0}$ of the dual Stokes problem
\begin{align}\label{eq:dual-stokes-revised}
& 2 \mu_f\bigl(\boldsymbol D(\boldsymbol v), \boldsymbol D(\boldsymbol \Psi_u)\bigr)_f
-\bigl(\Psi_p, \nabla \cdot \boldsymbol v\bigr)_f
+\bigl(\nabla \cdot \boldsymbol \Psi_u, q\bigr)_f \notag\\
& 
+\gamma\bigl\langle \boldsymbol P_f(\boldsymbol v), \boldsymbol P_f(\boldsymbol \Psi_u)\bigr\rangle_{\Gamma}
+L\bigl\langle \boldsymbol v\cdot \boldsymbol n_f, \boldsymbol \Psi_u\cdot \boldsymbol n_f\bigr\rangle_{\Gamma}
=
(\boldsymbol g,\boldsymbol v)_f, \quad \forall (\boldsymbol v,q)\in \boldsymbol V_{f,0}\times Q_{f,0}
\end{align}
satisfies
$
\|\boldsymbol\Psi_u\|_{H^{1+r}(\Omega_f)}
+\|\Psi_p\|_{H^r(\Omega_f)}
\le C\|\boldsymbol g\|_{L^2(\Omega_f)},
~ 0<r\le 1$ 
for $C>0$ depending on $\Omega_f$. 
Then the Stokes projection $(\boldsymbol\Pi_u\boldsymbol u,\Pi_p p)$ satisfies
\begin{equation}\label{eq:proj-L2-error-fluid}
\|\boldsymbol u-\boldsymbol\Pi_u\boldsymbol u\|_{L^2(\Omega_f)}
\le
C h^{k+r}
\left(
\|\boldsymbol u\|_{H^{k+1}(\Omega_f)}
+
\|p\|_{H^k(\Omega_f)}
\right).
\end{equation}
\end{lemma}
\begin{remark} It is known that $r = 1$ when $\Omega_f$ is a convex Lipschitz domain and $\Gamma = \emptyset$. For nonconvex Lipschitz domains, as well as nonconvex polygonal or polyhedral domains, the exponent $0 < r < 1$ depends on the geometric features of $\Omega_f$. We refer to \cite{Grisvard:85,Dauge:88,Mazya-Rossmann:10} for further details.
\end{remark}
\begin{proof}
\textcolor{black}{The proof is a consequence of  a variant of the standard duality argument for Stokes problems. Define the projection errors }
$$
\boldsymbol e_u:=\boldsymbol u-\boldsymbol\Pi_u\boldsymbol u,
\qquad
e_p:=p-\Pi_p p.
$$
\textcolor{black}{Let $(\boldsymbol\Psi_u,\Psi_p)\in \boldsymbol V_{f,0}\times Q_{f,0}$ be the solution of the dual problem \eqref{eq:dual-stokes-revised}
with the forcing term 
$\boldsymbol{g}:=\boldsymbol{e}_u$}. 
Taking $(\boldsymbol v,q)=(\boldsymbol e_u,e_p)$ in this dual problem,
we obtain
\begin{align}
\|\boldsymbol e_u\|_{L^2(\Omega_f)}^2
&=
2\mu_f\bigl(\boldsymbol D(\boldsymbol e_u),\boldsymbol D(\boldsymbol\Psi_u)\bigr)_f
-(\Psi_p,\nabla\cdot \boldsymbol e_u)_f
+(\nabla\cdot \boldsymbol\Psi_u,e_p)_f \notag\\
&\qquad
+\gamma\langle \boldsymbol P_f(\boldsymbol e_u),\boldsymbol P_f(\boldsymbol\Psi_u)\rangle_\Gamma
+L\langle \boldsymbol e_u\cdot \boldsymbol n_f,\boldsymbol\Psi_u\cdot \boldsymbol n_f\rangle_\Gamma.
\label{eq:eueu-dual}
\end{align}
Next, we let $(\boldsymbol\Psi_{u,h},\Psi_{p,h})\in \boldsymbol V_{f,h,0}\times Q_{f,h,0}$ be a suitable finite element
approximation of $(\boldsymbol\Psi_u,\Psi_p)$, where
$\boldsymbol V_{f,h,0}:=\boldsymbol V_{f,h}\cap \boldsymbol V_{f,0}$, $Q_{f,h,0}:=Q_{f,h}\cap Q_{f,0}$.
\textcolor{black}{By the Galerkin orthogonality} 
\begin{align}
\|\boldsymbol e_u\|_{L^2(\Omega_f)}^2
&=
2\mu_f\bigl(\boldsymbol D(\boldsymbol e_u),\boldsymbol D(\boldsymbol\Psi_u-\boldsymbol\Psi_{u,h})\bigr)_f
-(\Psi_p-\Psi_{p,h},\nabla\cdot \boldsymbol e_u)_f
\notag
\\
&+(\nabla\cdot (\boldsymbol\Psi_u-\boldsymbol\Psi_{u,h}),e_p)_f 
+\gamma\langle \boldsymbol P_f(\boldsymbol e_u),\boldsymbol P_f(\boldsymbol\Psi_u-\boldsymbol\Psi_{u,h})\rangle_\Gamma
\notag
\\
&+L\langle \boldsymbol e_u\cdot \boldsymbol n_f,
(\boldsymbol\Psi_u-\boldsymbol\Psi_{u,h})\cdot \boldsymbol n_f\rangle_\Gamma.
\notag
\end{align}
\textcolor{black}{Employing the Cauchy--Schwarz inequality and the trace inequality from Lemma~\ref{lemma4p2}} we get
\begin{align}
\|\boldsymbol e_u\|_{L^2(\Omega_f)}^2
&\le
C\|\boldsymbol e_u\|_{H^1(\Omega_f)}
\|\boldsymbol\Psi_u-\boldsymbol\Psi_{u,h}\|_{H^1(\Omega_f)}
+C\|\boldsymbol e_u\|_{H^1(\Omega_f)}
\|\Psi_p-\Psi_{p,h}\|_{L^2(\Omega_f)} \nonumber \\
&\quad
+C\|e_p\|_{L^2(\Omega_f)}
\|\boldsymbol\Psi_u-\boldsymbol\Psi_{u,h}\|_{H^1(\Omega_f)}\nonumber \\
\le& 
{\small{C\Bigl(
\|\boldsymbol e_u\|_{H^1(\Omega_f)}
+\|e_p\|_{L^2(\Omega_f)}
\Bigr)
\Bigl(
\|\boldsymbol\Psi_u-\boldsymbol\Psi_{u,h}\|_{H^1(\Omega_f)}
+\|\Psi_p-\Psi_{p,h}\|_{L^2(\Omega_f)}
\Bigr).}}
\end{align}
\textcolor{black}{Now choose} $(\boldsymbol\Psi_{u,h},\Psi_{p,h})$ \textcolor{black}{such that}
\[
\|\boldsymbol\Psi_u-\boldsymbol\Psi_{u,h}\|_{H^1(\Omega_f)}
+\|\Psi_p-\Psi_{p,h}\|_{L^2(\Omega_f)}
\le
C h^r
\Bigl(
\|\boldsymbol\Psi_u\|_{H^{1+r}(\Omega_f)}
+\|\Psi_p\|_{H^r(\Omega_f)}
\Bigr).
\]
Then, the estimate provided in the regularity assumption on the solution to the dual Stokes problem, stated in Lemma~\ref{LemmaReg2}:
\[
\|\boldsymbol\Psi_u\|_{H^{1+r}(\Omega_f)}
+\|\Psi_p\|_{H^r(\Omega_f)}
\le
C\|\boldsymbol e_u\|_{L^2(\Omega_f)},
\]
together with
\[
\|\boldsymbol e_u\|_{H^1(\Omega_f)}
+\|e_p\|_{L^2(\Omega_f)}
\le
C h^k
\left(
\|\boldsymbol u\|_{H^{k+1}(\Omega_f)}
+
\|p\|_{H^k(\Omega_f)}
\right),
\]
\textcolor{black}{yield the conclusion of the lemma.}
\end{proof}
\end{itemize}

\paragraph{Ritz projection for Biot structure displacement}
We introduce the following Ritz projection $\boldsymbol{\Pi}_\eta:\boldsymbol{V}_p\to\boldsymbol{V}_{p,h}$ 
for the Biot structure displacement:
\begin{equation}\label{eq:proj-structure}
\begin{array}{ll}
&~~~~2 \mu_p\left(\boldsymbol{D}\left(\boldsymbol{\Pi}_\eta \boldsymbol{\eta}^n\right), \boldsymbol{D}(\boldsymbol{\zeta})\right)_p
+\lambda_p\left(\nabla \cdot \boldsymbol{\Pi}_\eta \boldsymbol{\eta}^n, \nabla \cdot \boldsymbol{\zeta}\right)_p  \\
&= 2 \mu_p\left(\boldsymbol{D}\left(\boldsymbol{\eta}^n\right), \boldsymbol{D}(\boldsymbol{\zeta})\right)_p+\lambda_p\left(\nabla \cdot \boldsymbol{\eta}^n, \nabla \cdot \boldsymbol{\zeta}\right)_p, \quad\quad\quad\forall \zeta \in V_{p, h}.
\end{array}
\end{equation}
{\sunny{Similarly as before, we have the following approximation estimates:}}
\vskip 0.1in
\begin{itemize}[leftmargin=12pt,labelsep=0.5em]
\item{\emph{Approximation estimate 3:}}
For sufficiently smooth $\boldsymbol{\eta}$, the following holds \cite{Ciarlet2002, Brenner2008, Chen2024}:
\begin{align}
    \label{eq:proj-error-structure}
    \|\boldsymbol{\eta}-\boldsymbol{\Pi}_\eta\boldsymbol{\eta}\|_{H^1(\Omega_p)}
  \le C h^k \|\boldsymbol{\eta}\|_{H^{k+1}(\Omega_p)}.
\end{align}
\item{\emph{Approximation estimate 4:}}
\textcolor{black}{Let $V_{p,0}$ be the subspace of $V_p$ with vanishing trace on $\partial \Omega_p$. {\sunny{Suppose}} that solutions of the elasticity equations with the homogeneous Dirichlet boundary condition are $(r+1)$-regular on $\Omega_p$ ($0<r\le 1$) {\sunny{in the sense that}}  the solution $\boldsymbol{\Psi}_{\eta} \in V_{p,0}$ of 
\begin{align*}
    2 \mu_p\left(\boldsymbol{D}\left(\boldsymbol{\Psi}_\eta \right), \boldsymbol{D}(\boldsymbol{\zeta})\right)_p
+\lambda_p\left(\nabla \cdot \boldsymbol{\Psi}_\eta , \nabla \cdot \boldsymbol{\zeta}\right)_p = (\boldsymbol{g}, \boldsymbol{\zeta})_p
\end{align*}
satisfies $\|\boldsymbol{\Psi}_{\eta}\|_{H^{1+r}(\Omega_p)} \le C \|\boldsymbol{g}\|_{L^2(\Omega_p)}$ for $C>0$ depending on $\Omega_p$. 
\vskip 0.1in
Then, 
\begin{align}
    \label{eq:proj-L2-error-structure}
    \|\boldsymbol{\eta}-\boldsymbol{\Pi}_\eta\boldsymbol{\eta}\|_{L^2(\Omega_p)}
  \le C h^{k+r} \|\boldsymbol{\eta}\|_{H^{k+1}(\Omega_p)}.
\end{align}
}
{\sunny{See \cite{Ciarlet2002, Brenner2008} for more details.}}
\end{itemize}
\vskip 0.1in
\paragraph{Ritz projection for Biot pore pressure }
We introduce the following Ritz projection $\Pi_\phi:Q_p\to Q_{p,h}$ for the Biot pore pressure:
\begin{equation}\label{eq:proj-pressure}
\left(K \nabla \Pi_\phi \phi^n, \nabla \psi\right)_p+\frac{1}{L}\left\langle\Pi_\phi \phi^n, \psi\right\rangle_{\Gamma}=\left(K \nabla \phi^n, \nabla \psi\right)_p+\frac{1}{L}\left\langle\phi^n, \psi\right\rangle_{\Gamma}, \quad \forall \psi \in Q_{p, h} .
\end{equation}
{\sunny{Similarly as before, we have the following approximation estimates:}}
\vskip 0.1in
\begin{itemize}[leftmargin=12pt,labelsep=0.5em]
\item{\emph{Approximation estimate 5:}}
Standard elliptic projection estimates yield:
\begin{align}
    \label{eq:proj-error-pressure}
    \|\phi-\Pi_\phi \phi\|_{H^1(\Omega_p)}
    \le C h^k \|\phi\|_{H^{k+1}(\Omega_p)}. 
\end{align}
\item{\emph{Approximation estimate 6:}}
The $(r+1)$-regularity for Poisson equation on $\Omega_p$ with the homogeneous Dirichlet boundary condition and  $0<r\le 1$  can be defined similarly as before. 
In the same vein, if the $(r+1)$-regularity holds for Poisson equation, then 
\begin{align}
    \label{eq:improved-L2-approx}
    \|\phi-\Pi_\phi \phi\|_{L^2(\Omega_p)} \le C h^{k+r} \|\phi\|_{H^{k+1}(\Omega_p)}.
\end{align}
\end{itemize}
For simplicity of presentation in the rest of this paper we assume that the $r$-regularity assumption for the Stokes, elasticity, and Poisson equations holds 
for the same $0<r\le 1$.

{\sunny{We conclude this section  by summarizing the trace and approximation inequalities that will be used later in this manuscript.
}}
\paragraph{Trace and approximation inequalities}

For all ${v}\in H^{1}(\Omega_*)$
with $\Omega_*\in\{\Omega_f,\Omega_p\}$,
the standard trace inequality holds (cf.~\cite[(1.6.2)]{Brenner2008}):
\begin{equation}\label{eq:trace}
    \|{v}\|_{L^2(\Gamma)} \le C \|{v}\|_{H^1(\Omega_*)}^{\frac 12} \|{v}\|_{L^2(\Omega_*)}^{\frac 12} \le C \|{v}\|_{H^1(\Omega_*)}.
\end{equation}
Assuming that the Poincare inequality and Korn's inequality are available, we can derive  
\begin{align}
    \label{eq:scalar-trace}
    \|{v}\|_{L^2(\Gamma)} &\le C \|\nabla v \|_{L^2(\Omega_*)}^{\frac 12} \|{v}\|_{L^2(\Omega_*)}^{\frac 12}, \qquad v \in H^1(\Omega_*),
    \\
    \label{eq:vector-trace}
    \|\boldsymbol{v}\|_{L^2(\Gamma)} &\le C \|\boldsymbol{D}(\boldsymbol{v})\|_{L^2(\Omega_*)}^{\frac 12} \|\boldsymbol{v}\|_{L^2(\Omega_*)}^{\frac 12}, \qquad \boldsymbol{v} \in H^1(\Omega_*)^d .
\end{align}
Combining~\eqref{eq:scalar-trace} and \eqref{eq:vector-trace} with the projection estimates
\eqref{eq:proj-error-fluid}--\eqref{eq:improved-L2-approx}
gives:
\begin{equation}\label{eq:trace-approx}
\|\boldsymbol{v}-\boldsymbol{\Pi}\boldsymbol{v}\|_{L^2(\Gamma)}
  \le C h^k \|\boldsymbol{v}\|_{H^{k+1}(\Omega_*)},
\end{equation}
for the projections
$\boldsymbol{\Pi}=\boldsymbol{\Pi}_{u}, \boldsymbol{\Pi}_{\eta}, \Pi_p, \Pi_\phi$.
These estimates will be used repeatedly in the error analysis
in Section~\ref{sec:error}.

\section{Error Equations}
\label{sec:error}

In this section we derive the error equations
satisfied by the discrete solutions of
\eqref{eq:fluid-discrete}--\eqref{eq:poro-discrete}.
The starting point is to subtract the fully discrete scheme
from the corresponding continuous problem
\eqref{eq:continuous-Stokes-eq}--\eqref{eq:continuous-poroelasticity-eq} evaluated at the time level $t_{n+1}$.
The projection operators introduced in
Section~\ref{sec:projections}
will be used to eliminate interpolation terms from the bilinear forms.

\subsection{Error decomposition}

{\sunny{For each variable, we decompose the total error into the sum of an interpolation error and a discretization error. 
The superscript $I$ denotes the projection (interpolation) error
and the superscript $h$ the discrete (numerical) error:}}
%
\begin{equation}\label{eq:error-split}
\begin{aligned}
\boldsymbol{e}_u^{\,n+1}
  &= \boldsymbol{u}^{n+1} - \boldsymbol{u}_h^{\,n+1}
   = (\boldsymbol{u}^{n+1}-\boldsymbol{\Pi}_u\boldsymbol{u}^{n+1})
     + (\boldsymbol{\Pi}_u\boldsymbol{u}^{n+1}-\boldsymbol{u}_h^{\,n+1})
     \\
     &\qquad\qquad\qquad\quad
     =: -\boldsymbol{e}_u^{I,n+1} + \boldsymbol{e}_u^{h,n+1},\\[0.3em]
e_p^{\,n+1}
  &= p^{n+1} - p_h^{\,n+1}
   = -e_p^{I,n+1} + e_p^{h,n+1},\\[0.3em]
\boldsymbol{e}_\eta^{\,n+1}
  &= \boldsymbol{\eta}^{n+1} - \boldsymbol{\eta}_h^{\,n+1}
   = -\boldsymbol{e}_\eta^{I,n+1} + \boldsymbol{e}_\eta^{h,n+1},\\[0.3em]
\boldsymbol{e}_\xi^{\,n+1}
  &= \boldsymbol{\xi}^{n+1} - \boldsymbol{\xi}_h^{\,n+1}
   = -\boldsymbol{e}_\xi^{I,n+1} + \boldsymbol{e}_\xi^{h,n+1},\\[0.3em]
e_\varphi^{\,n+1}
  &= \varphi^{n+1} - \varphi_h^{\,n+1}
   = -e_\varphi^{I,n+1} + e_\varphi^{h,n+1}.
\end{aligned}
\end{equation}
Here, \textcolor{black}{we set $\boldsymbol{\Pi}_{\xi} = \boldsymbol{\Pi}_{\eta}$.} 
\subsection{Subtraction of continuous and discrete problems}

Subtracting the discrete fluid equation
\eqref{eq:fluid-discrete} from the continuous one
in~\eqref{eq:continuous-Stokes-eq} {\sunny{evaluated at $t^{n+1}$}}, and using the error decomposition
\eqref{eq:error-split}, we obtain
for all $(\boldsymbol{v}_h,q_h)\in
\boldsymbol{V}_{f,h}\times Q_{f,h}$:
\begin{align}
	\label{eq:Stokes-error-eq}
& \rho_{f}\left(\frac{e_{\boldsymbol{u}}^{n+1} - e_{\boldsymbol{u}}^{n}}{\Delta t}, \boldsymbol{v}\right)_{f} + 2\mu_{f}(\boldsymbol{D}(e_{\boldsymbol{u}}^{n+1}), \boldsymbol{D}(\boldsymbol{v}))_{f} - (e_{p}^{n+1}, \nabla\cdot \boldsymbol{v})_{f} + (\nabla\cdot e_{\boldsymbol{u}}^{n+1}, q)_{f} 
	\\
	\notag
	& + \gamma\langle P_{f}(e_{\boldsymbol{u}}^{n+1}), P_{f}(\boldsymbol{v})\rangle_{\Gamma} + L\langle e_{\boldsymbol{u}}^{n+1}\cdot \boldsymbol{n}_{f}, \boldsymbol{v}\cdot \boldsymbol{n}_{f}\rangle_{\Gamma} 
	\\
	\notag
	= ~& \langle \gamma P_{f}(\boldsymbol{\xi}^{n+1} - \boldsymbol{\xi}_h^{n}), P_{f}(\boldsymbol{v})\rangle_{\Gamma} 
	\\
	\notag
	& + \langle L (\boldsymbol{u}^{n+1}-\boldsymbol{u}_h^{n})\cdot \boldsymbol{n}_{f} - (\phi^{n+1} - \phi_h^n), \boldsymbol{v}\cdot \boldsymbol{n}_{f}\rangle_{\Gamma} + \rho_{f}\left(\frac{\boldsymbol{u}^{n+1} - \boldsymbol{u}^{n}}{\Delta t} - \partial_t \boldsymbol{u}^{n+1}, \boldsymbol{v}\right)_{f} .
\end{align}
Similarly, subtracting the discrete poroelastic problem
\eqref{eq:poro-discrete}  from its continuous counterpart \eqref{eq:continuous-poroelasticity-eq} {\sunny{evaluated at $t^{n+1}$}}
yields, for all $(\boldsymbol{\zeta}_h,\psi_h)
\in\boldsymbol{V}_{p,h}\times Q_{p,h}$,
\begin{align}
	\label{eq:poroelasticity-error-eq}
&\rho_{p}\left(\frac{e_{\boldsymbol{\xi}}^{n+1} - e_{\boldsymbol{\xi}}^{n}}{\Delta t}, \boldsymbol{\zeta}\right)_{p} + 2\mu_{p}(\boldsymbol{D}(e_{\boldsymbol{\eta}}^{n+1}), \boldsymbol{D}(\boldsymbol{\zeta}))_{p} + \lambda_{p}(\nabla\cdot e_{\boldsymbol{\eta}}^{n+1}, \nabla\cdot \boldsymbol{\zeta})_{p} 
	\\
	\notag
&- \alpha(e_{\phi}^{n+1}, \nabla\cdot \boldsymbol{\zeta})_{p} + C_{0}\left(\frac{e_{\phi}^{n+1} - e_{\phi}^{n}}{\Delta t}, \psi\right)_{p} + \alpha(\nabla\cdot e_{\boldsymbol{\xi}}^{n+1}, \psi)_{p} 
	\\
	\notag
&+ (K\nabla e_{\phi}^{n+1}, \nabla\psi)_{p} + \gamma\langle P_{p}(e_{\boldsymbol{\xi}}^{n+1}), P_{p}(\boldsymbol{\zeta})\rangle_{\Gamma} + \langle e_{\boldsymbol{\xi}}^{n+1}\cdot \boldsymbol{n}_{p}, \boldsymbol{\zeta}\cdot \boldsymbol{n}_{p}\rangle_{\Gamma} 
	\\
	\notag
&+ \langle e_{\phi}^{n+1}, \boldsymbol{\zeta}\cdot \boldsymbol{n}_{p}\rangle_{\Gamma} + \frac{1}{L}\langle e_{\phi}^{n+1}, \psi\rangle_{\Gamma} - \langle e_{\boldsymbol{\xi}}^{n+1}\cdot \boldsymbol{n}_{p}, \psi\rangle_{\Gamma} 
	\\
	\notag
=~& \langle \gamma P_{p}(\boldsymbol{u}^{n+1}- \boldsymbol{u}_h^{n}), P_{p}(\boldsymbol{\zeta})\rangle_{\Gamma} 
	\\
	\notag
& + \langle ( \boldsymbol{\xi}^{n+1}- \boldsymbol{\xi}_h^{n})\cdot \boldsymbol{n}_{p}, \boldsymbol{\zeta}\cdot \boldsymbol{n}_{p}\rangle_{\Gamma} + \langle -(\boldsymbol{u}^{n+1}-\boldsymbol{u}_h^{n})\cdot \boldsymbol{n}_{p} + ({\phi}^{n+1} - \phi_h^n)/L, \psi\rangle_{\Gamma} 
	\\
	\notag
& + \rho_{p}\left(\frac{\boldsymbol{\xi}^{n+1} - \boldsymbol{\xi}^{n}}{\Delta t} - \partial_t \boldsymbol{\xi}^{n+1} , \boldsymbol{\zeta}\right)_{p} + C_0 \left( \frac{ \phi^{n+1} - \phi^{n}}{\Delta t} - \partial_t \phi^{n+1} , \psi \right)_{p} .
\end{align}

Because the projections
$\boldsymbol{\Pi}_u$, $\boldsymbol{\Pi}_\eta$, and $\Pi_\phi$
are defined through the same bilinear forms as those used in the discrete scheme
(see Section~\ref{sec:projections}),
the terms involving only interpolation errors vanish. Therefore, the leading-order consistency errors are confined to
the time-discretization residuals
and the lagged interface data on the RHS, {\sunny{giving the following}} reduced error equations:

\begin{equation}
\begin{aligned}\label{eq:fluid-error-eq}
	&{\small{\rho_{f}\left(\frac{e_{\boldsymbol{u}}^{h,n+1} - e_{\boldsymbol{u}}^{h,n}}{\Delta t}, \boldsymbol{v}\right)_{f} + 2\mu_{f}(\boldsymbol{D}(e_{\boldsymbol{u}}^{h,n+1}), \boldsymbol{D}(\boldsymbol{v}))_{f} - (e_{p}^{h,n+1}, \nabla\cdot \boldsymbol{v})_{f} + (\nabla\cdot e_{\boldsymbol{u}}^{h,n+1}, q)_{f} 
    }}
	\\
	&+ \gamma\langle P_{f}(e_{\boldsymbol{u}}^{h,n+1}), P_{f}(\boldsymbol{v})\rangle_{\Gamma} + L\langle e_{\boldsymbol{u}}^{h,n+1}\cdot \boldsymbol{n}_{f}, \boldsymbol{v}\cdot \boldsymbol{n}_{f}\rangle_{\Gamma} 
	\\
	&= \langle \gamma P_{f}(\boldsymbol{\xi}^{n+1} - \boldsymbol{\xi}^{n} + e_{\boldsymbol{\xi}}^{h,n} - e_{\boldsymbol{\xi}}^{I,n}), P_{f}(\boldsymbol{v})\rangle_{\Gamma} 
	\\
	& + \langle L (\boldsymbol{u}^{n+1}-\boldsymbol{u}^{n} + e_{\boldsymbol{u}}^{h,n} - e_{\boldsymbol{u}}^{I,n})\cdot \boldsymbol{n}_{f} - (\phi^{n+1} - \phi^n + e_{\phi}^{h,n} -e_{\phi}^{I,n} ), \boldsymbol{v}\cdot \boldsymbol{n}_{f}\rangle_{\Gamma} 
	\\
	& + \rho_{f}\left(\frac{\Pi_{\boldsymbol{u}} \boldsymbol{u}^{n+1} - \Pi_{\boldsymbol{u}} \boldsymbol{u}^{n}}{\Delta t} - \partial_t \boldsymbol{u}^{n+1}, \boldsymbol{v}\right)_{f}
\end{aligned}
\end{equation}
and
\begin{equation}
\begin{aligned}\label{eq:poro-error-eq}
&{\small{\rho_{p}\left(\frac{e_{\boldsymbol{\xi}}^{h,n+1} - e_{\boldsymbol{\xi}}^{h,n}}{\Delta t}, \boldsymbol{\zeta}\right)_{p} 
\hskip -0.05in 
+ 2\mu_{p}(\boldsymbol{D}(e_{\boldsymbol{\eta}}^{h,n+1}), \boldsymbol{D}(\boldsymbol{\zeta}))_{p} + \lambda_{p}(\nabla\cdot e_{\boldsymbol{\eta}}^{h,n+1}, \nabla\cdot \boldsymbol{\zeta})_{p} 
- \alpha(e_{\phi}^{h,n+1}, \nabla\cdot \boldsymbol{\zeta})_{p}
}}\\	
&+ {\small{C_{0}\left(\frac{e_{\phi}^{h,n+1} - e_{\phi}^{h,n}}{\Delta t}, \psi\right)_{p} 
+ \alpha(\nabla\cdot e_{\boldsymbol{\xi}}^{h,n+1}, \psi)_{p} 
+ (K\nabla e_{\phi}^{h,n+1}, \nabla\psi)_{p} 
+ \gamma\langle P_{p}(e_{\boldsymbol{\xi}}^{h,n+1}), P_{p}(\boldsymbol{\zeta})\rangle_{\Gamma}}} \\
	&
+ \langle e_{\boldsymbol{\xi}}^{h,n+1}\cdot \boldsymbol{n}_{p}, \boldsymbol{\zeta}\cdot \boldsymbol{n}_{p}\rangle_{\Gamma} 
+ \langle e_{\phi}^{h,n+1}, \boldsymbol{\zeta}\cdot \boldsymbol{n}_{p}\rangle_{\Gamma} + \frac{1}{L}\langle e_{\phi}^{h,n+1}, \psi\rangle_{\Gamma} - \langle e_{\boldsymbol{\xi}}^{h,n+1}\cdot \boldsymbol{n}_{p}, \psi\rangle_{\Gamma} 
	\\
	&= - \alpha(e_{\phi}^{I,n+1}, \nabla\cdot \boldsymbol{\zeta})_{p} + \alpha(\nabla\cdot e_{\boldsymbol{\xi}}^{I,n+1}, \psi)_{p} + \gamma \langle \boldsymbol{P}_p (e_{\boldsymbol{\xi}}^{I,n+1}), \boldsymbol{P}_p( \boldsymbol{\zeta}) \rangle_{\Gamma}
	\\
	& + \langle e_{\boldsymbol{\xi}}^{I,n+1}\cdot \boldsymbol{n}_{p}, \boldsymbol{\zeta}\cdot \boldsymbol{n}_{p}\rangle_{\Gamma} + \langle e_{\phi}^{I,n+1}, \boldsymbol{\zeta}\cdot \boldsymbol{n}_{p}\rangle_{\Gamma} - \langle e_{\boldsymbol{\xi}}^{I,n+1}\cdot \boldsymbol{n}_{p}, \psi\rangle_{\Gamma} 
	\\
	& + \langle \gamma P_{p}((\boldsymbol{u}^{n+1}- \boldsymbol{u}^{n}) + e_{\boldsymbol{u}}^{h,n} - e_{\boldsymbol{u}}^{I,n}), P_{p}(\boldsymbol{\zeta})\rangle_{\Gamma} 
	\\
	& + \langle ( \boldsymbol{\xi}^{n+1}- \boldsymbol{\xi}^{n} + e_{\boldsymbol{\xi}}^{h,n} - e_{\boldsymbol{\xi}}^{I,n})\cdot \boldsymbol{n}_{p}, \boldsymbol{\zeta}\cdot \boldsymbol{n}_{p}\rangle_{\Gamma} 
	\\
	& + \langle -(\boldsymbol{u}^{n+1}-\boldsymbol{u}^{n} + e_{\boldsymbol{u}}^{h,n} - e_{\boldsymbol{u}}^{I,n})\cdot \boldsymbol{n}_{p} + ({\phi}^{n+1} - \phi^{n} + e_{\phi}^{h,n} - e_{\phi}^{I,n} )/L, \psi\rangle_{\Gamma} 
	\\
	& + \rho_{p}\left(\frac{\Pi_{ \boldsymbol{\eta}} \boldsymbol{\xi}^{n+1} - \Pi_{ \boldsymbol{\eta}} \boldsymbol{\xi}^{n}}{\Delta t} - \partial_t \boldsymbol{\xi}^{n+1} , \boldsymbol{\zeta}\right)_{p} + C_{0}\left(\frac{\Pi_{\phi} {\phi}^{n+1} - \Pi_{\phi} {\phi}^{n}}{\Delta t} - \partial_t \phi^{n+1}, \psi\right)_{p} .
\end{aligned}
\end{equation}

\section{Discrete Error Energy Balance}
\label{sec:energy}
We now derive a discrete energy balance {\sunny{satisfied by the errors}}, which is the foundation
for the apriori error estimate.
{\sunny{We start}} by testing the error equations
\eqref{eq:fluid-error-eq}--\eqref{eq:poro-error-eq} with the discrete errors $\boldsymbol{v}=e_{\boldsymbol{u}}^{h, n+1}, q=e_p^{h, n+1}, \boldsymbol{\zeta}=e_{\xi}^{h, n+1}, \psi=e_\phi^{h, n+1}$.
We then add the resulting equations, and multiply the sum by $2 \Delta t$. The resulting sum is:
\begin{equation*}
\begin{aligned}
&2 \rho_{f}\left(e_{\boldsymbol{u}}^{h,n+1} - e_{\boldsymbol{u}}^{h,n}, e_{\boldsymbol{u}}^{h,n+1}\right)_{f} + 2 \Delta t \cdot 2\mu_{f}(\boldsymbol{D}(e_{\boldsymbol{u}}^{h,n+1}), \boldsymbol{D}(e_{\boldsymbol{u}}^{h,n+1}))_{f} 
	\\
	&- (e_{p}^{h,n+1}, \nabla\cdot e_{\boldsymbol{u}}^{h,n+1})_{f} + (\nabla\cdot e_{\boldsymbol{u}}^{h,n+1}, e_{p}^{h,n+1})_{f} 
	\\
	&+ 2\Delta t \gamma\langle P_{f}(e_{\boldsymbol{u}}^{h,n+1}), P_{f}(e_{\boldsymbol{u}}^{h,n+1})\rangle_{\Gamma} + 2\Delta t L\langle e_{\boldsymbol{u}}^{h,n+1}\cdot \boldsymbol{n}_{f}, e_{\boldsymbol{u}}^{h,n+1}\cdot \boldsymbol{n}_{f}\rangle_{\Gamma} 
	\\
	&+ 2 \rho_{p}\left(e_{\boldsymbol{\xi}}^{h,n+1} - e_{\boldsymbol{\xi}}^{h,n}, e_{\boldsymbol{\xi}}^{h,n+1}\right)_{p} + 2\cdot 2\mu_{p}(\boldsymbol{D}(e_{\boldsymbol{\eta}}^{h,n+1}), \boldsymbol{D}(e_{\boldsymbol{\eta}}^{h,n+1} - e_{\boldsymbol{\eta}}^{h,n}))_{p} 
	\\
	&+ 2 \lambda_{p}(\nabla\cdot e_{\boldsymbol{\eta}}^{h,n+1}, \nabla\cdot (e_{\boldsymbol{\eta}}^{h,n+1} - e_{\boldsymbol{\eta}}^{h,n}) )_{p} - 2 \Delta t \alpha(e_{\phi}^{h,n+1}, \nabla\cdot e_{\boldsymbol{\xi}}^{h,n+1})_{p} 
\\
	&+ {\small{2 C_{0}\left(e_{\phi}^{h,n+1} - e_{\phi}^{h,n}, e_{\phi}^{h,n+1}\right)_{p} + 2 \Delta t \alpha(\nabla\cdot e_{\boldsymbol{\xi}}^{h,n+1}, e_{\phi}^{h,n+1})_{p} + 2 \Delta t  (K\nabla e_{\phi}^{h,n+1}, \nabla e_{\phi}^{h,n+1})_{p}
    }}
	\\
	&+ 2 \Delta t \gamma\langle P_{p}(e_{\boldsymbol{\xi}}^{h,n+1}), P_{p}(e_{\boldsymbol{\xi}}^{h,n+1})\rangle_{\Gamma} + 2 \Delta t \langle e_{\boldsymbol{\xi}}^{h,n+1}\cdot \boldsymbol{n}_{p}, e_{\boldsymbol{\xi}}^{h,n+1}\cdot \boldsymbol{n}_{p}\rangle_{\Gamma} 
	\\
	&+ 2 \Delta t \langle e_{\phi}^{h,n+1}, e_{\boldsymbol{\xi}}^{h,n+1}\cdot \boldsymbol{n}_{p}\rangle_{\Gamma} + 2 \Delta t \frac{1}{L}\langle e_{\phi}^{h,n+1}, e_{\phi}^{h,n+1}\rangle_{\Gamma} - 2 \Delta t \langle e_{\boldsymbol{\xi}}^{h,n+1}\cdot \boldsymbol{n}_{p}, e_{\phi}^{h,n+1}\rangle_{\Gamma} 
		\end{aligned}
\end{equation*}
\begin{equation*}
\begin{aligned}
	&= 2\Delta t \langle \gamma P_{f}(\boldsymbol{\xi}^{n+1} - \boldsymbol{\xi}^{n} + e_{\boldsymbol{\xi}}^{h,n} - e_{\boldsymbol{\xi}}^{I,n}), P_{f}(e_{\boldsymbol{u}}^{h,n+1})\rangle_{\Gamma} 
\\  
	& + 2\Delta t \langle L (\boldsymbol{u}^{n+1}-\boldsymbol{u}^{n} + e_{\boldsymbol{u}}^{h,n} - e_{\boldsymbol{u}}^{I,n})\cdot \boldsymbol{n}_{f} - (\phi^{n+1} - \phi^n + e_{\phi}^{h,n} -e_{\phi}^{I,n} ), e_{\boldsymbol{u}}^{h,n+1}\cdot \boldsymbol{n}_{f}\rangle_{\Gamma} 
\\
	& + 2\Delta t \rho_{f}\left(\frac{\Pi_{\boldsymbol{u}} \boldsymbol{u}^{n+1} - \Pi_{\boldsymbol{u}} \boldsymbol{u}^{n}}{\Delta t} - \partial_t \boldsymbol{u}^{n+1}, e_{\boldsymbol{u}}^{h,n+1}\right)_{f}
	\\
	& - {\small{2\Delta t \alpha(e_{\phi}^{I,n+1}, \nabla\cdot e_{\boldsymbol{\xi}}^{h,n+1})_{p} + 2\Delta t \alpha(\nabla\cdot e_{\boldsymbol{\xi}}^{I,n+1}, e_{\phi}^{h,n+1})_{p} + 2\Delta t \gamma \langle \boldsymbol{P}_p (e_{\boldsymbol{\xi}}^{I,n+1}), \boldsymbol{P}_p(e_{\boldsymbol{\xi}}^{h,n+1}) \rangle_{\Gamma}
    }}
	\\
	& + {\small{2\Delta t \langle e_{\boldsymbol{\xi}}^{I,n+1}\cdot \boldsymbol{n}_{p}, e_{\boldsymbol{\xi}}^{h,n+1}\cdot \boldsymbol{n}_{p}\rangle_{\Gamma} + 2\Delta t \langle e_{\phi}^{I,n+1}, e_{\boldsymbol{\xi}}^{h,n+1}\cdot \boldsymbol{n}_{p}\rangle_{\Gamma} - 2\Delta t \langle e_{\boldsymbol{\xi}}^{I,n+1}\cdot \boldsymbol{n}_{p}, e_{\phi}^{h,n+1}\rangle_{\Gamma}
    }}
	\\
	& + 2\Delta t \langle \gamma P_{p}((\boldsymbol{u}^{n+1}- \boldsymbol{u}^{n}) + e_{\boldsymbol{u}}^{h,n} - e_{\boldsymbol{u}}^{I,n}), P_{p}(e_{\boldsymbol{\xi}}^{h,n+1})\rangle_{\Gamma} 
	\\
	& + 2\Delta t \langle ( \boldsymbol{\xi}^{n+1}- \boldsymbol{\xi}^{n} + e_{\boldsymbol{\xi}}^{h,n} - e_{\boldsymbol{\xi}}^{I,n})\cdot \boldsymbol{n}_{p}, e_{\boldsymbol{\xi}}^{h,n+1}\cdot \boldsymbol{n}_{p}\rangle_{\Gamma} 
	\\
	& + 2\Delta t \langle -(\boldsymbol{u}^{n+1}-\boldsymbol{u}^{n} + e_{\boldsymbol{u}}^{h,n} - e_{\boldsymbol{u}}^{I,n})\cdot \boldsymbol{n}_{p} + ({\phi}^{n+1} - \phi^{n} + e_{\phi}^{h,n} - e_{\phi}^{I,n} )/L, e_{\phi}^{h,n+1}\rangle_{\Gamma} 
	\\
	& + 2\Delta t \rho_{p}\left(\frac{\Pi_{ \boldsymbol{\eta}} \boldsymbol{\xi}^{n+1} - \Pi_{ \boldsymbol{\eta}} \boldsymbol{\xi}^{n}}{\Delta t} - \partial_t \boldsymbol{\xi}^{n+1} , e_{\boldsymbol{\xi}}^{h,n+1}\right)_{p}\\
	& + 2\Delta t C_{0}\left(\frac{\Pi_{\phi} {\phi}^{n+1} - \Pi_{\phi} {\phi}^{n}}{\Delta t} - \partial_t \phi^{n+1}, e_{\phi}^{h,n+1}\right)_{p} .
\end{aligned}
\end{equation*}
By using the relation $e_{\boldsymbol{\xi}}^{h,n+1} = (e_{\boldsymbol{\eta}}^{h,n+1} - e_{\boldsymbol{\eta}}^{h,n})/\Delta t$ and applying the polarization identity:
$$
2(a-b, a)=\|a\|^2-\|b\|^2+\|a-b\|^2,
$$
we can rewrite the above equation as follows: 
\begin{equation}
\begin{aligned}\label{eq:error_energy}
	&\|e_{\boldsymbol{u}}^{h,n+1}\|_{\rho_f, \Omega_f}^2 - \|e_{\boldsymbol{u}}^{h,n}\|_{\rho_f, \Omega_f}^2 + \|e_{\boldsymbol{u}}^{h,n+1}-e_{\boldsymbol{u}}^{h,n}\|_{\rho_f, \Omega_f}^2
	\\
	&+ 2\Delta t \|\boldsymbol{D}(e_{\boldsymbol{u}}^{h,n+1})\|_{2\mu_f, \Omega_f}^2 + {2\Delta t \|\boldsymbol{P}_f(e_{\boldsymbol{u}}^{h,n+1})\|_{\gamma, \Gamma}^2} + {2\Delta t \|e_{\boldsymbol{u}}^{h,n+1} \cdot \boldsymbol{n}_f \|_{L, \Gamma}^2 } 
	\\
	&+\|e_{\boldsymbol{\xi}}^{h,n+1}\|_{\rho_p, \Omega_p}^2 - \|e_{\boldsymbol{\xi}}^{h,n}\|_{\rho_p, \Omega_p}^2 + \|e_{\boldsymbol{\xi}}^{h,n+1}-e_{\boldsymbol{\xi}}^{h,n}\|_{\rho_p, \Omega_p}^2
	\\
	&+\|\boldsymbol{D}(e_{\boldsymbol{\eta}}^{h,n+1})\|_{2\mu_p, \Omega_p}^2 - \|\boldsymbol{D}(e_{\boldsymbol{\eta}}^{h,n})\|_{2\mu_p, \Omega_p}^2 + \|\boldsymbol{D}(e_{\boldsymbol{\eta}}^{h,n+1})-\boldsymbol{D}(e_{\boldsymbol{\eta}}^{h,n})\|_{2\mu_p, \Omega_p}^2
	\\
	&+\|\nabla \cdot e_{\boldsymbol{\eta}}^{h,n+1}\|_{\lambda_p, \Omega_p}^2 - \|\nabla \cdot e_{\boldsymbol{\eta}}^{h,n}\|_{\lambda_p, \Omega_p}^2 + \|\nabla \cdot (e_{\boldsymbol{\eta}}^{h,n+1}-e_{\boldsymbol{\eta}}^{h,n})\|_{\lambda_p, \Omega_p}^2
	\\
	&+\|e_{\phi}^{h,n+1}\|_{C_0, \Omega_p}^2 - \|e_{\phi}^{h,n}\|_{C_0, \Omega_p}^2 + \|e_{\phi}^{h,n+1}-e_{\phi}^{h,n}\|_{C_0, \Omega_p}^2
	\\
	&+2\Delta t (\|\nabla e_{\phi}^{h,n+1}\|_{K, \Omega_p}^2 +  {\|\boldsymbol{P}_p (e_{\boldsymbol{\xi}}^{h,n+1})\|_{\gamma, \Gamma}^2} +  {\|e_{\boldsymbol{\xi}}^{h,n+1} \cdot \boldsymbol{n}_p\|_{\Gamma}^2} +  {\|e_{\phi}^{h,n+1}\|_{L^{-1}, \Gamma}^2 } )
	\\
	&= 2\Delta t (\langle \gamma P_{f}(\boldsymbol{\xi}^{n+1} - \boldsymbol{\xi}^{n} +  \underbrace{e_{\boldsymbol{\xi}}^{h,n}} - e_{\boldsymbol{\xi}}^{I,n}), P_{f}(e_{\boldsymbol{u}}^{h,n+1})\rangle_{\Gamma} )
    \\
    &+ 2\Delta t ( \langle L (\boldsymbol{u}^{n+1}-\boldsymbol{u}^{n} +  \underbrace{e_{\boldsymbol{u}}^{h,n}} - e_{\boldsymbol{u}}^{I,n})\cdot \boldsymbol{n}_{f} - (\phi^{n+1} - \phi^n +  {e_{\phi}^{h,n}} -e_{\phi}^{I,n} ), e_{\boldsymbol{u}}^{h,n+1}\cdot \boldsymbol{n}_{f}\rangle_{\Gamma} )
	\\
	& + 2\Delta t \rho_{f}\left(\frac{\Pi_{\boldsymbol{u}} \boldsymbol{u}^{n+1} - \Pi_{\boldsymbol{u}} \boldsymbol{u}^{n}}{\Delta t} - \partial_t \boldsymbol{u}^{n+1}, e_{\boldsymbol{u}}^{h,n+1}\right)_{f}
	\\
	& + 2\Delta t (- \alpha(e_{\phi}^{I,n+1}, \nabla\cdot e_{\boldsymbol{\xi}}^{h,n+1})_{p} + \alpha(\nabla\cdot e_{\boldsymbol{\xi}}^{I,n+1}, e_{\phi}^{h,n+1})_{p} )
\\
	& + 2 \Delta t \gamma \langle \boldsymbol{P}_p(e_{\boldsymbol{\xi}}^{I,n+1}), \boldsymbol{P}_p(e_{\boldsymbol{\xi}}^{h,n+1}) \rangle_{\Gamma}
	\\
	& + 2\Delta t (\langle e_{\boldsymbol{\xi}}^{I,n+1}\cdot \boldsymbol{n}_{p}, e_{\boldsymbol{\xi}}^{h,n+1}\cdot \boldsymbol{n}_{p}\rangle_{\Gamma} + \langle e_{\phi}^{I,n+1}, e_{\boldsymbol{\xi}}^{h,n+1}\cdot \boldsymbol{n}_{p}\rangle_{\Gamma} - \langle e_{\boldsymbol{\xi}}^{I,n+1}\cdot \boldsymbol{n}_{p}, e_{\phi}^{h,n+1}\rangle_{\Gamma} )
	\\
	& + 2\Delta t (\langle \gamma P_{p}((\boldsymbol{u}^{n+1}- \boldsymbol{u}^{n}) +  \underbrace{e_{\boldsymbol{u}}^{h,n}} - e_{\boldsymbol{u}}^{I,n}), P_{p}(e_{\boldsymbol{\xi}}^{h,n+1})\rangle_{\Gamma} )
	\\
	& + 2\Delta t (\langle ( \boldsymbol{\xi}^{n+1}- \boldsymbol{\xi}^{n} +  \underbrace{e_{\boldsymbol{\xi}}^{h,n}} - e_{\boldsymbol{\xi}}^{I,n})\cdot \boldsymbol{n}_{p}, e_{\boldsymbol{\xi}}^{h,n+1}\cdot \boldsymbol{n}_{p}\rangle_{\Gamma} )
\end{aligned}
\end{equation}       
\begin{equation*}
\begin{aligned}
	& + 2\Delta t (\langle -(\boldsymbol{u}^{n+1}-\boldsymbol{u}^{n} +  \underbrace{e_{\boldsymbol{u}}^{h,n}} - e_{\boldsymbol{u}}^{I,n})\cdot \boldsymbol{n}_{p} + ({\phi}^{n+1} - \phi^{n} +  \underbrace{e_{\phi}^{h,n}} - e_{\phi}^{I,n} )/L, e_{\phi}^{h,n+1}\rangle_{\Gamma} )
	\\
	& + 2\Delta t \rho_{p}\left(\frac{\Pi_{ \boldsymbol{\eta}} \boldsymbol{\xi}^{n+1} - \Pi_{ \boldsymbol{\eta}} \boldsymbol{\xi}^{n}}{\Delta t} - \partial_t \boldsymbol{\xi}^{n+1} , e_{\boldsymbol{\xi}}^{h, n+1}\right)_{p}\\
	& + 2\Delta t C_{0}\left(\frac{\Pi_{\phi} {\phi}^{n+1} - \Pi_{\phi} {\phi}^{n}}{\Delta t} - \partial_t \phi^{n+1}, e_{\phi}^{h,n+1}\right)_{p}.
\end{aligned}
\end{equation*}
{\sunny{Next, we collect all the boundary terms on the right hand-side 
involving $e_{\boldsymbol{u}}^{h, n}, e_{\boldsymbol{\xi}}^{h, n}$, and $e_\phi^{h, n}$ (underlined by curly brackets)
and apply the weighted polarization equality
\begin{equation}\label{eq:polarization-weighted}
2\langle a,b\rangle_{\Gamma,w}
=\|a\|_{\Gamma,w}^2+\|b\|_{\Gamma,w}^2-\|a-b\|_{\Gamma,w}^2
\end{equation}
with weights $w=\gamma$, $w=L$, and $w=L^{-1}$, where 
\[
\langle a,b\rangle_{\Gamma,w} := \int_{\Gamma} w\, a\, b \,ds
\]
is the weighted inner product. We move the resulting expressions to the left hand-side of the equations. 
Similarly, we collect all the boundary terms on the left hand-side involving $e_{\boldsymbol{u}}^{h, n+1}, e_{\boldsymbol{\xi}}^{h, n+1}$, and $e_\phi^{h, n+1}$ and use the weighted polarization equality \eqref{eq:polarization-weighted}.  
We combine the alike terms to obtain:
}}


%
\begin{itemize}
\item The terms involving tangential project of errors $e_{\boldsymbol{u}}$ and $e_{\boldsymbol{\xi}}$:
\begin{equation*}
\begin{aligned}
&\Delta t \|\boldsymbol{P}_f (e_{\boldsymbol{u}}^{h,n+1})\|_{\gamma,\Gamma}^2 + \Delta t \|\boldsymbol{P}_f (e_{\boldsymbol{u}}^{h,n+1}) - \boldsymbol{P}_p(e_{\boldsymbol{\xi}}^{h,n})\|_{\gamma,\Gamma}^2 - \Delta t \|\boldsymbol{P}_p (e_{\boldsymbol{\xi}}^{h,n})\|_{\gamma,\Gamma}^2 
\\
& + \Delta t \|\boldsymbol{P}_p (e_{\boldsymbol{\xi}}^{h,n+1})\|_{\gamma,\Gamma}^2 + \Delta t \|\boldsymbol{P}_p (e_{\boldsymbol{\xi}}^{h,n+1}) - \boldsymbol{P}_f(e_{\boldsymbol{u}}^{h,n})\|_{\gamma,\Gamma}^2 - \Delta t \|\boldsymbol{P}_f (e_{\boldsymbol{u}}^{h,n})\|_{\gamma,\Gamma}^2  
\end{aligned}
\end{equation*}
\item The terms involving error $e_{\boldsymbol{u}}$:
$$\Delta t \|e_{\boldsymbol{u}}^{h,n+1} \cdot \boldsymbol{n}_f\|_{L,\Gamma}^2 + \Delta t \|(e_{\boldsymbol{u}}^{h,n+1} - e_{\boldsymbol{u}}^{h,n}) \cdot \boldsymbol{n}_f\|_{L,\Gamma}^2 - \Delta t \|e_{\boldsymbol{u}}^{h,n} \cdot \boldsymbol{n}_f\|_{L,\Gamma}^2, 
$$
\item The terms involving error $e_{\boldsymbol{\xi}}$:
$$\Delta t \|e_{\boldsymbol{\xi}}^{h,n+1} \cdot \boldsymbol{n}_p\|_{\Gamma}^2 + \Delta t \|(e_{\boldsymbol{\xi}}^{h,n+1} - e_{\boldsymbol{\xi}}^{h,n}) \cdot \boldsymbol{n}_p\|_{\Gamma}^2 - \Delta t \|e_{\boldsymbol{\xi}}^{h,n} \cdot \boldsymbol{n}_p\|_{\Gamma}^2, 
$$
\item The terms involving error $e_{\phi}$:
$$\Delta t \|e_{\phi}^{h,n+1} \|_{L^{-1},\Gamma}^2 + \Delta t \|(e_{\phi}^{h,n+1} - e_{\phi}^{h,n}) \|_{L^{-1},\Gamma}^2 - \Delta t \|e_{\phi}^{h,n} \|_{L^{-1},\Gamma}^2.
$$
\end{itemize}

{\sunny{We also express the mixed interface term}}
$$
\begin{aligned}
&~2 \Delta t \langle e_{\phi}^{h,n}, e_{\boldsymbol{u}}^{h,n+1} \cdot \boldsymbol{n}_f \rangle_{\Gamma} + 2 \Delta t \langle e_{\boldsymbol{u}}^{h,n} \cdot \boldsymbol{n}_p, e_{\phi}^{h,n+1} \rangle_{\Gamma} 
\\
	=& ~2 \Delta t \langle e_{\phi}^{h,n}, (e_{\boldsymbol{u}}^{h,n+1} - e_{\boldsymbol{u}}^{h,n}) \cdot \boldsymbol{n}_f \rangle_{\Gamma} + 2 \Delta t \langle e_{\boldsymbol{u}}^{h,n} \cdot \boldsymbol{n}_p, e_{\phi}^{h,n+1} - e_{\phi}^{h,n} \rangle_{\Gamma} 
\\
	=&~ 2 \Delta t \langle e_{\phi}^{h,n} - e_{\phi}^{h,n+1}, e_{\boldsymbol{u}}^{h,n+1} \cdot \boldsymbol{n}_f \rangle_{\Gamma} + 2 \Delta t \langle (e_{\boldsymbol{u}}^{h,n} - e_{\boldsymbol{u}}^{h,n+1}) \cdot \boldsymbol{n}_p, e_{\phi}^{h,n+1} \rangle_{\Gamma}.    
\end{aligned}
$$
{\sunny{With these calculations, equation \eqref{eq:error_energy} can be rewritten as:
}}
\begin{equation}\label{eq:energy-compact}
   {\fbox{ 
   $
   X_{n+1}^2-X_n^2+Y_{n+1}^2+Z_{n+1}=F_{\boldsymbol{u}}^{n+1}\left(e_{\boldsymbol{u}}^{h, n+1}\right)+F_{\boldsymbol{\xi}}^{n+1}\left(e_{\boldsymbol{\xi}}^{h, n+1}\right)+F_\phi^{n+1}\left(e_\phi^{h, n+1}\right),
   $
   }}
\end{equation}
where:
\vskip 0.1in
\begin{itemize}[leftmargin=12pt,labelsep=0.5em]
\item {\bf{$X_n^2$ denotes the kinetic and elastic error energy at step $n$}}
plus the interface terms multiplied by $\Delta t$:
\begin{align}
    \notag
    X_n^2:= & {\small{
    \left\|e_{\boldsymbol{u}}^{h, n}\right\|_{\rho_f, \Omega_f}^2+\left\|e_{\boldsymbol{\xi}}^{h, n}\right\|_{\rho_p, \Omega_p}^2+\left\|\boldsymbol{D}\left(e_{\boldsymbol{\eta}}^{h, n}\right)\right\|_{2 \mu_p, \Omega_p}^2+\left\|\nabla \cdot e_{\boldsymbol{\eta}}^{h, n}\right\|_{\lambda_p, \Omega_p}^2+\left\|e_\phi^{h, n}\right\|_{C_0, \Omega_p}^2
    }}
    \\
    \label{eq:Xn-definition}
    & +\Delta t\left\|\boldsymbol{P}_f\left(e_{\boldsymbol{u}}^{h, n}\right)\right\|_{\gamma, \Gamma}^2+\Delta t\left\|\boldsymbol{P}_p\left(e_{\boldsymbol{\xi}}^{h, n}\right)\right\|_{\gamma, \Gamma}^2+\Delta t\left\|e_{\boldsymbol{u}}^{h, n} \cdot \boldsymbol{n}_f\right\|_{L, \Gamma}^2
    \\
    \notag
    & +\Delta t\left\|e_{\boldsymbol{\xi}}^{h, n} \cdot \boldsymbol{n}_p\right\|_{\Gamma}^2 +\Delta t\left\|e_\phi^{h, n}\right\|_{L^{-1}, \Gamma}^2,
\end{align}
\item {\bf{$Y_n^2$ is the sum of all increments plus dissipation}}:
\begin{align}
    \notag
    Y_n^2:= & {\small{\left\|e_{\boldsymbol{u}}^{h, n}-e_{\boldsymbol{u}}^{h, n-1}\right\|_{\rho_f, \Omega_f}^2+\left\|e_{\boldsymbol{\xi}}^{h, n}-e_{\boldsymbol{\xi}}^{h, n-1}\right\|_{\rho_p, \Omega_p}^2+\left\|\boldsymbol{D}\left(e_{\boldsymbol{\eta}}^{h, n}\right)-\boldsymbol{D}\left(e_{\boldsymbol{\eta}}^{h, n-1}\right)\right\|_{2 \mu_p, \Omega_p}^2
    }}
    \\
    \label{eq:Yn-definition}
    & +\left\|\nabla \cdot\left(e_{\boldsymbol{\eta}}^{h, n}-e_{\boldsymbol{\eta}}^{h, n-1}\right)\right\|_{\lambda_p, \Omega_p}^2+\left\|e_\phi^{h, n}-e_\phi^{h, n-1}\right\|_{C_0, \Omega_p}^2 
    \\
    \notag
    & +2 \Delta t\left\|\boldsymbol{D}\left(e_{\boldsymbol{u}}^{h, n}\right)\right\|_{2 \mu_f, \Omega_f}^2+2 \Delta t\left\|\nabla e_\phi^{h, n}\right\|_{K, \Omega_p}^2 
    \\
    \notag
    & +\Delta t\left\|\boldsymbol{P}_f\left(e_{\boldsymbol{u}}^{h, n}\right)-\boldsymbol{P}_p\left(e_{\boldsymbol{\xi}}^{h, n-1}\right)\right\|_{\gamma, \Gamma}^2+\Delta t\left\|\boldsymbol{P}_p\left(e_{\boldsymbol{\xi}}^{h, n}\right)-\boldsymbol{P}_f\left(e_{\boldsymbol{u}}^{h, n-1}\right)\right\|_{\gamma, \Gamma}^2 
    \\
    \notag
    & +\Delta t\left\|\left(e_{\boldsymbol{u}}^{h, n}-e_{\boldsymbol{u}}^{h, n-1}\right) \cdot \boldsymbol{n}_f\right\|_{L, \Gamma}^2+\Delta t\left\|\left(e_{\boldsymbol{\xi}}^{h, n}-e_{\boldsymbol{\xi}}^{h, n-1}\right) \cdot \boldsymbol{n}_p\right\|_{\Gamma}^2 
    \\
    \notag
    & +\Delta t\left\|e_\phi^{h, n}-e_\phi^{h, n-1}\right\|_{L^{-1}, \Gamma}^2
\end{align}
\item {\bf{$Z_n$ contains all the mixed terms}}:
$$
Z_n=2 \Delta t\left\langle e_\phi^{h, n-1},\left(e_{\boldsymbol{u}}^{h, n}-e_{\boldsymbol{u}}^{h, n-1}\right) \cdot \boldsymbol{n}_f\right\rangle_{\Gamma}+2 \Delta t\left\langle e_{\boldsymbol{u}}^{h, n-1} \cdot \boldsymbol{n}_p, e_\phi^{h, n}-e_\phi^{h, n-1}\right\rangle_{\Gamma},
$$
\item {\sunny{$F_{\boldsymbol{u}}^{n+1}$, $F_{\boldsymbol{\xi}}^{n+1}$ and $F_\phi^{n+1}$ {\bf{are the residual terms.}} They are  the functionals incorporating the remaining terms that do not represent error energy as they are given in terms of the approximate solutions themselves; they are grouped into three contributions, each from a different subproblem:}}

\begin{equation}
\begin{aligned}
    \label{F_u}
    F_{\boldsymbol{u}}^{n+1}&\left(e_{\boldsymbol{u}}^{h, n+1}\right)=  2 \Delta t\left\langle\gamma \boldsymbol{P}_f\left(\boldsymbol{\xi}^{n+1}-\boldsymbol{\xi}^n-e_{\boldsymbol{\xi}}^{I, n}\right), \boldsymbol{P}_f\left(e_{\boldsymbol{u}}^{h, n+1}\right)\right\rangle_{\Gamma} 
    \\
    & +2 \Delta t\left\langle L\left(\boldsymbol{u}^{n+1}-\boldsymbol{u}^n-e_{\boldsymbol{u}}^{I, n}\right) \cdot \boldsymbol{n}_f-\left(\phi^{n+1}-\phi^n-e_\phi^{I, n}\right), e_{\boldsymbol{u}}^{h, n+1} \cdot \boldsymbol{n}_f\right\rangle_{\Gamma} 
    \\
    & +2 \Delta t \rho_f\left(\frac{\Pi_{\boldsymbol{u}} \boldsymbol{u}^{n+1}-\Pi_{\boldsymbol{u}} \boldsymbol{u}^n}{\Delta t}-\partial_t \boldsymbol{u}^{n+1}, e_{\boldsymbol{u}}^{h, n+1}\right)_f,
\end{aligned}
\end{equation}
\begin{equation}
\begin{aligned}
    \label{F_xi}
    F_{\boldsymbol{\xi}}^{n+1}\left(e_{\boldsymbol{\xi}}^{h, n+1}\right)&=  
    {\small{2 \Delta t\left(-\alpha\left(e_\phi^{I, n+1}, \nabla \cdot e_{\boldsymbol{\xi}}^{h, n+1}\right)_p\right)+2 \Delta t\left(\left\langle e_{\boldsymbol{\xi}}^{I, n+1} \cdot \boldsymbol{n}_p, e_{\boldsymbol{\xi}}^{h, n+1} \cdot \boldsymbol{n}_p\right\rangle_{\Gamma}\right.
    }}
    \\
    &\left.+\left\langle e_\phi^{I, n+1}, e_{\boldsymbol{\xi}}^{h, n+1} \cdot \boldsymbol{n}_p\right\rangle_{\Gamma}\right)  +2 \Delta \gamma\left\langle \boldsymbol{P}_p\left(e_{\boldsymbol{\xi}}^{I, n+1}\right), \boldsymbol{P}_p\left(e_{\boldsymbol{\xi}}^{h, n+1}\right)\right\rangle_{\Gamma} 
    \\
    & +2 \Delta t\left\langle\gamma \boldsymbol{P}_p\left(\boldsymbol{u}^{n+1}-\boldsymbol{u}^n-e_{\boldsymbol{u}}^{I, n}\right), \boldsymbol{P}_p\left(e_{\boldsymbol{\xi}}^{h, n+1}\right)\right\rangle_{\Gamma} 
    \\
    & +2 \Delta t\left\langle\left(\boldsymbol{\xi}^{n+1}-\boldsymbol{\xi}^n-e_{\boldsymbol{\xi}}^{I, n}\right) \cdot \boldsymbol{n}_p, e_{\boldsymbol{\xi}}^{h, n+1} \cdot \boldsymbol{n}_p\right\rangle_{\Gamma} 
    \\
    & +2 \Delta t \rho_p\left(\frac{\Pi_{\boldsymbol{\eta}} \boldsymbol{\xi}^{n+1}-\Pi_{\boldsymbol{\eta}} \boldsymbol{\xi}^n}{\Delta t}-\partial_t \boldsymbol{\xi}^{n+1}, e_{\boldsymbol{\xi}}^{h, n+1}\right)_p,
\end{aligned}
\end{equation}

\begin{equation}
\begin{aligned}
F_\phi^{n+1}
\left(e_\phi^{h, n+1}\right)&=  2 \Delta t\left(\alpha\left(\nabla \cdot e_{\boldsymbol{\xi}}^{I, n+1}, e_\phi^{h, n+1}\right)_p-\left\langle e_{\boldsymbol{\xi}}^{I, n+1} \cdot \boldsymbol{n}_p, e_\phi^{h, n+1}\right\rangle_{\Gamma}\right) \\
& {\small{+2 \Delta t\left\langle-\left(\boldsymbol{u}^{n+1}-\boldsymbol{u}^n-e_{\boldsymbol{u}}^{I, n}\right) \cdot \boldsymbol{n}_p+\left(\phi^{n+1}-\phi^n-e_\phi^{I, n}\right) / L, e_\phi^{h, n+1}\right\rangle_{\Gamma}
}}
\\
& +2 \Delta t C_0\left(\frac{\Pi_\phi \phi^{n+1}-\Pi_\phi {\phi^{n}}}{\Delta t}-\partial_t \phi^{n+1}, e_\phi^{h, n+1}\right)_p.
\end{aligned}
\label{F_phi}
\end{equation}
\end{itemize}

Notice that all the interface contributions are included explicitly and remain nonnegative provided that $\gamma, L>0$.
The term $Z_{n+1}$ collects mixed products between
two consecutive time steps, which will later be bounded by Young’s inequality and absorbed into the left-hand side
of~\eqref{eq:energy-compact}. 

A summation of \eqref{eq:energy-compact} over index $n$ gives  {\sunny{the error energy equality:}}

\begin{equation}
    \label{eq:energy-compact-sum}
{\fbox{ $    X_{n}^2-X_0^2+\sum_{i=1}^n Y_{i}^2 + \sum_{i=1}^n Z_{i}= \sum_{i=1}^n \left[ F_{\boldsymbol{u}}^{i}\left(e_{\boldsymbol{u}}^{h,i}\right)+F_{\boldsymbol{\xi}}^{i}\left(e_{\boldsymbol{\xi}}^{h,i}\right)+F_\phi^{i}\left(e_\phi^{h, i}\right) \right]
$
}}
\end{equation}

for $n \ge 1$.

\section{{\emph{A priori}} Error Estimates}
\label{sec:estimate}
{\sunny{Based on the error energy equality presented above, and using discrete Gronwall's inequality, in this section we derive the desired {\emph{a priori}} energy estimates, from which we will show first-order accuracy in time, and optimal accuracy in space. 

We start by first estimating 
the right-hand side of the energy identity \eqref{eq:energy-compact-sum}, which contains the 
sum of the residual terms $F_{\boldsymbol{u}}^{i}(e_{\boldsymbol{u}}^{h, i})$, $F_{\boldsymbol{\xi}}^{i}(e_{\boldsymbol{\xi}}^{h, i})$, $F_{\phi}^{i}(e_{\phi}^{h, i})$, and show that they are bounded {by data-dependent quantities multiplying
$h^{2k+r} + (\Delta t)^2$ stemming from the spatial and temporal approximations. 
Then we estimate $Z_n$ comprized of  the mixed error terms and show that $Z_n$ is bounded by a quantity depending on the errors, multiplying $\Delta t$. Finally, we use a discrete Gronwall inequality argument} that will yield the final error estimate of the form 
\[
\max_{0\le n\le N} X_n
+ \Big( \sum_{n=1}^N Y_n^2 \Big)^{1/2}
\le
C \big( h^{k+\frac r2} + \Delta t \big).
\]
}}

Throughout this section we assume the exact solution satisfies the
following regularity (recall $\boldsymbol{\xi}=\partial_t\boldsymbol{\eta}$):
\begin{equation}\label{eq:regularity}
\begin{aligned}
    \boldsymbol{u}, \partial_t \boldsymbol{u}   &\in L^2(0,T; H^{k+1}(\Omega_f)^d),  &&\partial_{tt}\boldsymbol{u} \in L^{2}(0,T; {H^1}(\Omega_f)^d),
    \\
    \boldsymbol{\eta}, \boldsymbol{\xi} &\in L^\infty(0,T; H^{k+1}(\Omega_p)^d),  
    &&\partial_t \boldsymbol{\xi} \in L^{2}(0,T; {H^{k+1}}(\Omega_p)^d) \cap L^{\infty}(0,T; H^{k+1}(\Omega_p)^d), \\
    \partial_t^3 \boldsymbol{\eta} &\in L^{\infty}(0,T; H^{1}(\Omega_p)^d),
    &&\phi, \partial_t\phi \in {L^2}(0,T; H^{k+1}(\Omega_p)), \\
    \partial_{tt}\phi &\in L^{2}(0,T; L^2(\Omega_p)),
   \ && {p, \partial_t p~} {\in L^2(0,T; H^{k+1}(\Omega_p))}, 
\end{aligned}
\end{equation}
which is standard in Stokes--Biot analysis
\cite{ambartsumyan2018lagrange,guo2022decoupled, temam2001navier,biot1941general,wen2020strongly}.

\subsection{Estimates of the residual terms}\label{sec:residual_estimates}

\begin{theorem}[Estimates of residual terms]
    \label{thm:residual-estimates}
    For $\varepsilon>0$ let $C_{\varepsilon} = C \varepsilon^{-1}$ with a constant $C>0$ independent of $h$ and $\Delta t$. {\sunny{Then, for every $\varepsilon > 0$, there exists $C_{\varepsilon} > 0$ such that the residual terms $F_{\boldsymbol{u}}^{n+1}\left(e_{\boldsymbol{u}}^{h, n+1}\right)$, $F_{\boldsymbol{\xi}}^{n+1}\left(e_{\boldsymbol{\xi}}^{h, n+1}\right)$, $F_{\phi}^{n+1}\left(e_{\phi}^{h, n+1}\right)$ defined in \eqref{F_u}, \eqref{F_xi}, \eqref{F_phi}, are bounded as follows:}}
\begin{equation}
\begin{aligned}
&\left|\sum_{i=1}^n  F_{\boldsymbol{u}}^i\left(e_{\boldsymbol{u}}^{h, i}\right)\right| \leq  C \varepsilon \Delta t \sum_{i=1}^n\left(\left\|e_{\boldsymbol{u}}^{h, i}\right\|_{\Omega_f}^2+\left\|\boldsymbol{D}\left(e_{\boldsymbol{u}}^{h, i}\right)\right\|_{\Omega_f}^2\right) \\
& +C_{\varepsilon}(\Delta t)^2\left(\left\|\partial_{t t} \boldsymbol{u}\right\|_{L^2\left(0, t_n ; L^2\left(\Omega_f\right)\right)}^2
+\left\|\partial_t \boldsymbol{u}\right\|_{L^2\left(0, t_n ; H^1\left(\Omega_f\right)\right)}^2
+\left\|\partial_t \boldsymbol{\xi}\right\|_{L^2\left(0, t_n ; H^1\left(\Omega_p\right)\right)}^2\right.
\\
&\qquad \qquad \quad
\left.+\left\|\partial_t \phi\right\|_{L^2\left(0, t_n ; H^1\left(\Omega_p\right)\right)}^2\right) \\
& +C_{\varepsilon} \Delta t h^{2(k+r)} \sum_{i=1}^n\left(\left\|\partial_t \boldsymbol{u}^i\right\|_{H^{k+1}\left(\Omega_f\right)}^2+\left\|\partial_t p^i\right\|_{H^k\left(\Omega_f\right)}^2\right) \\
& +C_{\varepsilon} \Delta t h^{2 k+r} \sum_{j=0}^{n-1}\left(\left\|\boldsymbol{\xi}^j\right\|_{H^{k+1}\left(\Omega_p\right)}^2+\left\|\phi^j\right\|_{H^{k+1}\left(\Omega_p\right)}^2+\left\|\boldsymbol{u}^j\right\|_{H^{k+1}\left(\Omega_f\right)}^2+\left\|p^j\right\|_{H^k\left(\Omega_f\right)}^2\right) ,
    \label{eq:F_u-estimate-simplified}
\end{aligned}
\end{equation}
\begin{equation}
    \begin{aligned}
&\left|\sum_{i=1}^n F_{\boldsymbol{\xi}}^i\left(e_{\boldsymbol{\xi}}^{h, i}\right)\right| \leq  C \varepsilon \Delta t \sum_{i=1}^{n-1}\left\|\boldsymbol{D}\left(e_{\boldsymbol{\eta}}^{h, i}\right)\right\|_{\Omega_p}^2
+\varepsilon\left(\left\|\boldsymbol{D}\left(e_{\boldsymbol{\eta}}^{h, n}\right)\right\|_{\Omega_p}^2
+\left\|\boldsymbol{D}\left(e_{\boldsymbol{\eta}}^{h, 0}\right)\right\|_{\Omega_p}^2\right) \\
& +C_{\varepsilon}(\Delta t)^2\left[\left\|\partial_t^2 \boldsymbol{u}, \partial_t^3 \boldsymbol{\eta}\right\|_{L^2\left(t_1, t_n ; L^2(\Gamma)\right)}^2
+\left\|\partial_t^3 \boldsymbol{\eta}\right\|_{L^2\left(0, t_n ; L^2\left(\Omega_p\right)\right)}^2 \right.
\\
& \qquad \qquad \quad
\left. +\left\|\partial_t \boldsymbol{u}, \partial_t^2 \boldsymbol{\eta}, \partial_t^3 \boldsymbol{\eta}\right\|_{L^{\infty}\left(t_{n-1}, t_n ; L^2(\Gamma)\right)}^2\right] 
\\
& +C_{\varepsilon} h^{2 k+r}\left[\left\|\partial_t \boldsymbol{u}\right\|_{L^2\left(0, t_{n-1} ; H^{k+1}\left(\Omega_f\right)\right)}^2+\left\|\boldsymbol{u}^{n-1}, \boldsymbol{u}^0\right\|_{H^{k+1}\left(\Omega_f\right)}^2\right. \\
& \qquad \qquad \quad +\left\|\partial_t \boldsymbol{\eta}, \partial_t^2 \boldsymbol{\eta}, \partial_t \phi\right\|_{L^2\left(0, t_n ; H^{k+1}\left(\Omega_p\right)\right)}^2+\left\|\boldsymbol{\xi}^{n-1}, \boldsymbol{\xi}^0, \phi^1, \phi^n\right\|_{H^{k+1}\left(\Omega_p\right)}^2 \\
& \qquad \qquad \quad \left.+\left\|\partial_t \boldsymbol{\eta}\right\|_{L^{\infty}\left(0, t_1 ; H^{k+1}\left(\Omega_p\right)\right)}^2+\left\|\partial_t \boldsymbol{\eta}\right\|_{L^{\infty}\left(t_{n-1}, t_n ; H^{k+1}\left(\Omega_p\right)\right)}^2\right]
\end{aligned}
\label{eq:F_xi-estimate}
\end{equation}
\begin{equation}
\begin{aligned}
 &\left|\sum_{i=1}^n F_\phi^i\left(e_\phi^{h, i}\right) \right| \leq  C \varepsilon \Delta t \sum_{i=1}^n\left(\left\|e_\phi^{h, i}\right\|_{\Omega_p}^2+\left\|\nabla e_\phi^{h, i}\right\|_{\Omega_p}^2\right) \\
& +C_{\varepsilon}(\Delta t)^2 {\left(\left\|\partial_t \phi\right\|_{L^2\left(0, t_n ; H^1\left(\Omega_p\right)\right)}^2+\left\|\partial_{t t} \phi\right\|_{L^2\left(0, t_n ; L^2\left(\Omega_p\right)\right)}^2+\left\|\partial_t \boldsymbol{u}\right\|_{L^2\left(0, t_n ; H^1\left(\Omega_f\right)\right)}^2\right)} \\
& +C_{\varepsilon} \Delta t h^{2 k+r} \sum_{i=1}^n\left(\left\|\boldsymbol{\xi}^i, \phi^{i-1}, \partial_t \phi^i\right\|_{H^{k+1}\left(\Omega_p\right)}^2+\left\|\boldsymbol{u}^{i-1}\right\|_{H^{k+1}\left(\Omega_f\right)}^2+\left\|p^{i-1}\right\|_{H^k\left(\Omega_f\right)}^2\right) .
\end{aligned}
\label{eq:F_phi-estimate}
\end{equation}
\end{theorem}

{\sunny{We remark that in the final error estimate,
the terms with the coefficient  $C\Delta t \varepsilon$ will be absorbed into the LHS. Furthermore, we note that the terms with the coefficient $C_{\varepsilon} (\Delta t)^2$ correspond to the time-truncation error, and the terms with the coefficient $C_{\varepsilon}(\Delta t) h^{2 k+r}$ correspond to the spatial truncation error.}}

\begin{proof}

We estimate each of the three terms separately as follows.

\vskip 0.1in
\noindent
{\bf The estimate of $\sum_{i=1}^n F_{\boldsymbol{u}}^{i}\left(e_{\boldsymbol{u}}^{h, i}\right)$}: 

{\sunny{Express $F_{\boldsymbol{u}}^{i}\left(e_{\boldsymbol{u}}^{h, i}\right)$ as the sum of three terms:}}  $F_{\boldsymbol{u}}^{i}\left(e_{\boldsymbol{u}}^{h, i}\right) = T_{u,1}^{i} + T_{u,2}^{i} + T_{u,3}^{i}$ where:
\begin{align*}
    T_{u,1}^{i}&:=2 \Delta t\left\langle\gamma \boldsymbol{P}_f\left(\boldsymbol{\xi}^{i}-\boldsymbol{\xi}^{i-1}-e_{\boldsymbol{\xi}}^{I, i-1}\right), \boldsymbol{P}_f e_{\boldsymbol{u}}^{h, i}\right\rangle_{\Gamma},
    \\
    T_{u,2}^{i}&:=2 \Delta t\left\langle L\left(\boldsymbol{u}^{i}-\boldsymbol{u}^{i-1}-e_{\boldsymbol{u}}^{I, i-1}\right) \cdot \boldsymbol{n}_f-\left(\phi^{i}-\phi^{i-1}-e_\phi^{I, i-1}\right), e_{\boldsymbol{u}}^{h, i} \cdot \boldsymbol{n}_f\right\rangle_{\Gamma},
    \\
    T_{u,3}^{i}&:=\Delta t\rho_f\left(\frac{\Pi_{\boldsymbol{u}} \boldsymbol{u}^{i}-\Pi_{\boldsymbol{u}} \boldsymbol{u}^{i-1}}{\Delta t}-\partial_t \boldsymbol{u}^{i}, e_{\boldsymbol{u}}^{h, i}\right)_f .
\end{align*}
{\sunny{{\bf{Term $T_{u,1}^{i}$:}}}} Applying the Cauchy--Schwarz inequality, the inequality \eqref{eq:vector-trace}, and a generalized Young's inequality to $T_{u,1}^i$, we get:
\begin{align*}
    \left|T_{u,1}^{i}\right| & \leq 2 \Delta t \gamma \left\|\boldsymbol{P}_f\left(\boldsymbol{\xi}^{i}-\boldsymbol{\xi}^{i-1}-e_{\boldsymbol{\xi}}^{I, i-1}\right)\right\|_{\Gamma}\left\|\boldsymbol{P}_f e_{\boldsymbol{u}}^{h, i}\right\|_{\Gamma} 
    \\
    & \leq C \Delta t \gamma \left\|\boldsymbol{P}_f\left(\boldsymbol{\xi}^{i}-\boldsymbol{\xi}^{i-1}-e_{\boldsymbol{\xi}}^{I, i-1}\right)\right\|_{\Gamma}\left\|e_{\boldsymbol{u}}^{h, i}\right\|_{\Omega_f}^{\frac 12} \left\|\boldsymbol{D}(e_{\boldsymbol{u}}^{h, i})\right\|_{\Omega_f}^{\frac 12} 
    \\
    & \leq \Delta t \varepsilon\left\|e_{\boldsymbol{u}}^{h, i}\right\|_{\Omega_f}^2 + \Delta t \varepsilon \left\|\boldsymbol{D}(e_{\boldsymbol{u}}^{h, i})\right\|_{\Omega_f}^2 +  C\varepsilon^{-1} \gamma^2 \Delta t \left\|\boldsymbol{P}_f\left(\boldsymbol{\xi}^{i}-\boldsymbol{\xi}^{i-1}-e_{\boldsymbol{\xi}}^{I, i-1}\right)\right\|_{\Gamma}^2 .
\end{align*}
Using the triangle inequality and \eqref{eq:trace}, we get:
$$
\begin{aligned}\left\|\boldsymbol{P}_f\left(\boldsymbol{\xi}^i-\boldsymbol{\xi}^{i-1}\right. \right.
& \left. \left. -\boldsymbol{e}_{\boldsymbol{\xi}}^{I, i-1}\right)\right\|_{\Gamma} 
\leq\left\|\boldsymbol{\xi}^i-\boldsymbol{\xi}^{i-1}-\boldsymbol{e}_{\boldsymbol{\xi}}^{I, i-1}\right\|_{\Gamma} \\ & \leq\left\|\boldsymbol{\xi}^i-\boldsymbol{\xi}^{i-1}\right\|_{L^2(\Gamma)}+\left\|\boldsymbol{e}_{\boldsymbol{\xi}}^{I, i-1}\right\|_{L^2(\Gamma)} \\ & \leq C\left\|\boldsymbol{\xi}^i-\boldsymbol{\xi}^{i-1}\right\|_{H^1\left(\Omega_p\right)}+C\left\|\boldsymbol{e}_{\boldsymbol{\xi}}^{I, i-1}\right\|_{H^1\left(\Omega_p\right)}^{\frac 12}\left\|\boldsymbol{e}_{\boldsymbol{\xi}}^{I, i-1}\right\|_{L^2\left(\Omega_p\right)}^{\frac 12} \\ & \leq C\left(\left\|\boldsymbol{\xi}^i-\boldsymbol{\xi}^{i-1}\right\|_{H^1\left(\Omega_p\right)}+\left\|\boldsymbol{e}_{\boldsymbol{\xi}}^{I, i-1}\right\|_{H^1\left(\Omega_p\right)}^{\frac 12}\left\|\boldsymbol{e}_{\boldsymbol{\xi}}^{I, i-1}\right\|_{L^2\left(\Omega_p\right)}^{\frac 12}\right) .\end{aligned}
$$
Given that
\begin{subequations}
    \label{eq:diff-intp-bdy-estmate-eqs}
\begin{align}
    \label{eq:diff-intp-bdy-estmate-eq1}
    \left\|\boldsymbol{\xi}^{i}-\boldsymbol{\xi}^{i-1}\right\|_{H^1(\Omega_p)} &\leq \left\|\partial_t \boldsymbol{\xi}\right\|_{L^{1}\left(t_{i-1}, t_{i} ; H^1(\Omega_p)\right)} \leq (\Delta t)^{\frac 12} \left\|\partial_t \boldsymbol{\xi}\right\|_{L^{2}\left(t_{i-1}, t_{i} ; H^1(\Omega_p)\right)} ,
    \\
    \label{eq:diff-intp-bdy-estmate-eq2}
    \left\|e_{\boldsymbol{\xi}}^{I, {i-1}}\right\|_{H^1\left(\Omega_p\right)} &\leq  Ch^{k}\left\|\boldsymbol{\xi}^{i-1}\right\|_{ H^{k+1}\left(\Omega_p\right)},    
    \\
    \label{eq:diff-intp-bdy-estmate-eq3}
    \left\|e_{\boldsymbol{\xi}}^{I, {i-1}}\right\|_{L^2\left(\Omega_p\right)} &\leq  Ch^{k+r}\left\| \boldsymbol{\xi}^{i-1}\right\|_{H^{k+1}\left(\Omega_p\right)},    
\end{align}
\end{subequations}
%
we obtain:
\begin{multline}
    \label{eq:diff-intp-bdy-estimate}
    \left\|\boldsymbol{P}_f\left(\boldsymbol{\xi}^{i}-\boldsymbol{\xi}^{i-1}-e_{\boldsymbol{\xi}}^{I, {i-1}}\right)\right\|_{\Gamma} 
    \\
    \leq  C\left((\Delta t)^{\frac 12}\left\|\partial_t \boldsymbol{\xi}\right\|_{L^{2}\left(t_{i-1}, t_{i} ; H^1(\Omega_p)\right)}+ h^{k+\frac r2}\left\| \boldsymbol{\xi}^{i-1}\right\|_{ H^{k+1}\left(\Omega_p\right)}\right) . 
\end{multline}
Therefore,
\begin{align}
    \notag
    \left|T_{u,1}^{i}\right| &\leq  \Delta t \varepsilon\left\|e_{\boldsymbol{u}}^{h, i}\right\|_{\Omega_f}^2 + \Delta t \varepsilon \left\|\boldsymbol{D}(e_{\boldsymbol{u}}^{h, i})\right\|_{\Omega_f}^2
    \\
    \label{eq:Tu1-estimate}
    &\quad +C \varepsilon^{-1} \gamma^2 \Delta t \left( \Delta t \left\| \partial_t \boldsymbol{\xi}\right\|_{L^{2}\left(t_{i-1}, t_{i} ; H^1(\Omega_p)\right)}^2+ h^{2k+r} \left\|\boldsymbol{\xi}^{i-1} \right\|_{H^{k+1}\left(\Omega_p\right)}^2\right) .
\end{align}

\noindent
{\sunny{ {\bf{Term $T_{u,2}^{i}$:}}}} {\sunny{We split this term}} it into two terms
\begin{align*}
    T_{u,2a}^{i}&:=2 \Delta t\left\langle L\left({\boldsymbol{u}}^{i}-{\boldsymbol{u}}^{i-1}-e_{\boldsymbol{u}}^{I, {i-1}}\right) \cdot {\boldsymbol{n}}_f, e_{\boldsymbol{u}}^{h, i} \cdot {\boldsymbol{n}}_f\right\rangle_{\Gamma},
    \\
    T_{u,2b}^{i}&:=-2 \Delta t\left\langle\phi^{i}-\phi^{i-1}-e_\phi^{I, {i-1}}, e_{\boldsymbol{u}}^{h, i} \cdot {\boldsymbol{n}}_f\right\rangle_{\Gamma} .
\end{align*}
Again, by  the Cauchy--Schwarz inequality, \eqref{eq:vector-trace}, and a generalized Young's inequality we {\sunny{obtain}}:
\begin{align*}
    \left|T_{u,2a}^{i}\right| 
    & {\small{\leq \Delta t \varepsilon\left\|e_{\boldsymbol{u}}^{h, i}\right\|_{\Omega_f}^2 + \Delta t \varepsilon \left\|\boldsymbol{D}(e_{\boldsymbol{u}}^{h, i})\right\|_{\Omega_f}^2 + C \varepsilon^{-1}\Delta t L^2 \left\|\left(\boldsymbol{u}^{i}-\boldsymbol{u}^{i-1}-e_{\boldsymbol{u}}^{I, {i-1}}\right) \cdot \boldsymbol{n}_f\right\|_{\Gamma}^2.}}
\end{align*}
Regarding the third term on the RHS, we can again use the argument with \eqref{eq:diff-intp-bdy-estmate-eqs} to get \eqref{eq:diff-intp-bdy-estimate}. Hence, by \eqref{eq:proj-error-fluid} and \eqref{eq:proj-L2-error-fluid}, 
\begin{align}
    \notag
    \left|T_{u,2a}^{i}\right| &\leq \Delta t \varepsilon\left\|e_{\boldsymbol{u}}^{h, i}\right\|_{\Omega_f}^2 + \Delta t \varepsilon \left\|\boldsymbol{D}(e_{\boldsymbol{u}}^{h, i})\right\|_{\Omega_f}^2 + C \varepsilon^{-1} (\Delta t)^2  \left\|\partial_t \boldsymbol{u}\right\|^2_{L^{2}\left(t_{i-1}, t_{i} ; H^1(\Omega_f)\right)} 
    \\    
    \label{eq:Tu2a-estimate}
    &\quad + C \varepsilon^{-1} \Delta t h^{2k+r} \left( \left\|\boldsymbol{u}^{i-1}\right\|^2_{H^{k+1}\left(\Omega_f\right)} + \left\|p^{i-1} \right\|^2_{H^{k}\left(\Omega_f\right)} \right).
\end{align}
{\sunny{Using}} a similar argument for the term $\left|T_{u,2b}\right|$, we have:
\begin{align}
    \notag
    \left|T_{u,2b}^{i}\right| & \leq 2 \Delta t\left\|\phi^{i}-\phi^{i-1}-e_\phi^{I, {i-1}}\right\|_{\Gamma}\left\|e_{\boldsymbol{u}}^{h, i} \cdot \boldsymbol{n}_f\right\|_{\Gamma} 
    \\
    \notag
    & \leq \Delta t \varepsilon\left\|e_{\boldsymbol{u}}^{h, i}\right\|_{\Omega_f}^2 + \Delta t \varepsilon \left\|\boldsymbol{D}(e_{\boldsymbol{u}}^{h, i})\right\|_{\Omega_f}^2 + C \Delta t \varepsilon^{-1}\left\|\phi^{i}-\phi^{i-1}-e_\phi^{I, {i-1}}\right\|_{\Gamma}^2 
    \\
    \notag
    &\leq \Delta t \varepsilon\left\|e_{\boldsymbol{u}}^{h, i}\right\|_{\Omega_f}^2 + \Delta t \varepsilon \left\|\boldsymbol{D}(e_{\boldsymbol{u}}^{h, i})\right\|_{\Omega_f}^2
    \\
    \label{eq:Tu2b-estimate}
    &\quad + C \Delta t \varepsilon^{-1}\left(\Delta t \left\|\partial_t \phi\right\|^2_{L^{2}\left(t_{i-1}, t_{i} ; H^{1}(\Omega_p)\right)} +h^{2k+r}\left\|\phi^{i-1}\right\|^2_{H^{k+1}\left(\Omega_p\right)}\right) .
\end{align}
%
%

\noindent
{\sunny{{\bf{Term $T_{u,3}^{i}$:}}  We rewrite this term using  the following:}}
$$
\frac{\Pi_u \boldsymbol{u}^{i}-\Pi_u \boldsymbol{u}^{i-1}}{\Delta t}-\partial_t \boldsymbol{u}^{i}=\left(\Pi_u\left(\partial_t \boldsymbol{u}^{i}\right)-\partial_t \boldsymbol{u}^{i}\right)+\left(\frac{\boldsymbol{u}^{i}-\boldsymbol{u}^{i-1}}{\Delta t}-\partial_t \boldsymbol{u}^{i}\right) .
$$
By \eqref{eq:proj-L2-error-fluid}, the projection error term can be bounded:
\begin{align*}
    \left\|\Pi_u \partial_t \boldsymbol{u}^{i}-\partial_t \boldsymbol{u}^{i}\right\|_{\Omega_f} \leq  C h^{k+r} ( \|\partial_t \boldsymbol{u}^{i}\|_{H^{k+1}\left(\Omega_f\right)} + \|\partial_t p^{i}\|_{H^{k} \left(\Omega_f\right)}) .
\end{align*}
By Taylor expansion in time and the regularity
\eqref{eq:regularity}, the truncation error satisfies:
\begin{align*}
    \left\|\frac{\boldsymbol{u}^{i}-\boldsymbol{u}^{i-1}}{\Delta t}-\partial_t \boldsymbol{u}^{i}\right\|_{\Omega_f} &\leq C \left\|\partial_{t t} \boldsymbol{u}\right\|_{L^{1}\left(t_{i-1}, t_{i} ; L^2\left(\Omega_f\right)\right)} 
    \\
    &\leq C (\Delta t)^{\frac 12} \left\|\partial_{t t} \boldsymbol{u}\right\|_{L^{2}\left(t_{i-1}, t_{i} ; L^2\left(\Omega_f\right)\right)} .    
\end{align*}
Hence,
\begin{align*}
    \left\|
    \frac{\Pi_{\boldsymbol{u}} \boldsymbol{u}^{i} - \Pi_{\boldsymbol{u}} \boldsymbol{u}^{i-1}}{\Delta t}
    - \partial_t \boldsymbol{u}^{i}
    \right\|_{\Omega_f}
    &\le
    C h^{k+r} (\|\partial_t \boldsymbol{u}^{i}\|_{H^{k+1}(\Omega_f)} + \|\partial_t p^{i}\|_{H^{k}(\Omega_f)})
    \\
    &\quad  + C(\Delta t)^{\frac 12} \|\partial_{tt} \boldsymbol{u}\|_{L^{2}(t_{i-1}, t_{i}; L^2(\Omega_f))}
    . 
\end{align*}
By applying the Cauchy--Schwarz and Young's inequalities, we obtain:
%
\begin{align}
    \label{eq:T_u_3-estimate}
    |T_{u,3}^i|&\le 2 \Delta t \rho_f \left\| \frac{\Pi_{\boldsymbol{u}} \boldsymbol{u}^{i} - \Pi_{\boldsymbol{u}} \boldsymbol{u}^{i-1}}{\Delta t} - \partial_t \boldsymbol{u}^{i} \right\|_{\Omega_f} \left\| {e}_{\boldsymbol{u}}^{h,i} \right\|_{\Omega_f}
    \le \varepsilon \Delta t \|{e}_{\boldsymbol{u}}^{h,i}\|_{\Omega_f}^2 
    \\
    &+ {\small{C \varepsilon^{-1} \Delta t \left(h^{2(k+r)}(\|\partial_t \boldsymbol{u}^{i}\|_{H^{k+1}(\Omega_f)}^2 + \|\partial_t p^{i}\|_{H^{k}(\Omega_f)}^2 ) + \Delta t \|\partial_{tt} \boldsymbol{u}\|_{L^{2}(t_{i-1}, t_{i}; L^2(\Omega_f))}^2 \right).}}
\end{align}
%
{\sunny{We now combine \eqref{eq:Tu1-estimate}, \eqref{eq:Tu2a-estimate}, \eqref{eq:Tu2b-estimate}, and \eqref{eq:T_u_3-estimate}, to obtain}}
\begin{align*}
    &|F_{\boldsymbol{u}}^{i}\left(e_{\boldsymbol{u}}^{h, i}\right)| 
    \le 4 \Delta t \varepsilon\left\|e_{\boldsymbol{u}}^{h, i}\right\|_{\Omega_f}^2 + 3\Delta t \varepsilon \left\|\boldsymbol{D}(e_{\boldsymbol{u}}^{h, i})\right\|_{\Omega_f}^2 
    \\
    &+ C_{\varepsilon} \Delta t \left(h^{2(k+r)}(\|\partial_t \boldsymbol{u}^{i}\|_{H^{k+1}(\Omega_f)}^2 + \|\partial_t p^{i}\|_{H^k(\Omega_f)}^2 ) + \Delta t \|\partial_{tt} \boldsymbol{u}\|_{L^{2}(t_{i-1}, t_{i}; L^2(\Omega_f))}^2 \right)
    \\
    &+ C_{\varepsilon} (\Delta t)^2 \left[  \left\|\partial_t \boldsymbol{u}\right\|_{L^{2}\left(t_{i-1}, t_{i} ; H^{1}(\Omega_f)\right)}^2  + \left\|\partial_t \boldsymbol{\xi}, \partial_t \phi\right\|_{L^{2}\left(t_{i-1}, t_{i} ; H^{1}(\Omega_p)\right)}^2 \right]
    \\
    &+ {\small{C_{\varepsilon} \Delta t h^{2k+r}\left[\left\|\boldsymbol{\xi}^{i-1}, \phi^{i-1}\right\|_{H^{k+1}\left(\Omega_p\right)}^2 + \left\| \boldsymbol{u}^{i-1} \right\|_{H^{k+1}\left(\Omega_f\right)}^2 + \left\|p^{i-1}\right\|_{H^{k}\left(\Omega_f\right)}^2 \right].}}
\end{align*}
Then, {\sunny{the estimate \eqref{eq:F_u-estimate-simplified} in the statement of the theorem}} follows by summing the above estimate for $1\le i \le n$.

\vskip 0.1in
\noindent
{\bf {The estimate of $\sum_{i=1}^n F_{\boldsymbol{\xi}}^{i}\left(e_{\boldsymbol{\xi}}^{h, i}\right)$}: }

{\sunny{Express $F_{\boldsymbol{\xi}}^{i}\left(e_{\boldsymbol{\xi}}^{h, i}\right)$ as the sum of the following four terms:}}
\begin{align*}
    T_{\xi,1}^{i} &:= 2 \Delta t\left(-\alpha\left(e_\phi^{I, i}, \nabla \cdot e_{\boldsymbol{\xi}}^{h, i}\right)_p\right)
    \\
    T_{\xi,2}^{i}&:= 2 \Delta t\left(\left\langle e_{\boldsymbol{\xi}}^{I, i} \cdot \boldsymbol{n}_p, e_{\boldsymbol{\xi}}^{h, i} \cdot \boldsymbol{n}_p\right\rangle_{\Gamma}+\left\langle e_\phi^{I, i}, e_{\boldsymbol{\xi}}^{h, i} \cdot \boldsymbol{n}_p\right\rangle_{\Gamma} +  \gamma\left\langle \boldsymbol{P}_p\left(e_{\boldsymbol{\xi}}^{I, i}\right), \boldsymbol{P}_p\left(e_{\boldsymbol{\xi}}^{h, i}\right)\right\rangle_{\Gamma} \right),
    \\
    T_{\xi,3}^{i}&:=2 \Delta t\left\langle\gamma \boldsymbol{P}_p\left(\boldsymbol{u}^{i}-\boldsymbol{u}^{i-1}-e_{\boldsymbol{u}}^{I, {i-1}}\right), \boldsymbol{P}_p\left(e_{\boldsymbol{\xi}}^{h, i}\right)\right\rangle_{\Gamma} 
    \\
    & \quad +2 \Delta t\left\langle\left(\boldsymbol{\xi}^{i}-\boldsymbol{\xi}^{i-1}-e_{\boldsymbol{\xi}}^{I,{i-1}}\right) \cdot \boldsymbol{n}_p, e_{\boldsymbol{\xi}}^{h, i} \cdot \boldsymbol{n}_p\right\rangle_{\Gamma} ,
    \\
    T_{\xi,4}^{i}&:= 2 \Delta t \rho_p\left(\frac{\Pi_{\boldsymbol{\eta}} \boldsymbol{\xi}^{i}-\Pi_{\boldsymbol{\eta}} \boldsymbol{\xi}^{i-1}}{\Delta t}-\partial_t \boldsymbol{\xi}^{i}, e_{\boldsymbol{\xi}}^{h, i}\right)_p 
\end{align*}    
{\sunny{and estimate the sums from $ i = 1$ to $n$ of each of the four terms as follows.}}
\vskip 0.1in

\noindent
{\sunny{{\bf{Term $\sum_{i=1}^n T_{\xi,1}^{i}$:}}
By recalling $e_{\boldsymbol{\xi}}^{h,i} = (e_{\boldsymbol{\eta}}^{h,i} - e_{\boldsymbol{\eta}}^{h,i-1})/\Delta t$,      
we use the identity}}
\begin{align*}
    \notag
    \sum_{i=1}^n T_{\xi,1}^{i}&= -2 \Delta t  \alpha \sum_{i=1}^n \left(e_\phi^{I, i}, \nabla \cdot e_{\boldsymbol{\xi}}^{h, i}\right)_{p}
    \\
    \notag
    &=  -2  \alpha \sum_{i=1}^n \left(e_\phi^{I, i}, \nabla \cdot (e_{\boldsymbol{\eta}}^{h, i} - e_{\boldsymbol{\eta}}^{h, i-1})\right)_{p} \qquad (\text{because }\Delta t e_{\boldsymbol{\xi}}^{h, i}=(e_{\boldsymbol{\eta}}^{h, i} - e_{\boldsymbol{\eta}}^{h, i-1}))
    \\
    &= 2 \alpha \left( \sum_{i=2}^n \left(e_\phi^{I, i} - e_\phi^{I,i-1}, \nabla \cdot  e_{\boldsymbol{\eta}}^{h, i-1}\right)_{p} - \left(e_\phi^{I, n}, \nabla \cdot  e_{\boldsymbol{\eta}}^{h, n}\right)_{p} + \left(e_\phi^{I, 1}, \nabla \cdot  e_{\boldsymbol{\eta}}^{h, 0}\right)_{p} \right)
\end{align*}
{\sunny{to obtain}}
\begin{align}
    \notag
    & \left|\sum_{i=1}^n T_{\xi,1}^{i}\right|= \Delta t \varepsilon \sum_{i=1}^{n-1} \left\|\nabla \cdot e_{\boldsymbol{\eta}}^{h, i}\right\|_{\Omega_p}^2 + \varepsilon (\left\|\nabla \cdot e_{\boldsymbol{\eta}}^{h, n}\right\|_{\Omega_p}^2 + \left\|\nabla \cdot e_{\boldsymbol{\eta}}^{h, 0}\right\|_{\Omega_p}^2 )
    \\
    \label{eq:T_xi_1-estimate}
    & + C_{\varepsilon} \left\|\partial_t e_\phi^{I}\right\|_{L^2(t_{1}, t_{n}; L^2(\Omega_p))}^2 + C_{\varepsilon} \left\|e_\phi^{I, 1}\right\|_{\Omega_p}^2 + C_{\varepsilon} \left\|e_\phi^{I, n}\right\|_{\Omega_p}^2 
    \\
    \notag 
    &\leq \Delta t \varepsilon \sum_{i=1}^{n-1} \left\|\nabla \cdot e_{\boldsymbol{\eta}}^{h, i}\right\|_{\Omega_p}^2 + \varepsilon ( \left\|\nabla \cdot e_{\boldsymbol{\eta}}^{h, n}\right\|_{\Omega_p}^2 + \left\|\nabla \cdot e_{\boldsymbol{\eta}}^{h, 0}\right\|_{\Omega_p}^2 )
    \\
    \notag
    &+ C_{\varepsilon} h^{2k+2r} \left( \left\|\partial_t \phi \right\|_{L^2(t_{1}, t_{n}; H^{k+1}(\Omega_p))}^2 + C_{\varepsilon} \left\|\phi^{1}\right\|_{H^{k+1}(\Omega_p)}^2 + C_{\varepsilon} \left\|\phi^{n}\right\|_{H^{k+1}(\Omega_p)}^2 \right).
\end{align}

\vskip 0.1in

\noindent
{\sunny{{\bf{Term $\sum_{i=1}^n T_{\xi,2}^{i}$:}}
Again, by using $e_{\boldsymbol{\xi}}^{h,i} = (e_{\boldsymbol{\eta}}^{h,i} - e_{\boldsymbol{\eta}}^{h,i-1})/\Delta t$, we get:}}
\begin{align*} 
    &{\small{\sum_{i=1}^n T_{\xi,2}^i =\sum_{i=1}^n 2 \Delta t\left(\left\langle e_{\boldsymbol{\xi}}^{I, i} \cdot \boldsymbol{n}_p, e_{\boldsymbol{\xi}}^{h, i} \cdot \boldsymbol{n}_p\right\rangle_{\Gamma} + \left\langle e_\phi^{I, i}, e_{\boldsymbol{\xi}}^{h, i} \cdot \boldsymbol{n}_p\right\rangle_{\Gamma} + \gamma \left\langle \boldsymbol{P}_p \left(e_{\boldsymbol{\xi}}^{I, i} \right), \boldsymbol{P}_p \left( e_{\boldsymbol{\xi}}^{h, i} \right) \right\rangle_{\Gamma} \right)}}  
    \\
    &= -\sum_{i=2}^n 2 \left(\left\langle (e_{\boldsymbol{\xi}}^{I, i} - e_{\boldsymbol{\xi}}^{I, i-1}) \cdot \boldsymbol{n}_p + e_{\phi}^{I,i} - e_{\phi}^{I,i-1}, e_{\boldsymbol{\eta}}^{h, i-1} \cdot \boldsymbol{n}_p\right\rangle_{\Gamma} \right.
    \\
    & \qquad \qquad \qquad + \gamma \left. \left\langle \boldsymbol{P}_p \left(e_{\boldsymbol{\xi}}^{I, i} - e_{\boldsymbol{\xi}}^{I, i-1} \right) , \boldsymbol{P}_p\left( e_{\boldsymbol{\eta}}^{h, i-1} \right) \right\rangle_{\Gamma} \right)
    \\
    &\quad + 2 \left(- \left\langle e_{\boldsymbol{\xi}}^{I, 1} \cdot \boldsymbol{n}_p + e_\phi^{I, 1}, e_{\boldsymbol{\eta}}^{h, 0} \cdot \boldsymbol{n}_p\right\rangle_{\Gamma} + \left\langle e_{\boldsymbol{\xi}}^{I, n} \cdot \boldsymbol{n}_p + e_\phi^{I, n}, e_{\boldsymbol{\eta}}^{h, n} \cdot \boldsymbol{n}_p\right\rangle_{\Gamma} \right)
    \\
    &\quad + 2\gamma  \left(- \left\langle \boldsymbol{P}_p \left(e_{\boldsymbol{\xi}}^{I, 1}  \right) , \boldsymbol{P}_p\left( e_{\boldsymbol{\eta}}^{h, 0} \right) \right\rangle_{\Gamma} + \left\langle \boldsymbol{P}_p \left(e_{\boldsymbol{\xi}}^{I, n}  \right) , \boldsymbol{P}_p\left( e_{\boldsymbol{\eta}}^{h, n} \right) \right\rangle_{\Gamma}  \right) .
\end{align*}
{\sunny{Now, since $\| (e_{\boldsymbol{\xi}}^{I, i} - e_{\boldsymbol{\xi}}^{I, i-1}) \cdot \boldsymbol{n}_p \|_{\Gamma} \le \| \partial_t^2 e_{\boldsymbol{\eta}}^{I} \|_{L^1(t_{i-2}, t_i; L^2(\Gamma))}$, by using \eqref{eq:vector-trace} and Young's inequality, we get:}}
\begin{align}\label{eq:T_xi_2-estimate}
    \left| \sum_{i=1}^n T_{\xi,2}^i \right| &\leq C \sum_{i=2}^n (\| \partial_t^2 e_{\boldsymbol{\eta}}^{I} \|_{L^1(t_{i-2}, t_i; L^2(\Gamma))} + \|\partial_t e_{\phi}^{I} \|_{L^1(t_{i-1}, t_i; L^2(\Gamma))}) \left\|e_{\boldsymbol{\eta}}^{h, i-1} \right\|_{\Gamma} 
    \\
    \notag
    &\quad + C (\| e_{\boldsymbol{\xi}}^{I, 1} \|_{\Gamma} + \| e_{\phi}^{I,1} \|_{\Gamma}) \|e_{\boldsymbol{\eta}}^{h,0} \|_{\Gamma} + C ( \| e_{\boldsymbol{\xi}}^{I, n} \|_{\Gamma} + \|e_{\phi}^{I,n} \|_{\Gamma} )\|e_{\boldsymbol{\eta}}^{h,n} \|_{\Gamma}
    \\
    \notag
    &\leq C (\Delta t)^{\frac 12} \sum_{i=2}^n (\| \partial_t^2 e_{\boldsymbol{\eta}}^{I} \|_{L^2(t_{i-2}, t_i; L^2(\Gamma))} + \|\partial_t e_{\phi}^{I} \|_{L^2(t_{i-1}, t_i; L^2(\Gamma))}) \left\|e_{\boldsymbol{\eta}}^{h, i} \right\|_{\Gamma} 
    \\
    \notag
    &\quad +  C (\| e_{\boldsymbol{\xi}}^{I, 1} \|_{\Gamma} + \| e_{\phi}^{I,1} \|_{\Gamma}) \|e_{\boldsymbol{\eta}}^{h,0} \|_{\Gamma} + C ( \| e_{\boldsymbol{\xi}}^{I, n} \|_{\Gamma} + \|e_{\phi}^{I,n} \|_{\Gamma} )\|e_{\boldsymbol{\eta}}^{h,n} \|_{\Gamma}
    \\
    \notag
    & \leq \Delta t \varepsilon \sum_{i=1}^{n-1}   
    \left\|\boldsymbol{D}(e_{\boldsymbol{\eta}}^{h, i} )\right\|_{\Omega_p}^2  + \varepsilon (\left\|\boldsymbol{D}(e_{\boldsymbol{\eta}}^{h, n} )\right\|_{\Omega_p}^2 + \left\|\boldsymbol{D}(e_{\boldsymbol{\eta}}^{h, 0} )\right\|_{\Omega_p}^2)  
    \eqref{eq:vector-trace} 
    \\
    \notag
    &\quad +  C_{\varepsilon} \left(\left\|\partial_t^2 e_{\boldsymbol{\eta}}^{I} \right\|_{L^2(0,t_n; L^2(\Gamma))}^2+\left\|\partial_t e_\phi^{I}\right\|_{L^2(t_1,t_n; L^2(\Gamma))}^2 \right)
    \\
    \notag
    &\quad + C_{\varepsilon} \left( \| \partial_t e_{\boldsymbol{\eta}}^{I} \|_{L^{\infty}(0, t_1; L^2(\Gamma))} + \| e_{\phi}^{I,1} \|_{\Gamma} +  \| \partial_t e_{\boldsymbol{\eta}}^{I} \|_{L^{\infty}(t_{n-1}, t_n; L^2(\Gamma))} + \| e_{\phi}^{I,n} \|_{\Gamma}\right)
    \\
    \notag
    &\leq \Delta t \varepsilon \sum_{i=1}^{n-1}   
    \left\|\boldsymbol{D}(e_{\boldsymbol{\eta}}^{h, i} )\right\|_{\Omega_p}^2 + \varepsilon (\left\|\boldsymbol{D}(e_{\boldsymbol{\eta}}^{h, n} )\right\|_{\Omega_p}^2 + \left\|\boldsymbol{D}(e_{\boldsymbol{\eta}}^{h, 0} )\right\|_{\Omega_p}^2) 
    \notag
    \\
    &\quad + C_{\varepsilon} h^{2k+r} \left(\left\|\partial_t^2 \boldsymbol{\eta} \right\|_{L^2(0,t_n; H^{k+1}(\Omega_p))}^2+\left\|\partial_t \phi\right\|_{L^2(t_1,t_n; H^{k+1}(\Omega_p))}^2 \right) 
    \notag
    \\
    \notag
    &\quad + C_{\varepsilon} h^{2k+r} \left( \| \partial_t \boldsymbol{\eta} \|_{L^{\infty}(0, t_1; H^{k+1}(\Omega_p))} +  \| \partial_t \boldsymbol{\eta} \|_{L^{\infty}(t_{n-1}, t_n; H^{k+1}(\Omega_p))} 
    \right.
    \notag
    \\
    & \quad \left.+ \| \phi^1, \phi^{n} \|_{H^{k+1}(\Omega_p)}\right).
    \notag
\end{align}

\vskip 0.1in

\noindent
{\sunny{{\bf{Term $\sum_{i=1}^n T_{\xi,3}^{i}$:}}
To estimate the first term in the definition $T_{\xi,3}^{i}$  we observe that:}}
$$
\begin{aligned}
    &2 \Delta t \gamma \sum_{i=1}^n \left\langle \boldsymbol{P}_p\left(\boldsymbol{u}^{i}-\boldsymbol{u}^{i-1}-e_{\boldsymbol{u}}^{I, {i-1}}\right), \boldsymbol{P}_p e_{\boldsymbol{\xi}}^{h, i}\right\rangle_{\Gamma}
    \\
    &= -2 \gamma \sum_{i=2}^n \left\langle \boldsymbol{P}_p\left(\boldsymbol{u}^{i}-2\boldsymbol{u}^{i-1}+\boldsymbol{u}^{i-2} - (e_{\boldsymbol{u}}^{I, {i-1}} - e_{\boldsymbol{u}}^{I, {i-2}})\right), \boldsymbol{P}_p e_{\boldsymbol{\eta}}^{h, i-1}\right\rangle_{\Gamma} 
    \\
    &\quad + 2\gamma \left\langle \boldsymbol{P}_p\left(\boldsymbol{u}^{n}-\boldsymbol{u}^{n-1}-e_{\boldsymbol{u}}^{I, {n-1}}\right), \boldsymbol{P}_p e_{\boldsymbol{\eta}}^{h, n}\right\rangle_{\Gamma} - 2\gamma \left\langle \boldsymbol{P}_p\left(\boldsymbol{u}^{1}-\boldsymbol{u}^{0}-e_{\boldsymbol{u}}^{I, {0}}\right), \boldsymbol{P}_p e_{\boldsymbol{\eta}}^{h, 0}\right\rangle_{\Gamma} ,
\end{aligned}
$$
{\sunny{and obtain the following estimate:}}
$$
\begin{aligned}
    &{\small{2 \Delta t \gamma \left| \sum_{i=1}^n \left\langle \boldsymbol{P}_p\left(\boldsymbol{u}^{i}-\boldsymbol{u}^{i-1}-e_{\boldsymbol{u}}^{I, {i-1}}\right), \boldsymbol{P}_p e_{\boldsymbol{\xi}}^{h, i}\right\rangle_{\Gamma} \right|}}
    \\
    &\leq 2 \gamma \sum_{i=2}^n \| \boldsymbol{u}^{i}-2\boldsymbol{u}^{i-1}+\boldsymbol{u}^{i-2} - (e_{\boldsymbol{u}}^{I, {i-1}} - e_{\boldsymbol{u}}^{I, {i-2}})\|_{\Gamma} \|  e_{\boldsymbol{\eta}}^{h, i-1}\|_{\Gamma} 
\\
    &\quad + 2\gamma \| \boldsymbol{u}^{n}-\boldsymbol{u}^{n-1}-e_{\boldsymbol{u}}^{I, {n-1}} \|_{\Gamma} \| e_{\boldsymbol{\eta}}^{h, n}\|_{\Gamma} + 2\gamma \| \boldsymbol{u}^{1}-\boldsymbol{u}^{0}-e_{\boldsymbol{u}}^{I, {0}} \|_{\Gamma} \| e_{\boldsymbol{\eta}}^{h, 0}\|_{\Gamma}  
\\
    &\leq 2 \gamma \sum_{i=2}^n ( \Delta t \| \partial_t^2 \boldsymbol{u} \|_{L^1(t_{i-2}, t_i; L^2(\Gamma))} + \|\partial_t e_{\boldsymbol{u}}^{I}\|_{L^1(t_{i-2},t_{i-1}; L^2(\Gamma))}) \|  e_{\boldsymbol{\eta}}^{h, i-1}\|_{\Gamma} 
    \\
    &\quad + 2\gamma \| \boldsymbol{u}^{n}-\boldsymbol{u}^{n-1}-e_{\boldsymbol{u}}^{I, {n-1}} \|_{\Gamma} \| e_{\boldsymbol{\eta}}^{h, n}\|_{\Gamma} + 2\gamma \| \boldsymbol{u}^{1}-\boldsymbol{u}^{0}-e_{\boldsymbol{u}}^{I, {0}} \|_{\Gamma} \| e_{\boldsymbol{\eta}}^{h, 0}\|_{\Gamma}
    \\
    &\le 2 \gamma \sum_{i=2}^n ( (\Delta t)^{\frac 32} \| \partial_t^2 \boldsymbol{u} \|_{L^2(t_{i-2}, t_i; L^2(\Gamma))} + (\Delta t)^{\frac 12} \|\partial_t e_{\boldsymbol{u}}^{I}\|_{L^2(t_{i-2},t_{i-1}; L^2(\Gamma))}) \|  e_{\boldsymbol{\eta}}^{h, i-1}\|_{\Gamma} 
    \\
    &\quad + 2 \gamma (\Delta t \| \partial_t \boldsymbol{u} \|_{L^{\infty}(t_{n-1},t_n; L^2(\Gamma))} + \|e_{\boldsymbol{u}}^{I, {n-1}} \|_{\Gamma}) \| e_{\boldsymbol{\eta}}^{h, n}\|_{\Gamma} 
    \\
    &\quad + 2\gamma (\Delta t \| \partial_t \boldsymbol{u} \|_{L^{\infty}(t_{0},t_1; L^2(\Gamma))} + \|e_{\boldsymbol{u}}^{I, {0}} \|_{\Gamma}) \| e_{\boldsymbol{\eta}}^{h, 0}\|_{\Gamma} 
    \\
    &\le \varepsilon \Delta t \sum_{i=2}^n \| \boldsymbol{D}(e_{\boldsymbol{\eta}}^{h,i-1})\|_{\Omega_p}^2 + C_{\varepsilon} ((\Delta t)^2 \| \partial_t^2 \boldsymbol{u} \|_{L^2(t_{1}, t_n; L^2(\Gamma))}^2 
    \\
    & \quad + h^{2k+r} \|\partial_t \boldsymbol{u}\|_{L^2(0,t_{n-1}; H^{k+1}(\Omega_f))}^2) 
    + \varepsilon (\| \boldsymbol{D}(e_{\boldsymbol{\eta}}^{h, n})\|_{\Omega_p}^2 + \| \boldsymbol{D}(e_{\boldsymbol{\eta}}^{h, 0})\|_{\Omega_p}^2  ) 
    \\
    &\quad + C_{\varepsilon} ((\Delta t)^2 \| \partial_t \boldsymbol{u} \|_{L^{\infty}(t_{n-1},t_n; L^2(\Gamma))}^2 + h^{2k+r} \|\boldsymbol{u}^{n-1} , \boldsymbol{u}^0 \|_{H^{k+1}(\Omega_p)} ) .
\end{aligned}
$$
{\sunny{Similarly, to estimate the second term in the definition $T_{\xi,3}^{i}$  we observe that:}}
$$
\begin{aligned}
    & 2 \Delta t \sum_{i=1}^n \left\langle \left(\boldsymbol{\xi}^{i}-\boldsymbol{\xi}^{i-1}-e_{\boldsymbol{\xi}}^{I, {i-1}}\right) \cdot \boldsymbol{n}_p, e_{\boldsymbol{\xi}}^{h, i} \cdot \boldsymbol{n}_p \right\rangle_{\Gamma}
    \\
    &= -2 \sum_{i=2}^n \left\langle \left(\boldsymbol{\xi}^{i}-2\boldsymbol{\xi}^{i-1}+\boldsymbol{\xi}^{i-2} - (e_{\boldsymbol{\xi}}^{I, {i-1}} - e_{\boldsymbol{\xi}}^{I, {i-2}})\right) \cdot \boldsymbol{n}_p,  e_{\boldsymbol{\eta}}^{h, i-1} \cdot \boldsymbol{n}_p\right\rangle_{\Gamma} 
    \\
    &\quad+ 2 \left\langle \left(\boldsymbol{\xi}^{n}-\boldsymbol{\xi}^{n-1}-e_{\boldsymbol{\xi}}^{I, {n-1}}\right) \cdot \boldsymbol{n}_p,  e_{\boldsymbol{\eta}}^{h, n} \cdot \boldsymbol{n}_p \right\rangle_{\Gamma} - 2 \left\langle \left(\boldsymbol{\xi}^{1}-\boldsymbol{\xi}^{0}-e_{\boldsymbol{\xi}}^{I, {0}}\right) \cdot \boldsymbol{n}_p,  e_{\boldsymbol{\eta}}^{h, 0} \cdot \boldsymbol{n}_p\right\rangle_{\Gamma} 
\end{aligned}
$$
and obtain: 
$$
\begin{aligned}
    &2 \Delta t \left| \sum_{i=1}^n \left\langle \left(\boldsymbol{\xi}^{i}-\boldsymbol{\xi}^{i-1}-e_{\boldsymbol{\xi}}^{I, {i-1}}\right) \cdot \boldsymbol{n}_p, e_{\boldsymbol{\xi}}^{h, i} \cdot \boldsymbol{n}_p \right\rangle_{\Gamma} \right|
    \\
    &\leq 2 \sum_{i=2}^n \| \boldsymbol{\xi}^{i}-2\boldsymbol{\xi}^{i-1}+\boldsymbol{\xi}^{i-2} - (e_{\boldsymbol{\xi}}^{I, {i-1}} - e_{\boldsymbol{\xi}}^{I, {i-2}})\|_{\Gamma} \|  e_{\boldsymbol{\eta}}^{h, i-1}\|_{\Gamma} 
    \\
    &\quad + 2 \| \boldsymbol{\xi}^{n}-\boldsymbol{\xi}^{n-1}-e_{\boldsymbol{\xi}}^{I, {n-1}} \|_{\Gamma} \| e_{\boldsymbol{\eta}}^{h, n}\|_{\Gamma} + 2 \| \boldsymbol{\xi}^{1}-\boldsymbol{\xi}^{0}-e_{\boldsymbol{\xi}}^{I, {0}} \|_{\Gamma} \| e_{\boldsymbol{\eta}}^{h, 0}\|_{\Gamma}  
    \\
    &\le 2 \sum_{i=2}^n ( (\Delta t)^{\frac 32} \| \partial_t^2 \boldsymbol{\xi} \|_{L^2(t_{i-2}, t_i; L^2(\Gamma))} + (\Delta t)^{\frac 12} \|\partial_t e_{\boldsymbol{\xi}}^{I}\|_{L^2(t_{i-2},t_{i-1}; L^2(\Gamma))}) \|  e_{\boldsymbol{\eta}}^{h, i-1}\|_{\Gamma} 
    \\
    &\quad + 2 (\Delta t \| \partial_t \boldsymbol{\xi} \|_{L^{\infty}(t_{n-1},t_n; L^2(\Gamma))} + \|e_{\boldsymbol{\xi}}^{I, {n-1}} \|_{\Gamma}) \| e_{\boldsymbol{\eta}}^{h, n}\|_{\Gamma} 
    \\
    &\quad + 2\gamma (\Delta t \| \partial_t \boldsymbol{\xi} \|_{L^{\infty}(t_{0},t_1; L^2(\Gamma))} + \|e_{\boldsymbol{\xi}}^{I, {0}} \|_{\Gamma}) \| e_{\boldsymbol{\eta}}^{h, 0}\|_{\Gamma} 
\end{aligned}
$$
$$
\begin{aligned}
    &\le \varepsilon \Delta t \sum_{i=2}^n \| \boldsymbol{D}(e_{\boldsymbol{\eta}}^{h,i-1})\|_{\Omega_p}^2 + C_{\varepsilon} ((\Delta t)^2 \| \partial_t^2 \boldsymbol{\xi} \|_{L^2(t_{1}, t_n; L^2(\Gamma))}^2 \\
    & \quad + h^{2k+r} \|\partial_t \boldsymbol{\xi}\|_{L^2(0,t_{n-1}; H^{k+1}(\Omega_p))}^2) 
     + \varepsilon (\| \boldsymbol{D}(e_{\boldsymbol{\eta}}^{h, n})\|_{\Omega_p}^2 + \| \boldsymbol{D}(e_{\boldsymbol{\eta}}^{h, 0})\|_{\Omega_p}^2  ) 
    \\
    &\quad + C_{\varepsilon} ((\Delta t)^2 \| \partial_t \boldsymbol{\xi} \|_{L^{\infty}(t_{n-1},t_n; L^2(\Gamma))}^2 + h^{2k+r} \|\boldsymbol{\xi}^{n-1} , \boldsymbol{\xi}^0 \|_{H^{k+1}(\Omega_p)}^2 ) .
\end{aligned}
$$
Combining these two estimates yields the following estimate of $\sum_{i=1}^n T_{\xi,3}^{i}$:
\begin{align}
    \notag
    \left| \sum_{i=1}^n T_{\xi,3}^{i} \right| &\le \varepsilon \Delta t \sum_{i=1}^{n-1} \| \boldsymbol{D}(e_{\boldsymbol{\eta}}^{h,i})\|_{\Omega_p}^2 + \varepsilon (\| \boldsymbol{D}(e_{\boldsymbol{\eta}}^{h, n})\|_{\Omega_p}^2 + \| \boldsymbol{D}(e_{\boldsymbol{\eta}}^{h, 0})\|_{\Omega_p}^2  )  
    \\
    \label{eq:T_xi_3-estimate}
    &\quad  + C_{\varepsilon} (\Delta t)^2  (\| \partial_t^2 \boldsymbol{u}, \partial_t^2 \boldsymbol{\xi} \|_{L^2(t_{1}, t_n; L^2(\Gamma))}^2 + \| \partial_t \boldsymbol{u} , \partial_t \boldsymbol{\xi} \|_{L^{\infty}(t_{n-1},t_n; L^2(\Gamma))}^2)
    \\
    \notag
    &\quad + C_{\varepsilon} h^{2k+r}( \|\partial_t \boldsymbol{u}\|_{L^2(0,t_{n-1}; H^{k+1}(\Omega_f))}^2 +\|\boldsymbol{u}^{n-1} , \boldsymbol{u}^0 \|_{H^{k+1}(\Omega_f)}^2)
    \\
    \notag
    &\quad + C_{\varepsilon} h^{2k+r}( \|\partial_t \boldsymbol{\xi}\|_{L^2(0,t_{n-1}; H^{k+1}(\Omega_p))}^2 +\|\boldsymbol{\xi}^{n-1} , \boldsymbol{\xi}^0 \|_{H^{k+1}(\Omega_p)}^2) .
\end{align}

\vskip 0.1in

\noindent
{\sunny{{\bf{Term $\sum_{i=1}^n T_{\xi,4}^{i}$:}}
}}
As in \eqref{eq:T_u_3-estimate}, we can estimate $T_{\xi,4}^{i}$ by 
\begin{align}
    \notag 
    \sum_{i=1}^n|T_{\xi,4}^{i}| &=\sum_{i=1}^n\left| 2 \Delta t \rho_p \left( \frac{\Pi_{\boldsymbol{\eta}} \boldsymbol{\xi}^{i} - \Pi_{\boldsymbol{\eta}} \boldsymbol{\xi}^{i-1}}{\Delta t} - \partial_t \boldsymbol{\xi}^{i}, {e}_{\boldsymbol{\xi}}^{h,i} \right)_p \right|
    \\
    \label{eq:T_xi_4-estimate}
    &\le  \Delta t \varepsilon \sum_{i=1}^{n-1} \|{e}_{\boldsymbol{\xi}}^{h,i}\|_{\Omega_p}^2 + \varepsilon   \|{e}_{\boldsymbol{\xi}}^{h,n}\|_{\Omega_p}^2 + C_{\varepsilon} (\Delta t)^2 \|\partial_{t}^2 \boldsymbol{\xi}\|_{L^{\infty}(t_{n-1}, t_n; L^2(\Omega_p))}^2
    \\
    \notag
    &\quad + C_{\varepsilon} \left(h^{2(k+r)}\|\partial_t \boldsymbol{\xi}\|_{L^2(0,t_n; H^{k+1}(\Omega_p))}^2  + (\Delta t)^2 \|\partial_{t}^2 \boldsymbol{\xi}\|_{L^{2}(0, t_{n-1}; L^2(\Omega_p))}^2 \right) .
\end{align}
%

{\sunny{Finally, we can now obtain the desired estimate \eqref{eq:F_xi-estimate} stated in the theorem. More specifically, by adding the 
estimates  \eqref{eq:T_xi_1-estimate}--\eqref{eq:T_xi_4-estimate} and ignoring  higher order terms, we get:}}
$$
\begin{aligned}
    &\left|\sum_{i=1}^n F_{\boldsymbol{\xi}}^{i}\left(e_{\boldsymbol{\xi}}^{h, i}\right) \right|
    \leq  \varepsilon \Delta t \sum_{i=1}^{n-1} \| \boldsymbol{D}(e_{\boldsymbol{\eta}}^{h,i})\|_{\Omega_p}^2 + \varepsilon (\| \boldsymbol{D}(e_{\boldsymbol{\eta}}^{h, n})\|_{\Omega_p}^2 + \| \boldsymbol{D}(e_{\boldsymbol{\eta}}^{h, 0})\|_{\Omega_p}^2  )  
    \\
    &\quad  + C_{\varepsilon} (\Delta t)^2  \left(\| \partial_t^2 \boldsymbol{u}, \partial_t^3 \boldsymbol{\eta} \|_{L^2(t_{1}, t_n; L^2(\Gamma))}^2 + \| \partial_t^3 \boldsymbol{\eta} \|_{L^2(0, t_n; L^2(\Omega_p))}^2 \right.
    \notag
    \\
    &\quad  + \left.  \| \partial_t \boldsymbol{u} , \partial_t^2 \boldsymbol{\eta}, \partial_t^3 \boldsymbol{\eta} \|_{L^{\infty}(t_{n-1},t_n; L^2(\Gamma))}^2 \right)
    + C_{\varepsilon} h^{2k+r} \left( \|\partial_t \boldsymbol{u}\|_{L^2(0,t_{n-1}; H^{k+1}(\Omega_f))}^2 \right.
    \notag
    \\
    & \quad +\|\boldsymbol{u}^{n-1} , \boldsymbol{u}^0 \|_{H^{k+1}(\Omega_f)}^2
    + \|\partial_t \boldsymbol{\eta}, \partial_t^2 \boldsymbol{\eta}, \partial_t \phi \|_{L^2(0,t_{n}; H^{k+1}(\Omega_p))}^2 
    \notag
    \\
    & \qquad +  \|\boldsymbol{\xi}^{n-1} , \boldsymbol{\xi}^0, \phi^1, \phi^n \|_{H^{k+1}(\Omega_p)}^2
    + \|\partial_t \boldsymbol{\eta} \|_{L^{\infty}(0,t_{1}; H^{k+1}(\Omega_p))}^2 
    \notag
    \\
    &\qquad \left. + \|\partial_t \boldsymbol{\eta} \|_{L^{\infty}(t_{n-1},t_{n}; H^{k+1}(\Omega_p))}^2 \right) .
\end{aligned}
$$
{\sunny{ This is exactly \eqref{eq:F_xi-estimate}.}}

\vskip 0.1in
\noindent
{\bf The estimate of $\sum_{i=1}^n F_{\phi}^{i}\left(e_{\phi}^{h, i}\right)$}: 

{\sunny{We split $F_\phi^{i}
\left(e_\phi^{h, i}\right)$ into the following four terms and estimate each term separately. }}

\begin{align*}
    T_{\phi,1}^{i}&:= 2 \Delta t \alpha\left(\nabla \cdot e_{\boldsymbol{\xi}}^{I, i}, e_\phi^{h, i}\right)_p ,
    \qquad 
    T_{\phi,2}^{i}:=-2 \Delta t \left\langle e_{\boldsymbol{\xi}}^{I, i} \cdot \boldsymbol{n}_p, e_\phi^{h, i}\right\rangle_{\Gamma}, 
\\
    T_{\phi,3}^{i}&:=2 \Delta t\left\langle-\left(\boldsymbol{u}^{i}-\boldsymbol{u}^{i-1}-e_{\boldsymbol{u}}^{I, {i-1}}\right) \cdot \boldsymbol{n}_p+\left(\phi^{i}-\phi^{i-1}-e_\phi^{I, {i-1}}\right) / L, e_\phi^{h, i}\right\rangle_{\Gamma} ,
    \\
    T_{\phi,4}^{i}&:=2 \Delta t C_0\left(\frac{\Pi_\phi \phi^{i}-\Pi_\phi \phi^{i-1}}{\Delta t}-\partial_t \phi^{i}, e_\phi^{h, i}\right)_p.
\end{align*}

\noindent
{\bf{Term $T_{\phi,1}^{i}$:}}
{\sunny{We integrate by parts $T_{\phi,i}$, and us the Cauchy--Schwarz and Young's inequalities to obtain:}}
\begin{align*}
    T_{\phi,1}^{i} &= - 2 \Delta t \alpha \left( e_{\boldsymbol{\xi}}^{I, i}, \nabla e_{\phi}^{h,i} \right)_p + 2\Delta t \alpha \left\langle e_{\boldsymbol{\xi}}^{I, i} \cdot \boldsymbol{n}_p,  e_{\phi}^{h,i} \right\rangle_{\Gamma} 
    \\
    &\le \varepsilon \Delta t \left\| \nabla e_\phi^{h, i} \right\|_{\Omega_p}^2 + C_{\varepsilon} \Delta t h^{2k+2r} \|  \boldsymbol{\xi}^i\|_{H^{k+1}(\Omega_p)}^2 
    \\
    &\quad + C\Delta t \left\|e_{\boldsymbol{\xi}}^{I,i}\right\|_{\Omega_p}^{\frac 12}\left\|\nabla e_{\boldsymbol{\xi}}^{I,i}\right\|_{\Omega_p}^{\frac 12} \left\| e_\phi^{h, i} \right\|_{\Omega_p}^{\frac 12}  \left\| \nabla e_\phi^{h, i} \right\|_{\Omega_p}^{\frac 12} 
    \\
    &\le  C\varepsilon \Delta t \left( \left\| e_\phi^{h, i} \right\|_{\Omega_p}^2 + \left\| \nabla e_\phi^{h, i} \right\|_{\Omega_p}^2 \right) + C_{\varepsilon} \Delta t h^{2k+r}(1+h^r) \|  \boldsymbol{\xi}^{i}\|_{H^{k+1}(\Omega_p)}^2 .
\end{align*}

\noindent
{\sunny{\bf{The terms $T_{\phi,j}^{i}$, $j=2,3,4$:}}}
{\sunny{Similarly, we estimate the remaining three terms $T_{\phi,j}^{i}$, $j=2,3,4$, by using the arguments similar to those used in the estimates of $T_{{\xi},j}^{i}$, to obtain:}}
\begin{align*}
    T_{\phi,2}^{i} &\le C\varepsilon \Delta t \left( \left\| e_\phi^{h, i} \right\|_{\Omega_p}^2 + \left\| \nabla e_\phi^{h, i} \right\|_{\Omega_p}^2 \right) + C_{\varepsilon} \Delta t h^{2k+r} \left\|\boldsymbol{\xi}^{i} \right\|_{H^{k+1}(\Omega_p)}^2  ,
    \\
    T_{\phi,3}^{i} &\le C\varepsilon \Delta t \left( \left\| e_\phi^{h, i} \right\|_{\Omega_p}^2 + \left\| \nabla e_\phi^{h, i} \right\|_{\Omega_p}^2 \right) 
    \\
    \notag
    &\quad + C_{\varepsilon} (\Delta t)^2 \left(\left\| \partial_t \phi \right\|^2_{L^{2}(t_{i-1}, t_{i}; H^1(\Omega_p))} + \left\|\partial_t \boldsymbol{u} \right\|^2_{L^{2}(t_{i-1}, t_{i}; H^1(\Omega_f))} \right)  
    \\
    &\quad +C_{\varepsilon} \Delta t h^{2k+r}\left( \left\|\phi^{i-1} \right\|^2_{H^{k+1}\left(\Omega_p\right)} + \left\|\boldsymbol{u}^{i-1}\right\|^2_{H^{k+1}\left(\Omega_f\right)} + \left\|p^{i-1}\right\|^2_{H^{k}\left(\Omega_f\right)} \right),
    \\
    T_{\phi,4}^{i} &\le  \varepsilon\Delta t \|{e}_{\phi}^{h,i}\|_{\Omega_p}^2 +  C_{\varepsilon} \Delta t \left(h^{2k+r}\|\partial_t \phi^{i}\|_{H^{k+r}(\Omega_p)}^2  + \Delta t \|\partial_{tt} \phi\|_{L^{2}(t_{i-1}, t_{i}; L^2(\Omega_p))}^2 \right) .
\end{align*}
{\sunny{Now, the desired estimate of $F_\phi^{i}\left(e_\phi^{h, i}\right)$ 
follows by simplifying higher order terms:}}
\begin{align*}
    F_\phi^{i}\left(e_\phi^{h, i}\right) &  \leq  C\varepsilon \Delta t \left( \left\| e_\phi^{h, i} \right\|_{\Omega_p}^2 + \left\| \nabla e_\phi^{h, i} \right\|_{\Omega_p}^2 \right)
    + C_{\varepsilon} (\Delta t)^2 \left( \left\| \partial_t \phi \right\|^2_{L^{2}\left(t_{i-1}, t_{i}; H^1(\Omega_p)\right)} 
    \right.
    \\
    & \qquad \left. +  \left\| \partial_{tt} \phi \right\|^2_{L^{2}\left(t_{i-1}, t_{i}; L^2(\Omega_p)\right)} + \left\|\partial_t \boldsymbol{u} \right\|^2_{L^{2}\left(t_{i-1}, t_{i}; H^1(\Omega_f)\right)} \right)
    \\
    & \qquad + C_{\varepsilon} \Delta th^{2k+r} \left( \left\|\boldsymbol{\xi}^{i}, \phi^{i-1}, \partial_t \phi^{i} \right\|_{H^{k+1}(\Omega_p)}^2 + \left\|\boldsymbol{u}^{i-1} \right\|_{H^{k+1}(\Omega_f)}^2 
    \right.
    \\
    & \qquad \left. + \left\|p^{i-1}\right\|^2_{H^{k}\left(\Omega_f\right)} \right) .
\end{align*}
{\sunny{By taking the sum over $ i= 1, ..., n$ on both sides of the estimate above, we obtain the desired estimate \eqref{eq:F_phi-estimate}. This completes the proof. }}
\end{proof}

\subsection{Estimate of the mixed term $Z_{n}$}

\begin{lemma}[Estimate of $Z_{n}$]
    \label{lemma:Z-estimate}
    \begin{align}
        \label{eq:Z-sum-estimate}
        &\sum_{i=1}^n Z_i \leq \frac{\Delta t}2 \sum_{i=1}^n \left(\left\|\left(e_{\boldsymbol{u}}^{h, i}-e_{\boldsymbol{u}}^{h, i-1}\right) \cdot {\boldsymbol{n}}_f\right\|_{L, \Gamma}^2+\left\|e_\phi^{h, i}-e_\phi^{h, i-1}\right\|_{L^{-1}, \Gamma}^2\right) 
        \\
        \notag
        &  + C \varepsilon \Delta t \sum_{i=1}^n \left( \|\nabla e_{\phi}^{h,i-1}\|_{\Omega_p}^2 + \|\boldsymbol{D} (e_{\boldsymbol{u}}^{h,i-1})\|_{\Omega_f}^2 \right)
        + C_{\varepsilon} \Delta t \sum_{i=1}^n \left( \|e_{\phi}^{h,i-1}\|_{\Omega_p}^2 + \|e_{\boldsymbol{u}}^{h,i-1}\|_{\Omega_f}^2 \right) .
    \end{align}
\end{lemma}
\begin{proof}
    {\sunny{We estimate $Z_i$, defined in \eqref{eq:energy-compact}, as follows}}:
$$
\begin{aligned}
    Z_i&=2 \Delta t\left\langle e_\phi^{h, i-1},\left(e_{\boldsymbol{u}}^{h, i}-e_{\boldsymbol{u}}^{h, i-1}\right) \cdot {\boldsymbol{n}}_f\right\rangle_{\Gamma}+2 \Delta t\left\langle e_{\boldsymbol{u}}^{h, i-1} \cdot {\boldsymbol{n}}_p, e_\phi^{h, i}-e_\phi^{h, i-1}\right\rangle_{\Gamma}
    \\
    &\leq 2 \Delta t\left|\left\langle L^{-1 / 2} e_\phi^{h, i-1}, L^{1 / 2}\left(e_{\boldsymbol{u}}^{h, i}-e_{\boldsymbol{u}}^{h, i-1}\right) \cdot {\boldsymbol{n}}_f\right\rangle_{\Gamma}\right| 
    \\
    & \quad\quad +2 \Delta t\left| \left\langle L^{1 / 2}\left(e_{\boldsymbol{u}}^{h, i-1} \cdot {\boldsymbol{n}}_f\right), L^{-1 / 2}\left(e_\phi^{h, i}-e_\phi^{h, i-1}\right)\right\rangle_{\Gamma} \right|
    \\
    &\leq \frac{\Delta t}{2} \left(\left\|\left(e_{\boldsymbol{u}}^{h, i}-e_{\boldsymbol{u}}^{h, i-1}\right) \cdot {\boldsymbol{n}}_f\right\|_{L, \Gamma}^2+\left\|e_\phi^{h, i} -e_\phi^{h, i-1}\right\|_{L^{-1}, \Gamma}^2\right) 
    \\
    & \quad\quad +2\Delta t \left(\left\|e_\phi^{h, i-1}\right\|_{L^{-1}, \Gamma}^2+\left\|e_{\boldsymbol{u}}^{h, i-1} \cdot {\boldsymbol{n}}_f\right\|_{L, \Gamma}^2\right)
    \\
    &\leq \frac{\Delta t}{2} \left(\left\|\left(e_{\boldsymbol{u}}^{h, i}-e_{\boldsymbol{u}}^{h, i-1}\right) \cdot {\boldsymbol{n}}_f\right\|_{L, \Gamma}^2+\left\|e_\phi^{h, i}-e_\phi^{h, i-1}\right\|_{L^{-1}, \Gamma}^2\right) 
    \\
    &\quad + 2 C \Delta t \|e_{\phi}^{h,i-1}\|_{\Omega_p} \|\nabla e_{\phi}^{h,i-1}\|_{\Omega_p} + 2C \Delta t \|\boldsymbol{D} (e_{\boldsymbol{u}}^{h,i-1})\|_{\Omega_f} \|e_{\boldsymbol{u}}^{h,i-1}\|_{\Omega_f}
    \end{aligned}
    $$
    $$
    \begin{aligned}
    &\le  \frac{\Delta t}{2} \left(\left\|\left(e_{\boldsymbol{u}}^{h, i}-e_u^{h, i-1}\right) \cdot {\boldsymbol{n}}_f\right\|_{L, \Gamma}^2+\left\|e_\phi^{h, i}-e_\phi^{h, i-1}\right\|_{L^{-1}, \Gamma}^2\right) 
    \\
    &\quad + C \varepsilon \Delta t \left( \|\nabla e_{\phi}^{h,i-1}\|_{\Omega_p}^2 + \|\boldsymbol{D} (e_{\boldsymbol{u}}^{h,i-1})\|_{\Omega_f}^2 \right)
    \\
    &\quad + C_{\varepsilon} \Delta t \left( \|e_{\phi}^{h,i-1}\|_{\Omega_p}^2 + \|e_{\boldsymbol{u}}^{h,i-1}\|_{\Omega_f}^2 \right) .
\end{aligned}
$$

The conclusion follows by summing this inequality.
\end{proof}

\subsection{The Final {\emph{a priori}} Error Estimate}
Let us recall a discrete Gronwall inequality. 
\begin{lemma}\cite[Lemma~28]{Layton:book}
    \label{lemma:discrete-Gronwall}
    Let $\Delta t$, $B$, $a_i$, $b_i$, $c_i$, $d_i$ for integers $i\ge 0$ be nonnegative numbers. Suppose that 
    \begin{align*}
    	a_n + \Delta t \sum_{\textcolor{black}{i}=0}^n b_i \le \Delta t \sum_{i=0}^{n-1} d_i a_i + \Delta t \sum_{i=0}^n c_i + B \quad \text{ for } n \ge 1. 
    \end{align*}
    Then, 
    \begin{align*}
    	a_n + \Delta t \sum_{i=0}^n b_i \le \exp \left(\Delta t \sum_{\textcolor{black}{i}=0}^{n-1} d_i \right) \left( \Delta t \sum_{i=0}^n c_i + B \right) \quad \text{ for } n \ge 1.
    \end{align*}
\end{lemma}

{\sunny{Our final {\emph{a priori}} error estimate is a direct consequence of this Gronwall inequality. More precisely, we have the following result. }}
\begin{theorem}[A priori error estimate]
\label{thm:apriori}
Let the assumptions of Sections~\ref{sec:continuous}--\ref{sec:projections}
hold and let the exact solution of \eqref{eq:stokes-biot}--\eqref{interface condition} satisfy the regularity {\sunny{assumptions}}
\eqref{eq:regularity}.
Let
$(\boldsymbol{u}_h^n,p_h^n,
  \boldsymbol{\eta}_h^n,\boldsymbol{\xi}_h^n,\varphi_h^n)$
be the solution of the fully discrete explicit splitting scheme
\eqref{eq:fluid-discrete}--\eqref{eq:poro-discrete},
with initial data chosen as the projections of the exact solution.
Then there exists a constant $C>0$ independent of $h$ and $\Delta t$ such that:
\[
\max_{0\le n\le N} X_n
+ \Big( \sum_{n=1}^N Y_n^2 \Big)^{1/2}
\le
C \big( h^{k+\frac r2} + \Delta t \big),
\]
{\sunny{where $X_n$ and $Y_n$ are the error terms defined in \eqref{eq:Xn-definition} and \eqref{eq:Yn-definition}, respectively.}} 
{\sunny{The constant $C$ here depends on 
$C_\varepsilon$, $T$, the norms of the exact solution, and on the \textcolor{black}{model parameters}.}}  
In particular, the method is first-order accurate in time
and achieves spatial order $k+\frac r2$ for degree-$k$ elements.
\end{theorem}
\begin{proof}
{\sunny{We start by recalling the energy error equality \eqref{eq:energy-compact-sum} and use the estimates of the residual terms on the right hand-side of \eqref{eq:energy-compact-sum}, provided in Theorem~\ref{thm:residual-estimates}, along with the estimate of the mixed error terms, provided
in Lemma~\ref{lemma:Z-estimate}, to obtain that for sufficiently small $\varepsilon > 0$, the following estimate holds:}}
\textcolor{black}{
\begin{align*}
        &{\small{X_n^2 + \frac 14 \sum_{i=1}^n Y_i^2 \le X_0^2 + C \varepsilon \Delta t \sum_{i=1}^{n-1} X_i^2 
 + C_{\varepsilon} \Delta t h^{2(k+r)} \sum_{i=1}^n\left(\left\|\partial_t \boldsymbol{u}^i\right\|_{H^{k+1}\left(\Omega_f\right)}^2+\left\|\partial_t p^i\right\|_{H^k\left(\Omega_f\right)}^2\right)}}  \\
& +C_{\varepsilon}(\Delta t)^2\left(\left\|\partial_{t t} \boldsymbol{u}\right\|_{L^2\left(0, t_n ; L^2\left(\Omega_f\right)\right)}^2+\left\|\partial_t \boldsymbol{u}\right\|_{L^2\left(0, t_n ; H^1\left(\Omega_f\right)\right)}^2+\left\|\partial_t \boldsymbol{\xi}\right\|_{L^2\left(0, t_n ; H^1\left(\Omega_p\right)\right)}^2
\right.
\\
& \qquad  
 +\left\|\partial_t \phi\right\|_{L^2\left(0, t_n ; H^1\left(\Omega_p\right)\right)}^2
+ \left\|\partial_{t t} \phi\right\|_{L^2\left(0, t_n ; L^2\left(\Omega_p\right)\right)}^2+\left\|\partial_t \phi\right\|_{L^2\left(0, t_n ; H^1\left(\Omega_p\right)\right)}^2
\\
 &\qquad + \left\|\partial_t \boldsymbol{u}\right\|_{L^2\left(0, t_n ; H^1\left(\Omega_f\right)\right)}^2 
+ \left\|\partial_t^2 \boldsymbol{u}\right\|_{L^2\left(t_1, t_n ; L^2(\Gamma)\right)}^2+\left\|\partial_t^3 \boldsymbol{\eta}\right\|_{L^2\left(t_1, t_n ; L^2(\Gamma)\right)}^2 \\
& \qquad
+\left\|\partial_t^3 \boldsymbol{\eta}\right\|_{L^2\left(0, t_n ; L^2\left(\Omega_p\right)\right)}^2
+\left\|\partial_t \boldsymbol{u}\right\|_{L^{\infty}\left(t_{n-1}, t_n ; L^2(\Gamma)\right)}^2+\left\|\partial_t^2 \boldsymbol{\eta}\right\|_{L^{\infty}\left(t_{n-1}, t_n ; L^2(\Gamma)\right)}^2
\\
&\qquad \left. +\left\|\partial_t^3 \boldsymbol{\eta}\right\|_{L^{\infty}\left(t_{n-1}, t_n ; L^2(\Gamma)\right)}^2\right) \\
& +C_{\varepsilon} \Delta t h^{2 k+r} \sum_{j=0}^{n-1}\left(\left\|\boldsymbol{\xi}^j\right\|_{H^{k+1}\left(\Omega_p\right)}^2+\left\|\phi^j\right\|_{H^{k+1}\left(\Omega_p\right)}^2+\left\|\boldsymbol{u}^j\right\|_{H^{k+1}\left(\Omega_f\right)}^2+\left\|p^j\right\|_{H^k\left(\Omega_f\right)}^2\right. \\
& \qquad + \left\|\boldsymbol{\xi}^i\right\|_{H^{k+1}\left(\Omega_p\right)}^2+\left\|\phi^{i-1}\right\|_{H^{k+1}\left(\Omega_p\right)}^2+\left\|\partial_t \phi^i\right\|_{H^{k+1}\left(\Omega_p\right)}^2+\left\|\boldsymbol{u}^{i-1}\right\|_{H^{k+1}\left(\Omega_f\right)}^2
\\
&\qquad \left. +\left\|p^{i-1}\right\|_{H^k\left(\Omega_f\right)}^2\right) \\
& +C_{\varepsilon} h^{2 k+r}\left(\left\|\partial_t \boldsymbol{u}\right\|_{L^2\left(0, t_{n-1} ; H^{k+1}\left(\Omega_f\right)\right)}^2+\left\|\boldsymbol{u}^{n-1}\right\|_{H^{k+1}\left(\Omega_f\right)}^2+\left\|\boldsymbol{u}^0\right\|_{H^{k+1}\left(\Omega_f\right)}^2\right. \\
& \qquad +\left\|\partial_t \boldsymbol{\eta}\right\|_{L^2\left(0, t_n ; H^{k+1}\left(\Omega_p\right)\right)}^2+\left\|\partial_t^2 \boldsymbol{\eta}\right\|_{L^2\left(0, t_n ; H^{k+1}\left(\Omega_p\right)\right)}^2+\left\|\partial_t \phi\right\|_{L^2\left(0, t_n ; H^{k+1}\left(\Omega_p\right)\right)}^2
\\
& \qquad+\left\|\boldsymbol{\xi}^{n-1}\right\|_{H^{k+1}\left(\Omega_p\right)}^2+\left\|\boldsymbol{\xi}^0\right\|_{H^{k+1}\left(\Omega_p\right)}^2+\left\|\phi^1\right\|_{H^{k+1}\left(\Omega_p\right)}^2+\left\|\phi^n\right\|_{H^{k+1}\left(\Omega_p\right)}^2 
\\
& \qquad + \left. \left\|\partial_t \boldsymbol{\eta}\right\|_{L^{\infty}\left(0, t_1 ; H^{k+1}\left(\Omega_p\right)\right)}^2+\left\|\partial_t \boldsymbol{\eta}\right\|_{L^{\infty}\left(t_{n-1}, t_n ; H^{k+1}\left(\Omega_p\right)\right)}^2\right).
\end{align*}
}
The conclusion {\sunny{of the theorem}} follows from {\sunny{discrete Gronwall's inequality, stated in}} Lemma~\ref{lemma:discrete-Gronwall}, by taking $a_n, b_i, c_i, d_i$ and $B$ equal to:

\begin{align*}
a_n &= X_n^2, \quad 
b_i=  \frac{1}{4\Delta t} Y_i^2, \quad
d_i = C_{\varepsilon},
\\
\\
c_i &=  C_{\varepsilon} h^{2(k+r)}\left(\left\|\partial_t \boldsymbol{u}^i\right\|_{H^{k+1}\left(\Omega_f\right)}^2+\left\|\partial_t p^i\right\|_{H^k\left(\Omega_f\right)}^2\right) \\
& +C_{\varepsilon} h^{2 k+r}\left(\left\|\boldsymbol{\xi}^{i-1}\right\|_{H^{k+1}\left(\Omega_p\right)}^2+\left\|\phi^{i-1}\right\|_{H^{k+1}\left(\Omega_p\right)}^2+\left\|\boldsymbol{u}^{i-1}\right\|_{H^{k+1}\left(\Omega_f\right)}^2+\left\|p^{i-1}\right\|_{H^k\left(\Omega_f\right)}^2\right) \\
& +C_{\varepsilon} h^{2 k+r}\left(\left\|\boldsymbol{\xi}^i\right\|_{H^{k+1}\left(\Omega_p\right)}^2+\left\|\phi^{i-1}\right\|_{H^{k+1}\left(\Omega_p\right)}^2+\left\|\partial_t \phi^i\right\|_{H^{k+1}\left(\Omega_p\right)}^2+\left\|\boldsymbol{u}^{i-1}\right\|_{H^{k+1}\left(\Omega_f\right)}^2 \right.
\\
& \qquad\qquad  + \left. \left\|p^{i-1}\right\|_{H^k\left(\Omega_f\right)}^2\right),
\\
\\
B &= {\small{ X_0^2 
 +C_{\varepsilon}(\Delta t)^2\left(\left\|\partial_{t t} \boldsymbol{u}\right\|_{L^2\left(0, t_n ; L^2\left(\Omega_f\right)\right)}^2
 +\left\|\partial_t \boldsymbol{u}\right\|_{L^2\left(0, t_n ; H^1\left(\Omega_f\right)\right)}^2
 +\left\|\partial_t \boldsymbol{\xi}\right\|_{L^2\left(0, t_n ; H^1\left(\Omega_p\right)\right)}^2 
 \right.}}
 \\
 &\qquad \left.
 + \left\|\partial_t \phi\right\|_{L^2\left(0, t_n ; H^1\left(\Omega_p\right)\right)}^2
+ \left\|\partial_{t t} \phi\right\|_{L^2\left(0, t_n ; L^2\left(\Omega_p\right)\right)}^2+\left\|\partial_t \phi\right\|_{L^2\left(0, t_n ; H^1\left(\Omega_p\right)\right)}^2
\right.
\\
& \qquad  +\left\|\partial_t \boldsymbol{u}\right\|_{L^2\left(0, t_n ; H^1\left(\Omega_f\right)\right)}^2 
+ \left\|\partial_t^2 \boldsymbol{u}, \partial_t^3 \boldsymbol{\eta}\right\|_{L^2\left(t_1, t_n ; L^2(\Gamma)\right)}^2+\left\|\partial_t^3 \boldsymbol{\eta}\right\|_{L^2\left(0, t_n ; L^2\left(\Omega_p\right)\right)}^2
\\
& \qquad + \left. \left\|\partial_t \boldsymbol{u}, \partial_t^2 \boldsymbol{\eta}, \partial_t^3 \boldsymbol{\eta}\right\|_{L^{\infty}\left(t_{n-1}, t_n ; L^2(\Gamma)\right)}^2\right) \\
& +C_{\varepsilon} h^{2 k+r}\left(\left\|\partial_t \boldsymbol{u}\right\|_{L^2\left(0, t_{n-1} ; H^{k+1}\left(\Omega_f\right)\right)}^2+\left\|\boldsymbol{u}^{n-1}\right\|_{H^{k+1}\left(\Omega_f\right)}^2+\left\|\boldsymbol{u}^0\right\|_{H^{k+1}\left(\Omega_f\right)}^2\right. \\
& \qquad+\left\|\partial_t \boldsymbol{\eta}, \partial_t^2 \boldsymbol{\eta}, \partial_t \phi\right\|_{L^2\left(0, t_n ; H^{k+1}\left(\Omega_p\right)\right)}^2+\left\|\boldsymbol{\xi}^{n-1}\right\|_{H^{k+1}\left(\Omega_p\right)}^2+\left\|\boldsymbol{\xi}^0\right\|_{H^{k+1}\left(\Omega_p\right)}^2
\\
& \qquad + \left\|\phi^1\right\|_{H^{k+1}\left(\Omega_p\right)}^2+\left\|\phi^n\right\|_{H^{k+1}\left(\Omega_p\right)}^2 
+\left\|\partial_t \boldsymbol{\eta}\right\|_{L^{\infty}\left(0, t_1 ; H^{k+1}\left(\Omega_p\right)\right)}^2
\\
& \qquad \left. +\left\|\partial_t \boldsymbol{\eta}\right\|_{L^{\infty}\left(t_{n-1}, t_n ; H^{k+1}\left(\Omega_p\right)\right)}^2\right).
\end{align*}
\end{proof}

\section{Numerical validation}
\label{sec:numerics}
In this example, we considered a benchmark problem with manufactured solutions to examine the rates of convergence in time and space of the explicit splitting scheme. We solve the time-dependent Stoke-Biot system with added external forcing terms, {\sunny{given by the following}}:
$$
\left\{
\begin{array}{ll}
    \rho_f \partial_t \boldsymbol{u}=\nabla \cdot \boldsymbol{\sigma}_f\left(\boldsymbol{u}, p\right)+\boldsymbol{F}_f & \text { in } \Omega_f \times(0, T), \\
    \nabla \cdot \boldsymbol{u}=g_f & \text { in } \Omega_f \times(0, T), \\
    \partial_t \boldsymbol{\eta}=\boldsymbol{\xi} & \text { in } \Omega_p \times(0, T), \\
    \rho_p \partial_t \boldsymbol{\xi}=\nabla \cdot \boldsymbol{\sigma}_p\left(\boldsymbol{\eta}, \phi\right)+\boldsymbol{F}_e & \text { in } \Omega_p \times(0, T), \\
    \boldsymbol{u_p}=-\mathbb K\nabla \phi & \text { in } \Omega_p \times(0, T), \\
    C_0 \partial_t \phi+\alpha \nabla \cdot \boldsymbol{\xi}-\nabla \cdot (\mathbb K\nabla\phi)=F_d & \text { in } \Omega_p \times(0, T) .
\end{array}
\right.
$$
The FPSI problem is defined within {\sunny{the rectangular domain $\Omega = (0,1) \times (-1,1)$, where the fluid domain occupies the upper half of $\Omega$, i.e., $\Omega_f = (0,1) \times (0,1)$, and the solid domain occupies the lower half of $\Omega$, i.e., $\Omega_p = (0,1) \times (-1,0)$. The exact solution of this problem is given by:}}
\begin{align*}
& \boldsymbol{u}_{exact}=\pi \cos (\pi t)\left[\begin{array}{c}
-3 x+\cos (y) \\
y+1
\end{array}\right], \quad&&p_{exact}=e^t \sin (\pi x) \cos \left(\frac{\pi y}{2}\right)+2 \pi \cos (\pi t), \\
& \boldsymbol{\eta}_{exact}=\sin (\pi t)\left[\begin{array}{c}
-3 x+\cos (y) \\
y+1
\end{array}\right], \quad&&\phi_{exact}=e^t \sin (\pi x) \cos \left(\frac{\pi y}{2}\right).
\end{align*}
From the exact solutions, we can retrieve the corresponding forcing terms of $\boldsymbol{F}_f, g_f, \boldsymbol{F}_e$, and $F_d$:
\begin{equation}
\begin{aligned}
&\boldsymbol{F}_f=\left[\begin{array}{l}
\begin{aligned}
& \rho_f \pi^2 \sin (\pi t)(3 x-\cos y)+\pi e^t \cos (\pi x) \cos \left(\frac{\pi y}{2}\right)+\mu \pi \cos (\pi t) \cos y \\
& -\rho_f \pi^2 \sin (\pi t)(y+1)-\frac{\pi}{2} e^t \sin (\pi x) \sin \left(\frac{\pi y}{2}\right) 
\end{aligned}
\end{array}\right],  \\
&g_f=-2 \pi \cos (\pi t),\\
&\boldsymbol{F}_e=\left[\begin{array}{l}
\begin{aligned}
& \rho_p \pi^2 s(3 x-\cos y)+\alpha \pi e^t \cos (\pi x) \cos \left(\frac{\pi y}{2}\right)+\mu_s s \cos y \\
& -\rho_p \pi^2 s(y+1)-\alpha \frac{\pi}{2} e^t \sin (\pi x) \sin \left(\frac{\pi y}{2}\right) 
\end{aligned} 
\end{array}\right],  \\
&{F}_d=C_0 e^t \sin (\pi x) \cos \left(\frac{\pi y}{2}\right)-2 \alpha \pi \cos (\pi t)+\frac{5}{4} \pi^2 e^t \sin (\pi x) \cos \left(\frac{\pi y}{2}\right).
\end{aligned}
\end{equation}
{\sunny{We set the physical parameters all equal to one:}}
$$
\rho_p=\mu_p=\lambda_p=\alpha=C_0=\gamma=\rho_f=\mu_f=1, \quad \mathbb{K}=\mathbf{I}.
$$
{\sunny{Finite elements are used for spatial discretization. In particular, for the  fluid we use  Taylor-Hood elements $\mathbb{P}_2-\mathbb{P}_1$ for $\left(\boldsymbol{u}, p\right)$.}} For the Biot variables, we employ continuous $\mathbb{P}_2$ elements for the solid displacement $\boldsymbol{\eta}$ and continuous $\mathbb{P}_1$ elements for the Darcy pressure $\phi$, yielding a mixed stable pair. The system is {\sunny{solved on the time interval $(0,T) = (0,1)$, and we evaluate the numerical error at $T = 1$}}. To compute convergence rates, we define the final time errors for {\sunny{structure displacement and velocity $(\boldsymbol{\eta},\boldsymbol{\xi})$, Darcy pressure $\phi$, and fluid velocity and pressure $\left(\boldsymbol{u}, p\right)$, as follows:}}
$$
\begin{aligned}
e_{\boldsymbol{\eta}}  & :=\left\|\boldsymbol{\eta}_h(T)-\boldsymbol{\eta}_{\text {exact }}(T)\right\|_{L^2\left(\Omega_s\right)}, \\
e_{\boldsymbol{\xi}}  & :=\left\|\boldsymbol{\xi}_h(T)-\boldsymbol{\xi}_{\text {exact }}(T)\right\|_{L^2\left(\Omega_s\right)} \\
e_\phi & :=\left\|\phi_h(T)-\phi_{\text {exact }}(T)\right\|_{L^2\left(\Omega_s\right)} \\
e_{\boldsymbol{u}} & :=\left\|\boldsymbol{u}_{f, h}(T)-\boldsymbol{u}_{\text {exact }}(T)\right\|_{L^2\left(\Omega_f\right)} \\
e_p & :=\left\|p_{f, h}(T)-p_{\text {exact}}(T)\right\|_{L^2\left(\Omega_f\right)} .
\end{aligned}
$$
To determine the temporal convergence rate, we select the following time and space discretization parameters for various values of $n$ ranging from 4 to 128:
$$
\left\{\Delta t, \Delta x\right\} = \left\{\frac{0.05}{n}, \frac{0.5}{n}\right\},
$$
where $\Delta x$ represents the mesh size. The computed error data and corresponding convergent rates are reported in Table \ref{convergence1}. The result indicates that our partitioned method achieves first-order accuracy in time for both the Stokes and Biot variables without requiring {\sunny{any subiterations}}.

\begin{table}[h!]
    \caption{Temporal convergence at $T=1$. All variables exhibit first-order accuracy}
   \centering
    \begin{tabular}{c|c|c|c|c|c}
    \hline
        $n$ & $e_{\boldsymbol{\eta}}$& $e_{\boldsymbol{\xi}}$& $e_{\phi}$& $e_{\boldsymbol{u}}$& $e_{p}$\\
        \hline
       8 & 8.49E-02 & 6.36E-02 & 6.60E-03 & 7.24E-03 & 1.06E-01\\
      16 & 4.29E-02 & 3.21E-02 & 3.12E-03 & 3.67E-03 & 5.32E-02\\
      32 & 2.16E-02 & 1.61E-02 & 1.53E-03 & 1.85E-03 & 2.66E-02\\
      64 & 1.08E-02 & 8.08E-03 & 7.56E-04 & 9.29E-04 & 1.33E-02\\
 \hline
    \end{tabular}
    \label{convergence1}
\end{table}
Regarding the spatial convergence rate, a third-order accuracy in the $L^2$ norm is observed in space for velocity and displacement, and a second-order accuracy is observed for pressure, as illustrated in Table \ref{convergence3}. {\sunny{These are optimal spatial convergence rates consistent with the chosen approximation spaces.}}
\begin{table}[h!]
    \caption{Spatial errors at $T_0=10^{-4}$. We set $\Delta t=10^{-7}$ and refine with $h= 1 / n$ to assess pure spatial convergence.}
   \centering
    \begin{tabular}{c|c|c|c|c|c}
    \hline
        $n$ & $e_{\boldsymbol{\eta}}$& $e_{\boldsymbol{\xi}}$& $e_{\phi}$& $e_{\boldsymbol{u}}$& $e_{p}$\\
        \hline
       8 & 1.37E-03& 3.77E-03& 6.96E-03& 3.21E-03& 1.83E-02\\
      16 & 6.83E-04& 9.43E-04& 1.75E-03& 7.97E-04& 5.69E-03\\
      32 & 3.42E-04& 2.35E-04& 4.39E-04& 1.99E-04& 1.88E-03\\
      64 & 1.71E-04& 5.81E-05& 1.10E-04& 4.95E-05& 6.41E-04\\
      \hline
    \end{tabular}   
    \label{convergence3}
\end{table}

\section{Conclusion}
\label{sec:main6}
In this paper, we develop a rigorous \emph{a priori} error analysis for a fully discrete, parallelizable, explicit loosely coupled scheme for the Stokes--Biot problem \cite{24M1695713}. The key advantage of the method is that, at each time step, the fluid and poroelastic subproblems can be solved \emph{independently}, while \emph{stability} and \emph{provable convergence} are maintained through a consistent and stable enforcement of the interface conditions. In particular, the tangential Beaver--Joseph--Saffman slip condition is stabilized with a slip parameter $\gamma>0$, which plays a penalty-like role, and the dynamic coupling condition (relating normal stress and pressure) is stabilized through a penalty parameter $L$. A further important feature is the use of \emph{pore pressure} in the coupling conditions to replace stress terms, which improves the robustness of the explicit splitting without compromising consistency.

Our error analysis is based on a discrete energy framework. We introduce Ritz-type projections in each subdomain so that, upon subtracting the discrete scheme from the time-discrete continuous formulation, the dominant interpolation contributions cancel within the principal bilinear forms. This leads to reduced error equations in which the remaining consistency terms arise primarily from (i) time discretization residuals and (ii) lagged interface data introduced by the explicit splitting. 
The main result of this manuscript is the derivation of a discrete error energy identity, and establishment of unconditional error estimates in a combined energy-dissipation norm via a Gronwall--type argument. 
The estimate shows that the method is \emph{$1$st-order accurate in time} and achieves the expected \emph{spatial convergence rate} determined by the polynomial degree of the finite element spaces. To support the theory, we present numerical experiments based on a manufactured solution. The computations confirm \emph{$1$st-order temporal convergence} for all variables under the explicit loosely coupled scheme, and mesh-refinement studies demonstrate spatial convergence rates consistent with the chosen approximation spaces.


\section*{Acknowledgement}
{\sunny{
\v{C}ani\'{c}'s research has been supported in part by the
National Science Foundation under grants DMS-2408928, DMS-2247000 and by the  U.S. Department of Energy, Office of Science, Office of Advanced Scientific Computing Research's Applied Mathematics Competitive Portfolios program under Contract No. AC02-05CH11231.
Wang’s research has been supported in part by the National Science Foundation under grant DMS-2247001 and by Simons Foundation Travel Award.}}

\bibliographystyle{siamplain}

\bibliography{merged_references_clean}
\end{document}